\documentclass[12pt]{amsart}
\usepackage{amsmath, mathabx}
\usepackage{amsxtra}
\usepackage{amscd}
\usepackage{amsthm}
\usepackage{amsfonts}
\usepackage{amssymb}
\usepackage{mathrsfs}
\usepackage{mathbbol}

\usepackage {pstricks}
\usepackage {tikz}
\usepackage{pstricks,pst-node}
\usepackage[all]{xy}
\usepackage{mathbbol}

\usepackage{verbatim}

\newtheorem{theorem}{Theorem}[section]

\newtheorem{lemma}[theorem]{Lemma}

\newtheorem{corollary}[theorem]{Corollary}
\newtheorem{proposition}[theorem]{Proposition}

\theoremstyle{definition}
\newtheorem{definition}[theorem]{Definition}

\theoremstyle{remark}
\newtheorem{remark}[theorem]{Remark}

\numberwithin{equation}{section}
\numberwithin{figure}{section}

\newcommand{\mc}{\mathcal}

\newcommand{\lr}{\longrightarrow}
\newcommand{\inv}{^{-1}}

\newcommand{\be}{\begin{equation}}
\newcommand{\ee}{\end{equation}}

\newcommand{\B}{{\mathbb B}}
\newcommand{\C}{{\mathbb C}}

\newcommand{\Z}{{\mathbb Z}}
\newcommand{\N}{{\mathbb N}}

\newcommand{\K}{{\mathbb K}}

\newcommand{\CA}{{\mathcal A}}
\newcommand{\CB}{{\mathcal B}}

\newcommand{\CF}{{\mathcal F}}

\newcommand{\CO}{{\mathcal O}}

\newcommand{\CN}{{\mathcal N}}
\newcommand{\CT}{{\mathcal T}}

\newcommand{\CJ}{\mc J}

\newcommand{\OB}{{\mathcal{OB}}}

\newcommand{\mf}{\mathfrak}

\newcommand{\fg}{{\mf g}}
\newcommand{\fh}{{\mf h}}
\newcommand{\fb}{{\mf b}}

\newcommand{\fsl}{{\mathfrak {sl}}}

\newcommand{\Ker}{{\rm{Ker}}}
\newcommand{\id}{{\rm{id}}}

\newcommand{\Str}{{\rm{Str}}}

\newcommand{\ve}{\varepsilon}

\newcommand{\U}{{\rm{U}}}
\newcommand{\End}{{\rm{End}}}
\newcommand{\Hom}{{\rm{Hom}}}

\newcommand{\TL}{{\rm{TL}}}

\newcommand{\GL}{{\rm{GL}}}

\newcommand{\Sym}{{\rm{Sym}}}

\newcommand{\wt}{\widetilde}

\newcommand{\sdim}{{\rm{sdim}}}

\newcommand{\Sp}{{\rm Sp}}
\newcommand{\Or}{{\rm O}}
\newcommand{\ot}{\otimes}
\newcommand{\OSp}{{\rm OSp}}

\newcommand{\SO}{{\rm SO}}

\newcommand{\TLB}{{\rm{TLB}}}

\newcommand{\bA}{{\mathbb{A}}}
\newcommand{\bU}{{\mathbb{U}}}

\begin{document}

\normalfont

\title[Diagram categories and invariant theory]{Diagram categories and invariant theory for classical groups and supergroups}

\author{G.I. Lehrer and R.B. Zhang}
\thanks{This research was supported by the Australian Research Council}
\address{School of Mathematics and Statistics,
University of Sydney, N.S.W. 2006, Australia}
\email{gustav.lehrer@sydney.edu.au, ruibin.zhang@sydney.edu.au}
\date{\today}

\begin{abstract}
We introduce the notion of a diagram category and discuss its application to the invariant theory of classical groups and supergroups, 
with some indications concerning extensions to quantum groups and quantum supergroups. Tensor functors from various 
diagram categories to categories of representations are introduced and their properties are investigated, leading to first and second fundamental theorems (FFT and SFT) of invariant theory for classical supergroups, which include the FFTs and SFTs of the classical groups as special cases. Application of diagrammatic methods enables the construction of a  presentation for endomorphism algebras for the orthogonal and symplectic groups, leading to the solution of a problem raised by the work of Brauer and Weyl.  
\end{abstract}
\maketitle

\tableofcontents

\section{Introduction} 

Schur-Weyl duality has its origin in the 1901 thesis of I. Schur (see \cite{Sch}). It is shown there that 
if $V=\C^n$, then the homomorphism $\mu_r:\C\Sym_r\lr\End_{\GL(V)}(V^{\ot r})$ which is defined by the action of 
the symmetric group $\Sym_r$ on $V^{\ot r}$ by place permutations is surjective for all $r\geq 1$, a fact which is known 
as the first fundamental theorem (FFT) of invariant theory for the pair $(\GL(V), V)$. In {\it op. cit.} Schur also proved
that if $r\leq n$ the kernel $\Ker(\mu_r)$ is  zero
while for $r\geq n+1$ the kernel generated by the alternating idempotent 
\[
a_{n+1}^-:=\sum_{w\in\Sym_{n+1}\subseteq\Sym_r}\ve(w)w\in\C\Sym_r,
\]
where $\Sym_{n+1}$ is thought of as a subgroup of $\Sym_r$ for $r\geq n+1$ in the usual way, as the group of permutations of the first $n+1$ symbols. 
This is the second fundamental theorem (SFT) for the pair 
$(\GL(V), V)$ in the present setting. 

These facts were put to use in a series of papers by H. Weyl, who used them to study the representation theory of $\GL(V)$ \cite{W} using 
Frobenius' theory for the symmetric groups.

It was Richard Brauer \cite{B} who first pointed out that the invariant theory of the 
orthogonal and symplectic groups $G=\Or(V)$ or $G=\Sp(V)$ over $\C$ could be studied with the aid of certain diagrams, now known as ``Brauer diagrams''.
He and his contemporaries, such as Hermann Weyl, understood that Brauer's diagrams could be used to define an abstract
algebra $B_r(V)$ (see, e.g., \cite[\S5]{B}), with a surjective homomorphism  $\nu_r: B_r(V)\lr\End_G(V^{\ot r})$, for all $r\geq 1$,
where $G$ is one of the above groups and $V\cong\C^n$ is its natural 
representation. In \cite{W}, the algebras $B_r(V)$ are referred to as ``enigmatic", presumably because for $r\geq n+1$ (resp. $r\geq \frac{n}{2}+1)$)
if $G=\Or(V)$ (resp. $G=\Sp(V)$), the algebras $B_r(V)$ are not semi-simple \cite{We} (see also \cite{RS}). Furthermore, the kernels $\Ker(\nu_r)$ have only recently 
been determined (in 2012), and proved to be generated by a single idempotent, which is described in terms of the Brauer diagrams \cite{LZ4, LZ5}.

The Brauer diagrams may best be thought of categorically \cite{LZ5} and this category has been shown to be connected to 
several different representation categories, such as categories of representations of Lie superalgebras or Lie supergroups. 
In these latter cases, we have versions of the FFT and SFT in which isomorphisms are asserted between two non-semi-simple algebras \cite{DLZ, LZ6, LZ7}.

Many other types of diagrams occur in descriptions of various representation categories. 
Examples include  the ``up-down algebras'' of Brundan and Stroppel (cf. \cite{BS}), the Temperley-Lieb algebras of various types, and more importantly, the tangle diagrams \cite{FY, RT,RT1}, which have played a fundamental role in constructing quantum invariants of knots \cite{RT,RT1, ZGB, Z93, Z95}.

A particular advantage  \cite{CZ} of categorical descriptions of algebras in terms of diagrams is that in many cases, 
the bipartite nature of the diagrams (they have a ``top''
and ``bottom'') leads to a cellular structure for the endomorphism algebras \cite{GL96, GL98}, which in turn makes possible the study of base
change in a systematic way. Thus characteristic $p>0$ versions of the above theorems are available, and this remains a fruitful method
in the study of modular (i.e. non-semi-simple) representations, as in \cite{ALZ}, where there are still many open questions.

The application of diagrammatic methods in physics is discussed in \cite{PM}. 

In this work we shall give a general setting for such diagrammatic methods suitable for studying the invariant theory of classical groups and classical supergroups (cf. \cite{GW}), and explain the application of the methods in various cases. In addition, we shall touch on how the discussion below may be generalised to the case of modules for quantum groups \cite{D, Ji} and quantum supergroups 
\cite{BGZ, Z93, Z98}. 

We should mention that Karoubi envelopes (i.e., Cauchy completion) of the various diagram categories discussed here correspond to Deligne's categories \cite{Del}, 
which have been much studied in recent years (see, e.g., \cite{CW, Ck}), and applied, in particular, to study connections between the representation theory of Brauer algebras 
and parabolic category $\CO$ of $D$ type Lie algebras \cite{ES}. 

We intend to discuss diagrammatic methods and results in the invariant theory of quantum groups and quantum supergroups \cite{LZ1,LZZ1, LZZ2,LZ5} in the future, including connections 
with quantum topology \cite{Jo, RT,  T1, Wi, ZGB, Z93, Z95}.   




\part{Diagram categories}

\section{The Brauer category}

We consider first the Brauer category introduced in \cite{LZ5}. 
It is a ``categorified''  generalisation of the Brauer algebras which Richard Brauer introduced \cite{B} when studying the invariant theory of the orthogonal and symplectic groups.

\subsection{The Brauer category}\label{subsect:tangles}

Let $\N=\{0, 1, 2, \dots\}$ throughout the paper.

\begin{definition}\label{def:brauer-diag}
For any pair $k, \ell \in\N$,
a Brauer diagram from $k$ to $\ell$, or simply a $(k, \ell)$ Brauer diagram,  
is a partitioning of the set $\{1,2,\dots,k+\ell\}$ as a disjoint union of pairs.
\end{definition}
This is thought of as a diagram where $k+\ell$ points (the nodes, or vertices) are placed on
two parallel horizontal lines,
$k$ on the lower line and $\ell$ on the upper, with
arcs drawn to join points which are paired. We shall speak of the lower and upper nodes or vertices
of a diagram. The pairs will be known as {\em arcs} or {\em strings}. If $k=\ell=0$, there is by convention just one
Brauer $(0,0)$-diagram.

Figure \ref{fig1} below is a $(6, 4)$ Brauer diagram.

\begin{figure}[h]
\begin{center}
\begin{picture}(160, 60)(0,0)

\qbezier(0, 60)(30, 0)(60, 60)
\qbezier(0, 0)(60, 60)(90, 0)

\qbezier(30, 60)(90, 40)(150, 0)
\qbezier(90, 60)(60, 30)(30, 0)

\qbezier(60, 0)(90, 40)(120, 0)
\end{picture}
\end{center}
\caption{ }
\label{fig1}
\end{figure}

There are two operations on Brauer diagrams: {\em composition} defined using
concatenation of diagrams and {\em tensor product} defined using juxtaposition
(see below).

\begin{definition} \label{def:bkl} Let $\K$ be a commutative ring with identity,
and fix $\delta\in \K$.
Denote by ${B}_k^\ell(\delta)$ the free
$\K$-module with a basis consisting of
$(k, \ell)$ Brauer diagrams. Note that
${B}_k^\ell(\delta)\neq 0$ if and only if $k+\ell$ is even,
since the free $\K$-module with basis the empty set is zero.
By convention there is one diagram in $B_0^0(\delta)$, viz. the empty diagram.
Thus $B_0^0(\delta)=\K$.
\end{definition}

There are two $\K$-bilinear operations on diagrams.
\begin{eqnarray}\label{eq:products}
\begin{aligned}
&&\text{composition} & \quad \circ:  & & B_\ell^p(\delta)
\times B_k^\ell(\delta)\longrightarrow B_k^p(\delta), \quad and  \\
&&\text{tensor product} & \quad  \otimes: & & B_p^q(\delta)
\times B_k^\ell(\delta)\longrightarrow B_{k+p}^{q+\ell}(\delta)
\end{aligned}
\end{eqnarray}

These operations are defined as follows.
\begin{enumerate}
\item The composite $D_1\circ D_2$ of the Brauer diagrams $D_1\in B_\ell^p(\delta)$ and
$D_2\in B_k^\ell(\delta)$ is defined as follows.
First, the concatenation $D_1\#D_2$ is obtained by placing $D_1$ above $D_2$,
and identifying the $\ell$ lower nodes of $D_1$ with the corresponding upper nodes of $D_2$.
Then $D_1\#D_2$ is the union of a Brauer $(k,p)$ diagram $D$ with a certain number, $f(D_1,D_2)$ say,
of free loops. The composite $D_1\circ D_2$ is the element $\delta^{f(D_1,D_2)}D\in B_k^p(\delta)$.

\item The tensor product $D\otimes D'$ of
any two Brauer diagrams $D\in B_p^q(\delta)$
and $D'\in B_k^l(\delta)$ is the $(p+k, q+l)$
diagram obtained by juxtaposition, that is,
placing $D'$ on the right of $D$ without overlapping.
\end{enumerate}
Both operations are clearly associative.

\begin{definition}[\cite{LZ5}]
The {\em Brauer category with parameter $\delta$}, denoted by $\CB(\delta)$,
is the following $\K$-linear small category equipped with a bi-functor $\otimes$
(which will be called the tensor product):
\begin{enumerate}
\item the set of objects is $\N=\{0, 1, 2, \dots\}$,  and for any pair of objects $k, l$,
$\Hom_{\CB(\delta)}(k, l)$ is the $\K$-module $B_k^l(\delta)$; the composition of morphisms is
given by the composition of Brauer diagrams defined by \eqref{eq:products};
\item the tensor product $k\otimes l$ of objects  $k, l$ is  $k+l$ in $\N$, and the
tensor product of morphisms is given by the tensor product (juxtaposition) of Brauer diagrams of \eqref{eq:products}.
\end{enumerate}
\end{definition}
It follows from the associativity of composition of Brauer diagrams that
$\CB(\delta)$ is indeed a pre-additive category.

\subsection{Involutions}\label{ssec:involutions}
The category $\CB(\delta)$ has a {\em duality functor} $^*:\CB(\delta)\to \CB(\delta)^{\text{op}}$,
which takes each object to itself, and takes each diagram to its reflection in a horizontal line.
More formally, for any $(k,\ell)$ diagram $D$, $D^*$ is the $(\ell,k)$ diagram with precisely the same
pairs identified as $D$. Further, there is an involution $^\sharp:\CB(\delta)\to \CB(\delta)$
which also takes objects to themselves, but takes a diagram $D$ to its reflection in a vertical line.
Formally, if the upper nodes of the diagram $D$ are labelled $1,2.\dots,\ell$ and the lower
nodes are labelled $1',2',\dots,k'$, we apply the permutation $i\mapsto \ell+1-i,j'\mapsto k+1-j'$
to the nodes to get the arcs of $D^\sharp$. We shall meet the contravariant functor
$D\mapsto *D:=D^{*\circ\sharp}$ later.

It is easily checked that $(D_1\circ D_2)^*=D_2^*\circ D_1^*$, $(D_1\ot D_2)^*=D_1^*\ot D_2^*$,
and that
$(D_1\circ D_2)^\sharp=D_1^\sharp\circ D_2^\sharp$ and $(D_1\ot D_2)^\sharp=D_2^\sharp\ot D_1^\sharp$.

\subsection{Generators and relations} The exposition in this section and in Appendix \ref{sect:proof-presentation} is based on \cite{LZ5}.
The next theorem describes the Brauer diagrams in terms of generators and relations.  There is a corresponding description for tangle diagrams in
\cite{FY, T1, RT, T2}. 

\begin{theorem}[\cite{LZ5}]\label{thm:presentation}
\begin{enumerate}
\item\label{generators}
The four Brauer diagrams
\begin{center}
\begin{picture}(205, 40)(-5,0)
\put(0, 0){\line(0, 1){40}}
\put(5, 0){,}

\put(40, 0){\line(1, 2){20}}
\put(60, 0){\line(-1, 2){20}}
\put(65, 0){,}

\qbezier(100, 0)(115, 60)(130, 0)
\put(135, 0){,}

\qbezier(170, 30)(185, -30)(200, 30)
\put(200, 0){,}
\end{picture}
\end{center}
generate all Brauer diagrams by composition and tensor product (i.e., juxtaposition).
We shall refer to these generators as the elementary Brauer diagrams,
and denote them by $I$, $X$, $A$ and $U$
respectively. Note that
these diagrams are all fixed by $^\sharp$, and that $^*$ fixes $I$ and $X$, while
$A^*=U$ and $U^*=A$.
\item \label{relations} A complete set of relations among
these four generators is given by the following, and their transforms under $^*$ and $^\sharp$. This means that any
equation relating two words in these four generators can be deduced from the given relations.
\begin{eqnarray}
&I\circ I= I,\:(I\ot I)\circ X= X, \;(I\ot I)\circ A=A,\;(I\ot I)\circ U=U, \label{eq:identity}\\
&X\circ X= I,\label{eq:XX} \\
&(X\ot I)\circ(I\ot X)\circ(X\ot I) = (I\ot X)\circ(X\ot I)\circ(I\ot X),
\label{eq:braid}\\
&A\circ X = A,\label{eq:AX}\\
&A\circ U = \delta, \label{eq:AU}\\
&(A\ot I)\circ(I\ot X)=(I\ot A)\circ(X\ot I)\label{eq:slide}\\
&(A\ot I)\circ(I\ot U)=I.\label{eq:straight}
\end{eqnarray}
\end{enumerate}
The relations \eqref{eq:XX}-\eqref{eq:straight} are depicted diagrammatically in Figures \ref{fig:braid},
\ref{fig:loop} and \ref{fig:slide}.
\end{theorem}

\begin{figure}[h]
\begin{center}
\begin{picture}(100, 70)(90,-10)
\qbezier(0, 60)(50, 30)(0, 0)
\qbezier(30, 60)(-10, 30)(30, 0)

\put(40, 30){$=$}

\qbezier(60, 60)(60, 30)(60, 0)
\qbezier(90, 60)(90, 30)(90, 0)
\put(95, 0){;}

\put(10, -20){Double crossing}

\qbezier(150, 60)(130, 30)(150, 0)
\qbezier(120, 60)(170, 40)(180, 0)
\qbezier(120, 0)(170, 20)(180, 60)

\put(190, 30){$=$}

\qbezier(240, 60)(260, 30)(240, 0)
\qbezier(270, 0)(220, 20)(210, 60)
\qbezier(270, 60)(220, 40)(210, 0)

\put(180, -20){Braid relation}
\end{picture}
\end{center}
\caption{Relations \eqref{eq:XX} and \eqref{eq:braid}}
\label{fig:braid}
\end{figure}


\begin{figure}[h]
\begin{center}
\begin{picture}(100, 80)(40,-10)
\qbezier(0, 40)(15, 90)(30, 40)
\qbezier(0, 40)(3, 15)(30, 0)
\qbezier(30, 40)(27, 15)(0, 0)
\put(40, 30){$=$}
\qbezier(60, 10)(75, 90)(90, 10)
\put(95, 0){;}

\put(30, -20){De-looping}
\qbezier(120, 33)(135, 90)(150, 33)
\qbezier(120, 33)(135, -20)(150, 33)
\put(160, 30){$=\delta$}

\put(115, -20){Loop Removal}

\end{picture}
\end{center}
\caption{Relations \eqref{eq:AX} and \eqref{eq:AU}}
\label{fig:loop}
\end{figure}

\begin{figure}[h]
\begin{center}
\begin{picture}(100, 80)(50,-10)
\qbezier(0, 0)(5, 120)(40,0)
\qbezier(20, 0)(30, 30)(40,60)
\put(45, 30){$=$}
\qbezier(60, 0)(95, 120)(100,0)
\qbezier(60, 60)(70, 30)(80,0)

\put(105, 0){;}

\put(30, -20){Sliding}
\qbezier(130, 0)(135, 90)(150, 30)
\qbezier(150, 30)(165, -30)(170, 60)
\put(180, 30){$=$}
\qbezier(200, 0)(200, 30)(200, 60)

\put(130, -20){Straightening}

\end{picture}
\end{center}
\caption{Relations \eqref{eq:slide} and \eqref{eq:straight}}
\label{fig:slide}
\end{figure}
\begin{proof}
We will provide a purely algebraic proof
of the theorem in  Appendix \ref{sect:proof-presentation}. 
\end{proof}

\begin{remark}\label{rem:quotient-cat}
The operations in $\CB(\delta)$ mirror the operations in the tangle category
considered in \cite{FY, T1, RT, T2} and
the Brauer diagrams is a quotient category (in the sense of \cite[\S II.8]{MacL})
of the category of tangles. 
\end{remark}

\subsection{The Brauer algebra}\label{sect:Brauer-algebra}

We can recover  Brauer's algebra from  $\CB(\delta)$.

For any object $r$ in $\CB(\delta)$,  the $\K$-module $B_r^r(\delta)$ of morphisms forms
a unital associative $\K$-algebra under composition of Brauer diagrams. This is the Brauer algebra \cite[\S 5]{B} of 
degree $r$ with parameter $\delta$, which we will denote by $B_r(\delta)$.
The first two results of the following lemma are well known.

\begin{lemma}\label{lem:brprops}
\begin{enumerate}
\item For $i=1,\dots,r-1$, let $s_i$ and $e_i$ respectively be the $(r, r)$ Brauer diagrams shown in
Figure \ref{s-e} below.
\begin{figure}[ht]
\begin{center}
\begin{picture}(350, 60)(0,0)
\put(0, 0){\line(0, 1){60}}
\put(40, 0){\line(0, 1){60}}
\put(18, 30){...}
\put(10, 0){$i-1$}

\qbezier(60, 0)(70, 30)(80, 60)
\qbezier(60, 60)(70, 30)(80, 0)

\put(100, 0){\line(0, 1){60}}
\put(140, 0){\line(0, 1){60}}
\put(118, 30){...}
\put(150, 0){, }

\put(200, 0){\line(0, 1){60}}
\put(240, 0){\line(0, 1){60}}
\put(218, 30){...}
\put(210, 0){$i-1$}

\qbezier(260, 0)(270, 45)(280, 0)
\qbezier(260, 60)(270, 15)(280, 60)

\put(300, 0){\line(0, 1){60}}
\put(340, 0){\line(0, 1){60}}
\put(318, 30){...}
\put(350, 0){. }
\end{picture}
\end{center}
\caption{ }
\label{s-e}
\end{figure}
Then $B_r(\delta)$ has the following presentation as $\K$-algebra with anti-involution $*$.
The generators are $\{s_i, e_i \mid i=1, 2, \dots, r-1\}$, with relations
\[
\begin{aligned}
&s_is_j=s_js_i,\;s_ie_j=e_js_i,\;e_ie_j=e_je_i, \quad\text{if $|i-j|\geq 2$},\\
&s_i^2=1,\; s_is_{i+1}s_i=s_{i+1}s_is_{i+1},\\
&s_ie_i=e_is_i=e_i,\\
&e_i^2=\delta e_i,\\
&e_ie_{i\pm 1}e_i=e_i,\\
&s_ie_{i+1}e_i=s_{i+1}e_i,
\end{aligned}
\]
where the last five relations are valid for all applicable $i$.

\item The elements $s_1,\dots,s_{r-1}$ generate a subalgebra of $B_r(\delta)$, isomorphic to
the group algebra $\K\Sym_r$ of the symmetric group $\Sym_r$.

\item  The map $\ast$ of Lemma \ref{anti-invol} restricts to an anti-involution of the Brauer
algebra.
\end{enumerate}
\end{lemma}
Parts (1) and (2) follow from Theorem \ref{thm:presentation}.
Part (3) is easy to prove. However we note that
$*s_i= s_{r+1-i}$ and $*e_i= e_{r+1-i}$. This is different from
the standard cellular anti-involution $^*$ of the Brauer algebra.

We remark that multiplying the last relation above by $e_i$ on the left and using
two of the earlier relations, we obtain
\[
e_is_{i+ 1}e_i=e_i,
\]
a relation which we shall often use, together with its transform under $*$:
$e_is_{i-1}e_i=e_i$.

%
%
%
%
%
\subsection{Some useful diagrams}
We shall find the following diagrams useful in later sections of this work.
Let $A_q=A\circ (I\ot A\ot I)\dots (I^{\ot (q-1)}\ot A\ot I^{\ot (q-1)})$,
$U_q=(I^{\ot (q-1)}\ot U\ot I^{\ot (q-1)})\circ\dots\circ (I\ot U\ot I)\circ U$
and $I_q=I^{\ot q}$. These are depicted as diagrams in Figure \ref{Omega-U},
\begin{figure}[ht]
\begin{center}
\begin{picture}(340, 40)(-5,0)
\put(-5, 10){$A_q=$}
\qbezier(25, 0)(55, 60)(85, 0)
\put(33, 10){...}
\put(30, 0){\tiny$q$}
\qbezier(40, 0)(55, 40)(70, 0)
\put(90, 0){,}

\put(140, 10){$U_q=$}
\qbezier(165, 30)(195, -35)(225, 30)
\put(172, 25){\tiny$q$}
\put(172, 20){...}
\qbezier(180, 30)(195, -15)(210, 30)
\put(215, 0){,}

\put(270, 10){$I_q=$}
\put(300, 0){\line(0, 1){40}}
\put(310, 20){...}
\put(314, 0){\tiny$q$}
\put(330, 0){\line(0, 1){40}}
\put(335, 0){.}
\end{picture}
\end{center}
\caption{ }
\label{Omega-U}
\end{figure}

The following results are easy to prove diagrammatically.

\begin{lemma} \label{isomorphism-pri}
\begin{enumerate}
\item For any Brauer diagrams $D_1\in B_k^r(\delta)$ and
$D_2\in B_r^q(\delta)$, we have $I_r\circ D_1 =D_1$ and $D_2\circ I_r =D_2$.
That is, $I_r=\id_r$ for any object $r$ of $\CB(\delta)$.
\item  The following relation holds.
\[(I_q\otimes A_q)\circ(U_q\otimes I_q)=(U_q\otimes I_q)\circ(I_q\otimes A_q)=I_q.\]
\end{enumerate}
\end{lemma}

\begin{corollary}\label{isomorphism}
The following linear maps are inverses of each other for all $p, q$ and $r$ in $\N$.
\[
\begin{aligned}
{\mathbb U}_p^q=(-\otimes I_q)\circ(I_p\otimes U_q):
B_{p+q}^r(\delta)\longrightarrow B_p^{r+q}(\delta)\\
{\mathbb A}^r_q=(I_{r+q}\otimes A_q)\circ(- \otimes I_q):
B_p^{r+q}(\delta)\longrightarrow B_{p+q}^r(\delta).
\end{aligned}
\]
\end{corollary}

We also note that ${\mathbb U}_p^q=R^q$ and ${\mathbb A}^r_q=L^q$ in the notation of Definition \ref{def:raising}, and 
$L^q$ and $R^q$ are mutually inverse as shown in Lemma \ref{lem:red-raise}.

\begin{lemma}\label{anti-invol}
Let $\ast: B_p^q(\delta) \longrightarrow B_q^p(\delta)$ be the linear map defined for any
$D \in B_p^q(\delta)$ by
$
*D=(I_p\otimes A_q)\circ(I_p\otimes D\otimes I_q)\circ(U_p\otimes I_q).
$ 
Then $\ast$ coincides with the anti-involution
$D\mapsto *D:=D^{*\circ\sharp}$ discussed in \S\ref{ssec:involutions}.
\end{lemma}

Pictorially, $*D$ is obtained from $D$ as in Figure \ref{D-star}.
\begin{figure}[h]
\begin{center}
\begin{picture}(100, 60)(-20,0)

\put(-5, 20){\line(1, 0){35}}
\put(-5, 20){\line(0, 1){20}}
\put(30, 20){\line(0, 1){20}}
\put(-5, 40){\line(1, 0){35}}
\put(8, 25){D}

\qbezier(5, 40)(40, 95)(60, 0)
\qbezier(20, 40)(35, 70)(45, 0)
\put(45, 10){...}

\qbezier(5, 20)(-15, -10)(-20, 60)
\qbezier(20, 20)(-20, -35)(-35, 60)
\put(-30, 50){...}

\end{picture}
\end{center}
\caption{$*D$}
\label{D-star}
\end{figure}

The Brauer diagrams $X_{s, t}$ shown in Figure \ref{X} give rise to a braiding of $\CB(\delta)$.
\begin{figure}[h]
\begin{center}
\begin{picture}(50, 60)(0,0)

\qbezier(0, 60)(15, 30)(30, 0)
\qbezier(15, 60)(30, 30)(45, 0)

\qbezier(0, 0)(15, 30)(30, 60)
\qbezier(15, 0)(30, 30)(45, 60)

\put(5, 5){...}
\put(5, -5){\tiny$s$}

\put(30, 5){...}
\put(35, -5){\tiny$t$}
\end{picture}
\end{center}
\caption{Braiding}
\label{X}
\end{figure}

\noindent
Thus $\CB(\delta)$ has the structure of  a braided tensor category with all objects being self dual.

\begin{lemma} \label{lem:Sigma-1} \label{lem:red}
Let $\Sigma_\epsilon(r)=\sum_{\sigma\in \Sym_r} (-\epsilon)^{|\sigma|} \sigma\in B_r(\delta)$,
where $\epsilon= \pm 1$ and $|\sigma|$ is the length of $\sigma$.
Represent $\Sigma_\epsilon(r)$  pictorially by Figure \ref{Sigma-r}.

\begin{figure}[h]
\begin{center}
\begin{picture}(100, 60)(-5,0)
\put(20, 40){\line(0, 1){20}}
\put(35, 50){...}
\put(60, 40){\line(0, 1){20}}

\put(0, 20){\line(1, 0){80}}
\put(0, 20){\line(0, 1){20}}
\put(80, 20){\line(0, 1){20}}
\put(0, 40){\line(1, 0){80}}
\put(40, 28){$r$}

\put(20, 20){\line(0, -1){20}}
\put(35, 10){...}
\put(60, 20){\line(0, -1){20}}

\put(85, 0){.}
\end{picture}
\end{center}
\caption{}
\label{Sigma-r}
\end{figure}

Then the following relations hold for all $r$.
\begin{enumerate}
\item
\[
\begin{picture}(80, 60)(0, -30)
\put(0, 10){\line(1, 0){60}}
\put(0, -10){\line(1, 0){60}}
\put(0, 10){\line(0, -1){20}}
\put(60, 10){\line(0, -1){20}}
\put(25, -3){$r$}

\put(10, 10){\line(0, 1){15}}
\put(23, 15){$\cdots$}
\put(50, 10){\line(0, 1){15}}

\put(10, -10){\line(0, -1){15}}
\put(23, -20){$\cdots$}
\put(50, -10){\line(0, -1){15}}

\put(70, -2){$=$}
\end{picture}
\begin{picture}(80, 60)(-10, -30)
\put(0, 10){\line(1, 0){50}}
\put(0, -10){\line(1, 0){50}}
\put(0, 10){\line(0, -1){20}}
\put(50, 10){\line(0, -1){20}}
\put(15, -3){$r-1$}

\put(10, 10){\line(0, 1){15}}
\put(18, 15){$\cdots$}
\put(40, 10){\line(0, 1){15}}

\put(10, -10){\line(0, -1){15}}
\put(18, -20){$\cdots$}
\put(40, -10){\line(0, -1){15}}

\put(60, -25){\line(0, 1){50}}

\put(70, -3){$-$}

\end{picture}
\begin{picture}(80, 60)(-70, -30)

\put(-60, -3){$\epsilon (r-2)!^{-1}$}

\put(0, 10){\line(1, 0){40}}
\put(0, 25){\line(1, 0){40}}
\put(0, 10){\line(0, 1){15}}
\put(40, 10){\line(0, 1){15}}
\put(12, 15){\tiny$r-1$}

\put(5, 10){\line(0, -1){20}}
\put(25, 10){\line(0, -1){20}}
\put(8, -3){$\cdots$}

\qbezier(35,10)(45, 5)(50, -35)
\qbezier(35,-10)(45, -5)(50, 35)

\put(0, -10){\line(1, 0){40}}
\put(0, -25){\line(1, 0){40}}
\put(0, -10){\line(0, -1){15}}
\put(40, -10){\line(0, -1){15}}
\put(12, -20){\tiny$r-1$}

\put(5, 25){\line(0, 1){10}}
\put(35, 25){\line(0, 1){10}}
\put(14, 28){$\cdots$}

\put(5, -25){\line(0, -1){10}}
\put(35, -25){\line(0, -1){10}}
\put(14, -33){$\cdots$}

\put(55, -35){.}
\end{picture}
\]

\item
\[
\begin{picture}(80, 60)(0, -30)
\put(0, 10){\line(1, 0){60}}
\put(0, -10){\line(1, 0){60}}
\put(0, 10){\line(0, -1){20}}
\put(60, 10){\line(0, -1){20}}
\put(25, -3){$r$}

\put(10, 10){\line(0, 1){15}}
\put(20, 15){$\cdots$}
\put(40, 10){\line(0, 1){15}}

\qbezier(50, -10)(60, -40)(65, 0)
\qbezier(50, 10)(60, 40)(65, 0)

\put(10, -10){\line(0, -1){15}}
\put(20, -20){$\cdots$}
\put(40, -10){\line(0, -1){15}}

\put(75, -3){$=$}
\end{picture}
\begin{picture}(80, 60)(-90, -30)
\put(-80, -5){$-\epsilon(r-1-\epsilon\delta)$}
\put(0, 10){\line(1, 0){60}}
\put(0, -10){\line(1, 0){60}}
\put(0, 10){\line(0, -1){20}}
\put(60, 10){\line(0, -1){20}}
\put(20, -3){$r-1$}

\put(10, 10){\line(0, 1){15}}
\put(23, 15){$\cdots$}
\put(50, 10){\line(0, 1){15}}

\put(10, -10){\line(0, -1){15}}
\put(23, -20){$\cdots$}
\put(50, -10){\line(0, -1){15}}

\put(60, -25){.}
\end{picture}
\]

\item
\[
\begin{picture}(80, 60)(0, -30)
\put(0, 10){\line(1, 0){60}}
\put(0, -10){\line(1, 0){60}}
\put(0, 10){\line(0, -1){20}}
\put(60, 10){\line(0, -1){20}}
\put(25, -3){$r$}

\put(10, 10){\line(0, 1){15}}
\put(20, 15){$\cdots$}
\put(40, 10){\line(0, 1){15}}

\put(50, -10){\line(0, -1){15}}
\qbezier(50, 10)(65, 40)(70, -25)

\put(10, -10){\line(0, -1){15}}
\put(20, -20){$\cdots$}
\put(40, -10){\line(0, -1){15}}

\put(75, -3){$=$}
\end{picture}
\begin{picture}(80, 60)(-65, -30)
\put(-55, -5){$\sum_{i=0}^{r-1} (-\epsilon)^i$}
\put(0, 10){\line(1, 0){60}}
\put(0, -10){\line(1, 0){60}}
\put(0, 10){\line(0, -1){20}}
\put(60, 10){\line(0, -1){20}}
\put(20, -3){$r-1$}

\put(10, 10){\line(0, 1){15}}
\put(23, 15){$\cdots$}
\put(50, 10){\line(0, 1){15}}

\put(10, -10){\line(0, -1){15}}
\put(11, -20){\tiny{...}}
\put(20, -10){\line(0, -1){15}}
\put(40, -10){\line(0, -1){15}}
\put(41, -20){\tiny{...}}
\put(42, -28){\tiny$i$}
\put(50, -10){\line(0, -1){15}}

\qbezier(30, -25)(45, 0)(60, -25)

\put(65, -25){.}
\end{picture}
\]
\end{enumerate}
\end{lemma}
\begin{proof}
Part (1) generalises \cite[Lemma 5.1 (i)]{LZ4} and is a simple consequence of the double coset decomposition of
$\Sym_r$ as $\Sym_r=\Sym_{r-1}\amalg \Sym_{r-1}s_{r-1}\Sym_{r-1}$. Part (2) is immediate from
(1) and the statement (3) may be obtained from (1) by induction on $r$.
\end{proof}

\begin{lemma} \label{lem:Sigma}
Set $\epsilon=-1$. Then for all $k\ge 0$,
\begin{eqnarray}\label{eq:Sigma}
\begin{aligned}
\begin{picture}(80, 50)(50, -25)
\put(0, 10){\line(1, 0){60}}
\put(0, -10){\line(1, 0){60}}
\put(0, 10){\line(0, -1){20}}
\put(60, 10){\line(0, -1){20}}
\put(25, -3){$r$}

\put(5, 10){\line(0, 1){15}}
\put(15, 15){$\cdots$}
\put(35, 10){\line(0, 1){15}}
\qbezier(45, 10)(50, 35)(55, 10)

\put(5, -10){\line(0, -1){15}}
\put(7, -20){...}
\put(20, -10){\line(0, -1){15}}
\qbezier(45, -10)(50, -35)(55, -10)
\qbezier(25, -10)(30, -35)(35, -10)
\put(35, -20){...}
\put(38, -30){\tiny $k$}

\put(65, -3){$=$}
\end{picture}
\begin{picture}(80, 50)(-45, -25)
\put(-100, -5){$4k(r+\frac{\delta}{2}-k-1)$}
\put(-10, 10){\line(1, 0){60}}
\put(-10, -10){\line(1, 0){60}}
\put(-10, 10){\line(0, -1){20}}
\put(50, 10){\line(0, -1){20}}
\put(10, -3){$r-2$}

\put(-5, 10){\line(0, 1){15}}
\put(13, 15){$\cdots$}
\put(45, 10){\line(0, 1){15}}

\put(-5, -10){\line(0, -1){15}}
\put(-2, -20){...}
\put(10, -10){\line(0, -1){15}}
\qbezier(35, -10)(40, -35)(45, -10)
\qbezier(15, -10)(20, -35)(25, -10)
\put(25, -20){...}
\put(22, -30){\tiny $k-1$}

\put(55, -3){$+$}
\end{picture}
\begin{picture}(80, 60)(-105, -25)
\put(-75, -3){$(r-2-2k)!^{-1}$}

\put(0, 10){\line(1, 0){55}}
\put(0, 25){\line(1, 0){55}}
\put(0, 10){\line(0, 1){15}}
\put(55, 10){\line(0, 1){15}}
\put(18, 15){\tiny $r-2$}

\put(5, 10){\line(0, -1){20}}
\put(15, 10){\line(0, -1){20}}
\put(5, -3){...}

\qbezier(20,10)(25, -7)(30, 10)
\qbezier(40,10)(45, -7)(50, 10)
\put(30, 7){\tiny ...}
\put(33, -2){\tiny$k$}
\qbezier(20,-10)(25, 7)(30, -10)

\put(0, -10){\line(1, 0){40}}
\put(0, -25){\line(1, 0){40}}
\put(0, -10){\line(0, -1){15}}
\put(40, -10){\line(0, -1){15}}
\put(10, -20){\tiny $r-2k$}

\put(5, 25){\line(0, 1){10}}
\put(50, 25){\line(0, 1){10}}
\put(20, 28){$\cdots$}

\put(5, -25){\line(0, -1){10}}
\put(35, -25){\line(0, -1){10}}
\put(14, -33){$\cdots$}

\put(50, -35){.}
\end{picture}
\end{aligned}
\end{eqnarray}
\end{lemma}
\vspace{.2cm}
\begin{proof} For $k=0$, the formula is an identity.
The important case is $k=1$, where the formula becomes
\begin{eqnarray}\label{eq:k-1}
\begin{aligned}
\begin{picture}(80, 50)(20, -25)
\put(0, 10){\line(1, 0){60}}
\put(0, -10){\line(1, 0){60}}
\put(0, 10){\line(0, -1){20}}
\put(60, 10){\line(0, -1){20}}
\put(25, -3){$r$}

\put(5, 10){\line(0, 1){15}}
\put(15, 15){$\cdots$}
\put(35, 10){\line(0, 1){15}}
\qbezier(45, 10)(50, 35)(55, 10)

\put(5, -10){\line(0, -1){15}}
\put(15, -20){$\cdots$}
\put(35, -10){\line(0, -1){15}}
\qbezier(45, -10)(50, -35)(55, -10)

\put(65, -3){$=$}
\end{picture}
\begin{picture}(80, 50)(-40, -25)
\put(-65, -5){$4(r-2+\frac{\delta}{2})$}
\put(0, 10){\line(1, 0){50}}
\put(0, -10){\line(1, 0){50}}
\put(0, 10){\line(0, -1){20}}
\put(50, 10){\line(0, -1){20}}
\put(15, -3){$r-2$}

\put(10, 10){\line(0, 1){15}}
\put(18, 15){$\cdots$}
\put(40, 10){\line(0, 1){15}}

\put(10, -10){\line(0, -1){15}}
\put(18, -20){$\cdots$}
\put(40, -10){\line(0, -1){15}}

\put(60, -3){$+$}
\end{picture}
\begin{picture}(80, 60)(-90, -25)

\put(-55, -3){$(r-4)!^{-1}$}

\put(0, 10){\line(1, 0){50}}
\put(0, 25){\line(1, 0){50}}
\put(0, 10){\line(0, 1){15}}
\put(50, 10){\line(0, 1){15}}
\put(15, 15){\tiny $r-2$}

\put(5, 10){\line(0, -1){20}}
\put(25, 10){\line(0, -1){20}}
\put(8, -3){$\cdots$}

\qbezier(35,10)(40, -7)(45, 10)
\qbezier(35,-10)(40, 7)(45, -10)

\put(0, -10){\line(1, 0){50}}
\put(0, -25){\line(1, 0){50}}
\put(0, -10){\line(0, -1){15}}
\put(50, -10){\line(0, -1){15}}
\put(15, -20){\tiny $r-2$}

\put(5, 25){\line(0, 1){10}}
\put(45, 25){\line(0, 1){10}}
\put(18, 28){$\cdots$}

\put(5, -25){\line(0, -1){10}}
\put(45, -25){\line(0, -1){10}}
\put(18, -33){$\cdots$}

\put(55, -35){.}
\end{picture}
\end{aligned}
\end{eqnarray}
To prove it, we first obtain from Lemma \ref{lem:Sigma-1}(1) with $\epsilon=-1$
the following relation:
\[
\begin{picture}(80, 50)(0, -30)
\put(0, 10){\line(1, 0){60}}
\put(0, -10){\line(1, 0){60}}
\put(0, 10){\line(0, -1){20}}
\put(60, 10){\line(0, -1){20}}
\put(25, -3){$r$}

\put(5, 10){\line(0, 1){15}}
\put(15, 15){$\cdots$}
\put(35, 10){\line(0, 1){15}}
\qbezier(45, 10)(50, 35)(55, 10)

\put(5, -10){\line(0, -1){15}}
\put(15, -20){$\cdots$}
\put(35, -10){\line(0, -1){15}}
\qbezier(45, -10)(50, -35)(55, -10)

\put(65, -3){$=$}
\end{picture}
\begin{picture}(80, 60)(0, -30)
\put(0, 10){\line(1, 0){60}}
\put(0, -10){\line(1, 0){60}}
\put(0, 10){\line(0, -1){20}}
\put(60, 10){\line(0, -1){20}}
\put(20, -3){$r-1$}

\put(10, 10){\line(0, 1){15}}
\put(20, 15){$\cdots$}
\put(40, 10){\line(0, 1){15}}

\qbezier(50, -10)(60, -40)(65, 0)
\qbezier(50, 10)(60, 40)(65, 0)

\put(10, -10){\line(0, -1){15}}
\put(20, -20){$\cdots$}
\put(40, -10){\line(0, -1){15}}

\put(75, -3){$+$}
\end{picture}
\begin{picture}(80, 60)(-70, -30)
\put(-60, -3){$(r-2)!^{-1}$}

\put(0, 10){\line(1, 0){40}}
\put(0, 25){\line(1, 0){40}}
\put(0, 10){\line(0, 1){15}}
\put(40, 10){\line(0, 1){15}}
\put(12, 15){\tiny $r-1$}

\put(5, 10){\line(0, -1){20}}
\put(25, 10){\line(0, -1){20}}
\put(8, -3){$\cdots$}

\qbezier(35,10)(45, 0)(50, -15)
\qbezier(35,-10)(45, 0)(50, 15)
\qbezier(35,25)(45, 45)(50, 15)
\qbezier(35,-25)(45, -45)(50, -15)

\put(0, -10){\line(1, 0){40}}
\put(0, -25){\line(1, 0){40}}
\put(0, -10){\line(0, -1){15}}
\put(40, -10){\line(0, -1){15}}
\put(12, -20){\tiny $r-1$}

\put(5, 25){\line(0, 1){10}}
\put(25, 25){\line(0, 1){10}}
\put(10, 28){$\cdots$}

\put(5, -25){\line(0, -1){10}}
\put(25, -25){\line(0, -1){10}}
\put(10, -33){$\cdots$}

\put(55, -35){.}
\end{picture}
\]
Applying Lemma \ref{lem:Sigma-1}(2) to the first diagram on the right hand side, and further applying
Lemma \ref{lem:Sigma-1}(3) and the corresponding relation under the anti-involution $\ast$ to the
second diagram, we obtain \eqref{eq:k-1}.

The general case is proved by induction on $k$. From
\eqref{eq:Sigma} at $k$, we obtain
\[
\begin{picture}(80, 50)(50, -30)
\put(0, 10){\line(1, 0){60}}
\put(0, -10){\line(1, 0){60}}
\put(0, 10){\line(0, -1){20}}
\put(60, 10){\line(0, -1){20}}
\put(25, -3){$r$}

\put(5, 10){\line(0, 1){15}}
\put(15, 15){$\cdots$}
\put(35, 10){\line(0, 1){15}}
\qbezier(45, 10)(50, 35)(55, 10)

\put(5, -10){\line(0, -1){15}}
\put(7, -20){...}
\put(20, -10){\line(0, -1){15}}
\qbezier(45, -10)(50, -35)(55, -10)
\qbezier(25, -10)(30, -35)(35, -10)
\put(35, -20){...}
\put(35, -30){\tiny $k+1$}

\put(65, -3){$=$}
\end{picture}
\begin{picture}(80, 50)(-45, -30)
\put(-100, -5){$4k(r+\frac{\delta}{2}-k-1)$}
\put(-10, 10){\line(1, 0){60}}
\put(-10, -10){\line(1, 0){60}}
\put(-10, 10){\line(0, -1){20}}
\put(50, 10){\line(0, -1){20}}
\put(10, -3){$r-2$}

\put(-5, 10){\line(0, 1){15}}
\put(13, 15){$\cdots$}
\put(45, 10){\line(0, 1){15}}

\put(-5, -10){\line(0, -1){15}}
\put(-2, -20){...}
\put(10, -10){\line(0, -1){15}}
\qbezier(35, -10)(40, -35)(45, -10)
\qbezier(15, -10)(20, -35)(25, -10)
\put(25, -20){...}
\put(25, -30){\tiny $k$}

\put(55, -3){$+$}
\end{picture}
\begin{picture}(80, 60)(-105, -30)
\put(-75, -3){$(r-2-2k)!^{-1}$}

\put(0, 10){\line(1, 0){55}}
\put(0, 25){\line(1, 0){55}}
\put(0, 10){\line(0, 1){15}}
\put(55, 10){\line(0, 1){15}}
\put(18, 15){\tiny $r-2$}

\put(5, 10){\line(0, -1){20}}
\put(15, 10){\line(0, -1){20}}
\put(5, -3){...}

\qbezier(20,10)(25, -7)(30, 10)
\qbezier(40,10)(45, -7)(50, 10)
\put(30, 7){\tiny ...}
\put(33, -2){\tiny$k$}
\qbezier(20,-10)(25, 7)(30, -10)

\put(0, -10){\line(1, 0){40}}
\put(0, -25){\line(1, 0){40}}
\put(0, -10){\line(0, -1){15}}
\put(40, -10){\line(0, -1){15}}
\put(10, -20){\tiny $r-2k$}

\put(5, 25){\line(0, 1){10}}
\put(50, 25){\line(0, 1){10}}
\put(20, 28){$\cdots$}

\put(5, -25){\line(0, -1){10}}
\put(15, -25){\line(0, -1){10}}
\put(6, -33){...}
\qbezier(20,-25)(25, -45)(30, -25)
\put(50, -35){.}
\end{picture}
\]
Applying \eqref{eq:k-1} to the lower part of the second term on the right hand side, we arrive, after collecting
terms, at \eqref{eq:Sigma} with $k$ replaced by $k+1$. This completes the proof.
\end{proof}

\section{The enhanced Brauer category $\wt\CB(\delta)$.}\label{s:eb}

Given an element $\delta\in \K$, and a positive integer $m$, we shall, following \cite{LZ8}, define a tensor category $\wt\CB(\delta)$,  the  enhanced Brauer category, 
which contains a quotient of the usual Brauer category $\CB(\delta)$ as a subcategory. We shall see that the 
relations we impose imply a relationship between $\delta$ and $m$, so that for each $m$ there are only finitely
many values of $\delta$ which make our relations consistent. 
Both categories have objects $\N=\{0,1,2,\dots\}$ and morphisms which may be described diagrammatically.
There is an involution $^*:\Hom_{\wt\CB(\delta)}(k,\ell)\lr\Hom_{\wt\CB(\delta)}(\ell,k)$ which is described on diagrams by reflecting diagrams
in a horizontal line. This map may be interpreted as a functor from the category
to its opposite.

\subsection{Definition of $\wt\CB(\delta)$}\label{ss:enhanced} 
We have seen that $\CB(\delta)$ may be presented as the category with 
object set $\N$ and morphisms which are generated by the four morphisms $I,U,A$ and $X$ under composition,
tensor product and duality, subject to certain relations, which are described in Theorem \ref{thm:presentation}.
In the definition below, we shall make extensive use of the total anti-symmetriser $\Sigma_r=\Sigma_+(r)\in B_r^r(\delta)$ defined in Lemma \ref{lem:Sigma-1}, i.e., 
\be\label{eq:sigr}
\Sigma_r=\sum_{\pi\in\Sym_r}(-1)^{|\pi|}\pi,
\ee 
which is depicted diagramatically in Figure \ref{Sigma-r}. 

\begin{definition}[\cite{LZ8}] \label{def:wbrcat} Let $\K$ be a ring, $\delta\in R$ and $m\geq 2$ a positive integer.
The enhanced Brauer category $\wt\CB(\delta)$ is a category with a duality functor
$^*:\wt\CB(\delta)\to \wt\CB(\delta)^{\text{op}}$,
which takes each object to itself, and takes each diagram to its reflection in a horizontal line.
The object set of $\wt\CB(\delta)$ is $\N$. 
The $\K$-modules of morphisms are generated by the Brauer morphisms $I,U,A,X$ and new generators
$\Delta_m\in\wt B_0^m(\delta)$ and $\Delta^*_m$ where $\Delta^*_m=(\Delta_m)^*$, 
subject to the following relations and their transforms under $^*$, which describe the interaction of the new generators with the
Brauer morphisms.
\begin{enumerate}
\item The relations in Theorem \ref{thm:presentation} for the generators $I,U,A$ and $X$.
\item (Harmonicity) For each positive integer $r$ with $0\leq r\leq m-2$, $(I^{\ot r}\ot A\ot I^{\ot m-r-2})\circ\Delta_m=0$.
\item For each positive integer $r$ with $0\leq r\leq m-2$, $(I^{\ot r}\ot X\ot I^{\ot m-r-2})\circ\Delta_m=-\Delta_m$.
\item $\Delta_m\circ \Delta_m^*=\Sigma_m$.
\item $\Delta_m\ot I\ot \Delta_m=(c_{m+1}\ot I^{\ot m})\circ (\Delta_m\ot\Delta_m\ot I)$, where $c_{m+1}$ is the 
$(m+1)$-cycle $(m+1,m,m-1,\dots,1)\in\Sym_{m+1}$.
\end{enumerate}
\end{definition}
The new generator $\Delta_m$ will be depicted diagrammatically (as a morphism from $0$ to $m$) as follows.

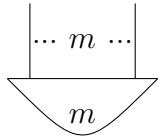
\centerline{
\begin{tikzpicture}
\draw (-1,0).. controls (0,-1) .. (1,0);
\draw (-1,0)-- (1,0);
\draw (-.7,0)-- (-.7,1);
\draw (.7,0)-- (.7,1);
\node at (0,.5) {$m$};
\node at (0,-.5) {$m$};
\node at (-.5,.5) {$...$};
\node at (.5,.5) {$...$};
\end{tikzpicture}}

\noindent
The relations above have suggestive diagrammatical interpretations, which are helpful in performing computations in the category
$\wt\CB(\delta)$. For example, the relation (4) may be depicted diagrammatically as in Fig. \ref{fig:Delta-1}, and the relation (5) is shown in Fig. \ref{fig:Delta-2}.

\begin{figure}[h]
\begin{tikzpicture}[scale=0.7]
\draw (-1,0).. controls (0,-1) .. (1,0);
\draw (-1,0)-- (1,0);
\draw (-.7,0)-- (-.7,1);
\draw (.7,0)-- (.7,1);
\node at (0,.5) {$m$};
\node at (0,-.5) {$m$};
\node at (-.5,.5) {\small$...$};
\node at (.5,.5) {\small$...$};

\draw (-1,-2).. controls (0,-1) .. (1,-2);
\draw (-1,-2)-- (1,-2);
\draw (-.7,-2)-- (-.7,-3);
\draw (.7,-2)-- (.7,-3);
\node at (0,-2.5) {$m$};
\node at (0,-1.5) {$m$};
\node at (-.5,.-2.5) {\small$...$};
\node at (.5,-2.5) {\small$...$};

\node at (2,-1) {$=$};

\node at (5,-2) {$m$};
\node at (4,-2) {$...$};
\node at (6,-2) {$...$};
\node at (5,-1) {$m$};
\node at (5,0) {$m$};
\node at (4,0) {$...$};
\node at (6,0) {$...$};
\draw (3,-1.5) rectangle (7,-.5);
\draw (3.5,-1.5)--(3.5,-2.5);\draw (6.5,-1.5)--(6.5,-2.5);
\draw (3.5,0.5)--(3.5,-.5);\draw (6.5,0.5)--(6.5,-.5);
\node at (7.5,-1.5) {$.$};
\end{tikzpicture}
\caption{The relation (4).}
\label{fig:Delta-1}
\end{figure}
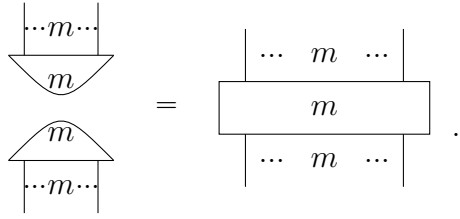

\begin{figure}[h]
\begin{tikzpicture}[scale=0.7]
\draw (0,0).. controls (1,-2) .. (2,0);
\draw (0,0)--(2,0);

\draw (3,0).. controls (4,-2) .. (5,0);
\draw (3,0)--(5,0);

\draw (0.25,0)--(0.25,2);
\draw (4.75,0)--(4.75,2);

\draw (1.75,0)--(1.75,2);
\draw (3.25,0)--(3.25,2);

\draw (2.5,-1.5)--(2.5,2);
\node at (1,1) {\small$...m...$};
\node at (4,1) {\small$...m...$};



\node at (1,-1) {\small$m$};
\node at (4,-1) {\small$m$};

\node at (6,0) {$=$};

\draw (8,0).. controls (9,-2) .. (10,0);
\draw (8,0)--(10,0);

\draw (11,0).. controls (12,-2) .. (13,0);
\draw (11,0)--(13,0);

\draw (8.25,0)--(8.25,2);
\draw (12.75,0)--(12.75,2);

\draw (9.75,0)--(9.75,2);
\draw (11.25,0)--(11.25,2);

\draw (2.5,-1.5)--(2.5,2);
\draw (7,-1.5).. controls (7.5,1) and (10,1) .. (10.5,2);
\node at (9,1) {\small$...m...$};
\node at (12,1) {\small$...m...$};



\node at (9,-1) {\small$m$};
\node at (12,-1) {\small$m$};




\end{tikzpicture}
\caption{The relation (5)}
\label{fig:Delta-2}
\end{figure}
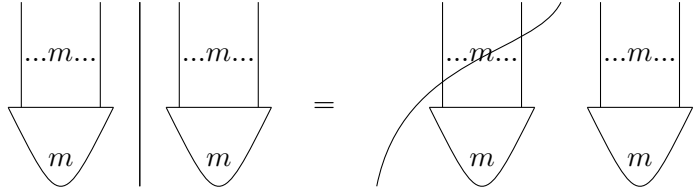

\begin{remark}\label{rem:nonzero}
With our application to invariant theory in mind, we shall assume that the base ring $\K$ is an integral domain, and that  $m!\neq 0$ in $\K$.

For any integers $r>0$ and $i=1,2,\dots,r-1$, write $\sigma_i=(I^{i-1})\ot X\ot (I^{\ot r-i-1})\in \wt{B}_r^r$. Evidently 
the $\sigma_i$ generate the symmetric group $\Sym_r\subseteq\wt{B}_r^r$,
and condition (3) of Definition \ref{def:wbrcat} asserts that the generators $\sigma_i$ of $\Sym_m$ satisfy
$\sigma_i\circ\Delta_m=-\Delta_m$, whence $w\circ \Delta_m=\ve(w)\Delta_m$
for $w\in\Sym_m$, where $\ve$ is the alternating character of $\Sym_m$.
It follows that $\Sigma_m\circ\Delta_m=m!\Delta_m$, and hence by the above assumptions, 
that, if $\Sigma_m=0$, then $\Delta_m=0$. If $\Sigma_m=0$, the category $\wt\CB(\delta)$ is therefore 
 just a quotient category of $\CB(\delta)$.

To avoid this degeneracy, we shall therefore assume that $\Sigma_m\neq 0$.
\end{remark}

\begin{remark}
Note that although $\delta$ does not appear explicitly in the definition above, it is inherent in the definition of $\CB(\delta)$, where 
it is stipulated that $U\circ A=\delta$ (note that $B_0^0=\K$). The integer $m$ enters only in the definition of $\wt\CB$.
\end{remark}

\subsection{Some properties of the category $\wt\CB(\delta)$} 

The defining relations of $\wt\CB(\delta)$ implies stringent conditions on the morphisms and on the parameter $\delta$.  In particular, we have the following results, which are extracted from \cite[Theorem 5.7]{LZ8}. 
\begin{theorem}[\cite{LZ8}]\label{thm:comp}  Assume that $m!\neq 0$ in $\K$ and that $\Sigma_m\neq 0$ as a morphism in $\wt\CB(\delta)$. Then 
\[
\delta=m,  \ \  \text{  and  }  \  \  \Sigma_{m+1}=0.
\]
\end{theorem}
This follows from the lemma below.

\begin{lemma}\label{lem:dreln}
Assume that $m!\neq 0$ in $\K$ and that $\Sigma_m\neq 0$.  Then the following hold in the category $\wt\CB(\delta)$.
\begin{enumerate}
\item We have $\Delta_m^*\Delta_m=m!\in \K$ as morphism in $\wt\CB(\delta)$.
\item 
The parameter $\delta$ satisfies the polynomial equation 
\be\label{eq:deltarel}
\delta(\delta-1)\dots(\delta-(m-1))=m!.
\ee
\item We have the equality of morphisms
$\Sigma_{m+1}=f_m(\delta) \Sigma_m\ot I$. Here $m$ is the positive integer occurring in the definition of $\wt\CB(\delta)$
and $f_m$ is the polynomial in $\delta$ given by $f_m(\delta)=(\delta-(m-1))(\delta-(m-2))\dots(\delta-1)-(m-1)!$.
\end{enumerate}
\end{lemma}

\begin{proof}
We shall compute $\Delta_m^*\Delta_m$ in two different ways. First,
observe that by inspection of the relevant diagrams, it is evident that
 $(\Delta_m\Delta_m^*)^2=(\Delta_m(\Delta_m^*\Delta_m)\Delta_m^*)$, where $\Delta_m^*\Delta_m$ is a scalar;
thus applying Relation (4), $(\Delta_m\Delta_m^*)^2=(\Delta_m^*\Delta_m)\Sigma_m.$

But again by Relation (4), $(\Delta_m\Delta_m^*)^2=\Sigma_m^2=m!\Sigma_m$, whence
comparing coefficients of the non-zero element $\Sigma_m$, it follows that
\be\label{eq:dstd}
\Delta_m^*\Delta_m=m!.
\ee

Next, note that
we have the relations (6) and (7) depicted in Fig. \ref{fig:Relation-(8-1)} and Fig. \ref{fig:Relation-(8-2)} respectively.

\vspace{-1.5cm}
\begin{figure}[h]
\begin{tikzpicture}[scale=0.72]
\draw (-5,-.5).. controls (-4,0.5) .. (-3,-.5);
\draw (-5,-.5)-- (-3,-.5);

\draw (-5,-1.5).. controls (-4,-2.5) .. (-3,-1.5);
\draw (-5,-1.5)-- (-3,-1.5);

\node at (-4,0) {$m$};
\node at (-4,-2) {$m$};
\node at (-4,-1) {$...$};

\node at (-2.4,-1) {$=$};

\draw (-4.75,-.5)-- (-4.75,-1.5);
\draw (-3.25,-.5)-- (-3.25,-1.5);

\draw (-1,0).. controls (0,-1) .. (1,0);
\draw (-1,0)-- (1,0);
\draw (-.7,0).. controls (-1.5,1) and (-1.5,-3).. (-.7,-2);
\node at (0,-.5) {$m$};

\draw (-1,-2).. controls (0,-1) .. (1,-2);
\draw (-1,-2)-- (1,-2);
\draw (.7,0).. controls (-2.5,3) and (-2.5,-5).. (.7,-2);
\node at (0,-1.5) {$m$};
\node at (0,.25) {$...$};
\node at (0,-2.25) {$...$};


\end{tikzpicture}
\vspace{-1.5cm}
\caption{Relation (6)}
\label{fig:Relation-(8-1)}
\end{figure}
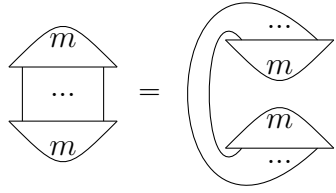

\vspace{-2.2cm}
\begin{figure}[h]
\begin{tikzpicture}[scale=0.7]
\draw (-1,0).. controls (0,-1) .. (1,0);
\draw (-1,0)-- (1,0);
\draw (-.7,0).. controls (-1.5,1) and (-1.5,-3).. (-.7,-2);
\node at (0,-.5) {$m$};

\draw (-1,-2).. controls (0,-1) .. (1,-2);
\draw (-1,-2)-- (1,-2);
\draw (.7,0).. controls (-2.5,3) and (-2.5,-5).. (.7,-2);
\node at (0,-1.5) {$m$};
\node at (0,.25) {$...$};
\node at (0,-2.25) {$...$};

\node at (1.5,-1) {$=$};

\node at (5,-1) {$m$};
\node at (5,0) {$...$};
\node at (5,-2) {$...$};

\draw (3,-1.5) rectangle (7,-.5);

\draw (3.5,-1.5).. controls (2,-3) and (2,1).. (3.5,-.5);
\draw (6.75,-1.5).. controls (0.6,-6) and (0.6,4).. (6.75,-.5);
\end{tikzpicture}
\vspace{-2cm}
\caption{Relation (7)}
\label{fig:Relation-(8-2)}
\end{figure}
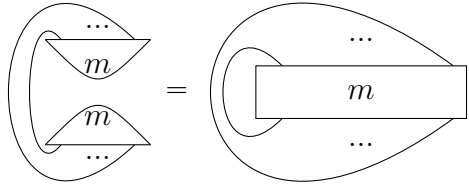

The right side of relation (7) is, by $m$ applications of Lemma \ref{lem:red}, equal to $(\delta-(m-1))(\delta-(m-2))\dots(\delta-1)\delta$,
while the left side of relation (6) is just $\Delta_m^*\Delta_m$. Equation \eqref{eq:deltarel} is now clear from \eqref{eq:dstd} by combining relations (6) and (7).

This proves part (1) and part (2). 

To prove part (3), we shall make liberal use, both explicit and implicit, 
of the mutually inverse isomorphisms $\bU_p^q:\wt B_{p+q}^r\to \wt B_p^{r+q}$
and $\bA_q^r:\wt B_p^{r+q}\to\wt B_{p+q}^r$ defined in Corollary \ref{isomorphism}. 
Note that these isomorphisms involve only operations (tensor product and composition) 
with the Brauer morphisms in $\wt\CB(\delta)$.

Note that in our situation, the relation (1) of Lemma \ref{lem:red} implies  the relation (8) given in Figure \ref{fig:Relation-(6)}.
\begin{figure}[h]
\setlength{\unitlength}{0.4mm}
\begin{picture}(80, 60)(0, -30)
\put(0, 10){\line(1, 0){60}}
\put(0, -10){\line(1, 0){60}}
\put(0, 10){\line(0, -1){20}}
\put(60, 10){\line(0, -1){20}}
\put(20, -3){$m+1$}

\put(10, 10){\line(0, 1){15}}
\put(23, 15){$\cdots$}
\put(50, 10){\line(0, 1){15}}

\put(10, -10){\line(0, -1){15}}
\put(23, -20){$\cdots$}
\put(50, -10){\line(0, -1){15}}

\put(70, -2){$=$}
\end{picture}
\begin{picture}(80, 60)(-10, -30)
\put(0, 10){\line(1, 0){50}}
\put(0, -10){\line(1, 0){50}}
\put(0, 10){\line(0, -1){20}}
\put(50, 10){\line(0, -1){20}}
\put(15, -3){$m$}

\put(10, 10){\line(0, 1){15}}
\put(18, 15){$\cdots$}
\put(40, 10){\line(0, 1){15}}

\put(10, -10){\line(0, -1){15}}
\put(18, -20){$\cdots$}
\put(40, -10){\line(0, -1){15}}

\put(60, -25){\line(0, 1){50}}

\put(70, -3){$-$}

\end{picture}
\begin{picture}(80, 60)(-70, -30)

\put(-60, -3){$ (m-1)!^{-1}$}

\put(0, 10){\line(1, 0){40}}
\put(0, 25){\line(1, 0){40}}
\put(0, 10){\line(0, 1){15}}
\put(40, 10){\line(0, 1){15}}
\put(12, 15){\tiny$m$}

\put(5, 10){\line(0, -1){20}}
\put(25, 10){\line(0, -1){20}}
\put(8, -3){$\cdots$}

\qbezier(35,10)(45, 5)(50, -35)
\qbezier(35,-10)(45, -5)(50, 35)

\put(0, -10){\line(1, 0){40}}
\put(0, -25){\line(1, 0){40}}
\put(0, -10){\line(0, -1){15}}
\put(40, -10){\line(0, -1){15}}
\put(12, -20){\tiny$m$}

\put(5, 25){\line(0, 1){10}}
\put(35, 25){\line(0, 1){10}}
\put(14, 28){$\cdots$}

\put(5, -25){\line(0, -1){10}}
\put(35, -25){\line(0, -1){10}}
\put(14, -33){$\cdots$}

\end{picture}
\caption{Relation (8)}
\label{fig:Relation-(6)}
\end{figure}
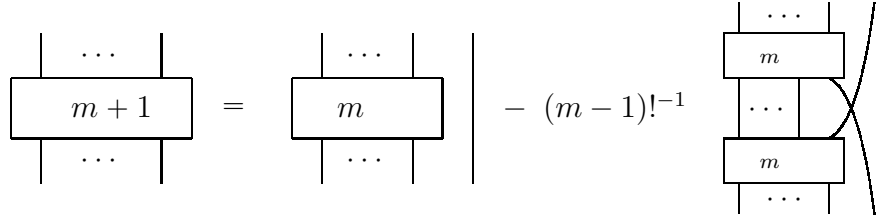

We replace each of the two rectangles in the second summand on the right side by the left side of Figure \ref{fig:Delta-1}. 
A little manipulation then shows that the result will follow if we prove the relation (9) in Fig. \ref{fig:Relation-(9)}.

\medskip
{
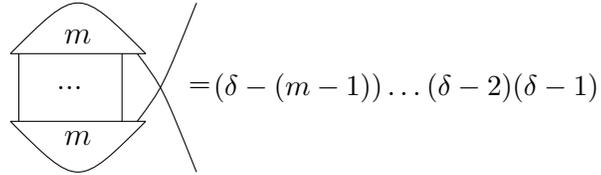
\begin{figure}[h]
\begin{tikzpicture}[scale=0.450]
\draw (1,1).. controls (3,3) .. (5,1);
\draw (1,1)--(5,1);

\draw (1,-1).. controls (3,-3) .. (5,-1);
\draw (1,-1)--(5,-1);

\draw (1.25,-1)--(1.25,1);
\draw (4.3,-1)--(4.3,1);

\draw (4.75,1).. controls (5.5,0) .. (6.5,-2.5);
\draw (4.75,-1).. controls (5.5,0) .. (6.5,2.5);

\node at (3,1.5) {$m$};
\node at (3,-1.5) {$m$};
\node at (2.75,0) {$...$};


\node at (6.6,0) {$=$};
\node at (12.7,0) {\small$(\delta-(m-1))\dots(\delta-2)(\delta-1)$};

\draw (18.8,-2.5)--(18.8,2.5);

\end{tikzpicture}
\caption{Relation (9)}
\label{fig:Relation-(9)}
\end{figure}
}

Next, observe that by rotating the top half of the left side of relation (9) in Fig. \ref{fig:Relation-(9)} anticlockwise by 
$\pi$ and then applying the isomorphism $\bU_0^1$
from $\CB_1^1$ to $\CB_0^2$, the relation (9) is equivalent to relation (10) in Fig. \ref{fig:Relation-(10)}.

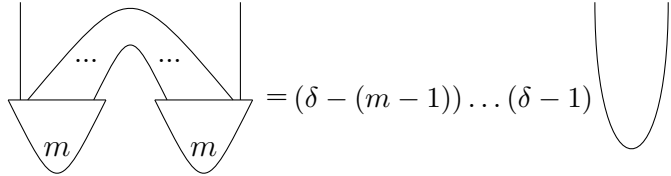
\begin{figure}[h]
\begin{tikzpicture}[scale=0.65]
\draw (0,0).. controls (1,-2) .. (2,0);
\draw (0,0)--(2,0);

\draw (3,0).. controls (4,-2) .. (5,0);
\draw (3,0)--(5,0);

\draw (0.25,0)--(0.25,2);
\draw (4.75,0)--(4.75,2);

\draw (1.75,0).. controls (2.5,1.5) .. (3.25,0);
\draw (0.4,0).. controls (2.5,2.5) .. (4.6,0);

\node at (1,-1) {$m$};
\node at (4,-1) {$m$};

\node at (1.6,.8) {\small$...$};
\node at (3.3,.8) {\small$...$};

\node at (5.5,0) {$=$};
\node at (8.9,0) {\small$(\delta-(m-1))\dots(\delta-1)$};

\draw (12,2).. controls (12,-2) and (13.5,-2).. (13.5,2);

\end{tikzpicture}
\vspace{-.3cm}
\caption{Relation (10)}
\label{fig:Relation-(10)}
\end{figure}

Now to prove relation (10), observe first that applying the isomorphism $\bU_0^m$ to both sides of
the relation (4) as shown in Fig. \ref{fig:Delta-1}, we obtain the relation (11) in Fig. \ref{fig:Relation-(11)}.

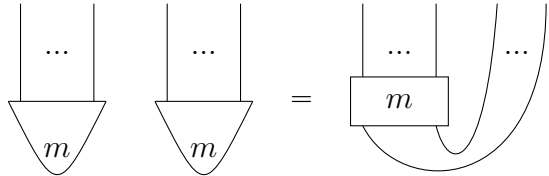
\begin{figure}[h]
\begin{tikzpicture}[scale=0.65]
\draw (0,0).. controls (1,-2) .. (2,0);
\draw (0,0)--(2,0);

\draw (3,0).. controls (4,-2) .. (5,0);
\draw (3,0)--(5,0);

\draw (0.25,0)--(0.25,2);
\draw (4.75,0)--(4.75,2);

\draw (1.75,0) -- (1.75,2);
\draw (3.25,0) -- (3.25,2);

\node at (1,-1) {$m$};
\node at (4,-1) {$m$};

\node at (6,0) {$=$};

\draw (7,-0.5) rectangle (9,0.5);
\draw (7.25,.5) -- (7.25,2);
\draw (8.75,.5) -- (8.75,2);

\draw (8.75,-.5).. controls (9,-1.5) and (9.75,-1.5).. (10,2);
\draw (7.25,-.5).. controls (8,-2) and (11,-2).. (11,2);

\node at (1,1) {$...$};
\node at (4,1) {$...$};
\node at (8,0) {$m$};
\node at (10.4,1) {$...$};
\node at (8,1) {$...$};

\end{tikzpicture}
\vspace{-.3cm}
\caption{Relation (11)}
\label{fig:Relation-(11)}
\end{figure}

Then, applying $I\ot A^{\ot m-1}\ot I$ to both sides of relation (11), and applying Lemma \ref{lem:red} (7) $m-1$ times, 
we obtain the relation (10) of Fig. \ref{fig:Relation-(10)}, and the proof of part (3) is complete.

This proves the lemma. 
\end{proof}

\begin{proof}[Proof of Theorem \ref{thm:comp}]
Note that it follows from part (3) of Lemma \ref{lem:dreln} that
\[
\Sigma_{m+1}\circ (I^{\ot m-1}\ot U)=f_m(\delta) (\Sigma_m\ot I)\circ(I^{\ot m-1}\ot U).
\]
But the left side of this equation is evidently zero, while the right side is an invetible multiple of $f_m(\delta)\Sigma_m\ot I$.
It follows that $f_m(\delta)=0$, and hence by (1), that $\Sigma_{m+1}=0$.

To prove the second statement, observe that $f_m(\delta)=0$ implies that $\delta(\delta-1)\dots(\delta-(m-1))=\delta(m-1)!$.
Comparing this to the relation $\delta(\delta-1)\dots(\delta-(m-1))=m!$ of  \eqref{eq:deltarel}, we see that $\delta=m$.
\end{proof}

We will also need the following result.
\begin{lemma}\label{lem:str} Let $\wt\CB_0$ be the subcategory of $\wt\CB(m)$ generated by all Brauer 
diagrams (i.e. by the morphisms $I,X,A$ and $U$). 

\begin{enumerate}
\item Each diagram of $\wt\CB(m)$ is either in $\wt\CB_0$ or is obtained 
from $\Delta_m$ by tensoring and composing with elements of $\wt\CB_0$.
\item Let $s,t\in\N$. Then 
\[
\wt B_s^t=\wt B_{s,0}^t\oplus \wt B_{s,1}^t,
\]
where $\wt B_{s,0}^t$ is the span of the Brauer diagrams in $\wt B_s^t$, and 
$\wt B_{s,1}^t$ is the span of diagrams of the {second} type described in (1).
\end{enumerate}
\end{lemma}
\begin{proof}
If the diagram $D\in\wt\CB(m)$ is not in $\wt\CB_0$, then it may be expressed as a `word' in the generators $I,X,A,U$ and $\Delta_m$, 
with connectives $\ot$ (tensor product) and $\circ$ (composition), since $\Delta_m^*=\bA_m^0(\Delta_m)$. 
But the relation (5) in Fig. \ref{fig:Delta-2} above shows that any diagram with two occurrences 
of $\Delta_m$, is equal in $\wt\CB$ to an element where the occurrences are adjacent. Hence by the relation
(4) in Fig. \ref{fig:Delta-1}, this diagram is equal to one in $\wt\CB_0$. Thus we may assume that there is precisely one
occurrence of $\Delta_m$ in the word expression for $D$. This proves (1).

The statement (2) is an immediate consequence of (1), since each $\Hom$ space is spanned by diagrams, and the two
types of diagrams in (1) are complementary.
\end{proof}

\section{The oriented Brauer category}\label{sect:orient-bc}

We describe an oriented Brauer category, the category of Brauer diagrams with oriented arcs, which is a categorification of the walled Brauer algebras.

\subsection{Definition of the oriented Brauer category}

Let $\CN$ be the set of sequences
$(\varepsilon_1, \varepsilon_2, ..., \varepsilon_k)$, where  $k\in \N$
and $\varepsilon_i\in \{+, \ - \}$, with the $k=0$ case corresponding to the empty sequence. 
Define two functions $sl:\CN\longrightarrow \N\times \N$ and $l: \CN\longrightarrow \N$ as follows. Let $\#_\pm(\eta)$ denote the numbers 
of $\pm$'s in $\eta\in\CN$ respectively. Then $sl(\eta)=(\#_+(\eta), \#_-(\eta))$, and $l(\eta)=\#_+(\eta)+\#_-(\eta)$. 
We have the following three operations on $\CN$.
\begin{enumerate}
\item Joining sequences.  
Any two given sequences $\eta, \zeta\in\CN$ can be {\it  joined} by concatenation to form  a new sequence $(\eta, \zeta)$.
Clearly $\#_\pm(\eta, \zeta)=\#_\pm(\eta)+\#_\pm(\zeta)$, where addition of pairs is component-wise. 
\item If we write $+^\vee=-$ and  $-^\vee=+$, the negative $^\vee:\CN\longrightarrow\CN$ is defined by 
\[
\eta=(\varepsilon_1, \varepsilon_2, \dots, \varepsilon_k)\mapsto \eta^\vee=(\varepsilon^\vee_1, \varepsilon^\vee_2, \dots, \varepsilon^\vee_k).
\]
\item The opposite (or reverse) of a sequence $\eta=(\ve_1,\dots,\ve_k)\in\CN$ is $op(\eta):=(\ve_k,\dots, \ve_1)$.
\end{enumerate}

Given any $(k, \ell)$ Brauer diagram, we arbitrarily assign an orientation to each of its arcs, indicated by an arrow. 

\begin{definition}
 An {\em oriented Brauer diagram} is defined to be a Brauer diagram with oriented arcs. Each oriented arc will be said to have starting point
 and end point, with the obvious meaning.
 \end{definition}

We now associate to any oriented $(k, \ell)$ diagram $\Gamma$ two elements of $\CN$ as follows:
the source $s(\Gamma)=(\varepsilon_1,   \varepsilon_2, ...,
\varepsilon_k)$ and target $t(\Gamma)=(\varepsilon'_1,  \varepsilon'_2, ...,
 \varepsilon'_\ell)$,
in the following way.
If the $p$-th bottom vertex is an end point of a string with the arrow pointing outward (resp. inward), then $\varepsilon_p=+$ (resp. $\varepsilon_p=-$). Similarly, if 
the $q$-th top vertex is an end point of a string with the arrow pointing inward (resp. outward), then $\varepsilon'_q=+$ (resp. $\varepsilon'_q=-$).

Figure \ref{fig:orient-diagram} below is an oriented $(5, 5)$ Brauer diagram, which has source $(-, - + + -)$ and target $(- - + - +)$. 
\begin{figure}[h]
\begin{center}
\begin{picture}(160, 60)(0,0)

\qbezier(0, 60)(30, 0)(60, 60)
\qbezier(0, 0)(60, 60)(90, 0)

\qbezier(30, 60)(90, 20)(120, 60)
\qbezier(90, 60)(60, 30)(30, 0)

\qbezier(60, 0)(90, 40)(120, 0)

\put(88, 3){\vector(1,-1){5}}
\put(62, 3){\vector(-1, -1){5}}

\put(2, 58){\vector(-1, 1){5}}
\put(32, 58){\vector(-1, 1){5}}
\put(88, 58){\vector(1, 1){5}}
\end{picture}
\end{center}
\caption{An oriented Brauer diagram }
\label{fig:orient-diagram}
\end{figure}
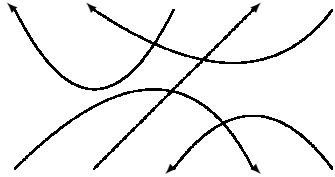

\begin{lemma}\label{lem:numbers}
For any oriented Brauer diagram $\Gamma$, we have
\[
\#_+(t(\Gamma))+\#_-(s(\Gamma))=\#_+(s(\Gamma))+\#_-(t(\Gamma))=\frac{l(t(\Gamma))+l(s(\Gamma))}{2}.
\]
\end{lemma}

This is evident from the fact that the left side counts the number of end points of the arcs in $\Gamma$, while the middle term counts the number 
of starting points of arcs. Clearly, both are equal to the number of arcs, which is the right hand term.

Fix $\delta\in \K$. Denote by ${OB}_s^t(\delta)$ the free
$\K$-module with a basis consisting of
oriented Brauer diagrams with source $s$ and target $t$. 
The $\K$-module $OB_\eta^\zeta(\delta)$ is $0$ unless $l(\eta)+l(\zeta)$ is even. 

As for usual Brauer diagrams, we have the $\K$-linear map 
\begin{eqnarray}\label{eq:products-or}
\begin{aligned}
&&\text{composition} & \quad \circ:  & & OB_t^u(\delta)
\times OB_s^t(\delta)\longrightarrow OB_s^u(\delta), 
\end{aligned}
\end{eqnarray} 
given by concatenation of diagrams followed by loop removal. The concatenation $\Gamma\#\Gamma'$ of two oriented Brauer diagrams is defined only if $t(\Gamma')=s(\Gamma)$,  which is also an oriented Brauer diagram. Any loop in the concatenation is also oriented; loop removal replaces each loop by a factor $\delta$ irrespective of the orientation. 
We also have the $\K$-linear map 
\begin{eqnarray}\label{eq:tensor-or}
\text{tensor product} & \quad  \otimes:  & OB_u^v(\delta)
\times OB_s^t(\delta)\longrightarrow OB_{(u, s)}^{(v, t)}(\delta)
\end{eqnarray} 
defined by juxtaposition of oriented Brauer diagrams in exactly the same way as for ordinary Brauer diagrams. 

\begin{definition}
The {\em oriented Brauer category} with parameter $\delta\in\K$, denoted by $\OB(\delta)$,  is the $\K$-linear category, which has the object set 
$\mathcal N$ and morphism sets $\Hom(\eta, \eta')=OB_\eta^{\eta'}(\delta)$  as defined above
for any $\eta, \eta'\in\CN$. The composition of morphisms
is defined  by \eqref{eq:products-or}
\end{definition}

\begin{remark}\label{rem:or-tang} The oriented Brauer category can be obtained as a quotient category of the oriented tangle category \cite{FY, RT} (see also \cite[\S XII.2.2]{Ka} )  by identifying over-crossings with under-crossings and imposing the loop-removal relation given in Theorem \ref{thm:tensor-cat-or}(3)(g) below. 
\end{remark}

\subsection{Generators and relations}
Oriented Brauer diagrams can also be described in terms generators and relations as for ordinary Brauer diagrams. The following theorem mirrors  \cite[Theorem XII.2.2]{Ka} for oriented tangles  \cite{FY, T1, RT, T2}. 

\begin{theorem} \label{thm:tensor-cat-or}
The category $\OB(\delta)$ has the following properties. 
\begin{enumerate} 
\item There is a bi-functor $\otimes: \OB(\delta)\times \OB(\delta)\longrightarrow \OB(\delta)$, called the tensor product, which is defined as follows. 
For any pair of objects $\eta$ and $\eta'$, we have
$
\eta\otimes\eta'=(\eta, \eta').
$
The tensor product of morphisms is given by the bilinear map \eqref{eq:tensor-or}. 

\item The morphisms are generated by the following elementary diagrams under tensor product and composition, 

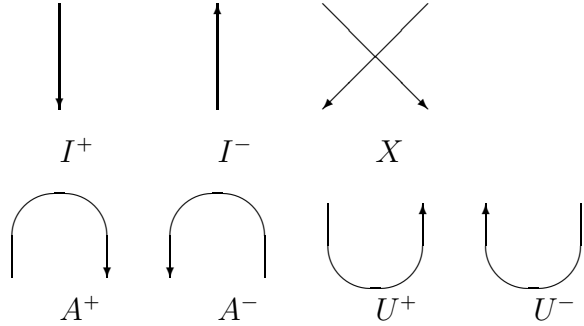
\begin{figure}[h]
\setlength{\unitlength}{.7mm}
\begin{picture}(140, 65)(-15, 0)
\put(10,0){$A^+$}
\put(40,0){$A^-$}
\put(70,0){$U^+$}
\put(100,0){$U^-$}
\put(10, 30){$I^+$}
\put(40,30){$I^-$}
\put(70,30){$X$}

\put(10,60){\vector(0,-1){20}}
\put(40,40){\vector(0,1){20}}

\put(60,60){\vector(1,-1){20}}
\put(80,60){\vector(-1,-1){20}}


\put(10,15){\oval(18,18)[t]}
\put(1,15){\line(0,-1){7}}
\put(19,15){\vector(0,-1){7}}

\put(40,15){\oval(18,18)[t]}
\put(31,15){\vector(0,-1){7}}
\put(49,15){\line(0,-1){7}}

\put(70,15) {\oval(18,18)[b]}
\put(61,15){\line(0,1){7}}
\put(79,15){\vector(0,1){7}}

\put(100,15){\oval(18,18)[b]}
\put(91,15){\vector(0,1){7}}
\put(109,15){\line(0,1){7}}
\end{picture}
\caption{Generators of oriented Brauer diagrams}
\label{fig:generators}
\end{figure}

\item The defining relations among the above generators are as follows.
\begin{enumerate}
\item Involution property of crossing: 
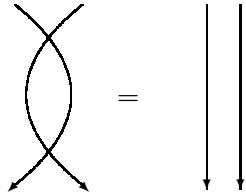
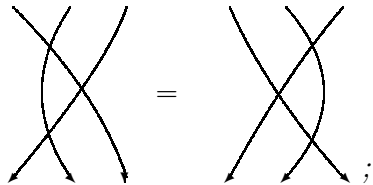
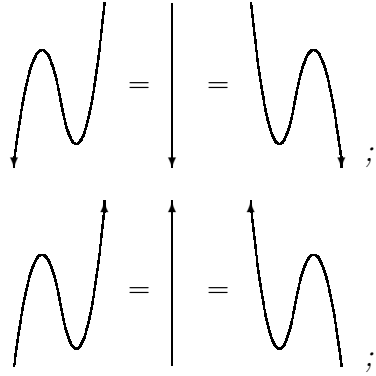
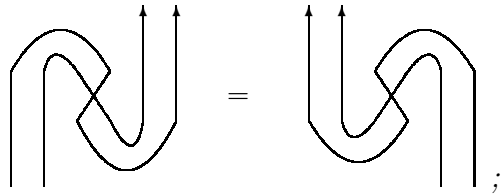
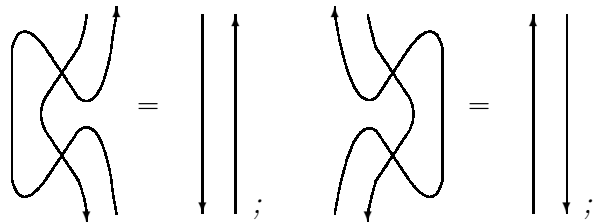
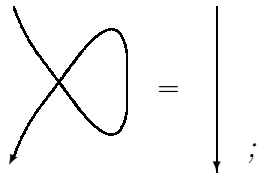
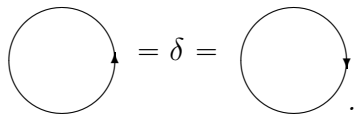
\[
\setlength{\unitlength}{0.3mm}
\begin{picture}(70, 90)(65, 0)
\qbezier(60, 0)(110, 40)(60, 80)
\put(62, 2){\vector(-1, -1){5}}

\qbezier(90, 0)(40, 40)(90, 80)
\put(88, 2){\vector(1, -1){5}}

\put(105, 35){$=$}
\end{picture}
\begin{picture}(50, 90)(110, 0)
\put(120, 0){\line(0, 1){80}}
\put(135, 0){\line(0, 1){80}}
\put(120, 3){\vector(0, -1){5}}
\put(135, 3){\vector(0, -1){5}}
\end{picture}
\]

\item Braid relation:  
\[\setlength{\unitlength}{0.25mm}
\begin{picture}(85, 100)(0, 0)
\qbezier(30, 90)(0, 45)(30, 0)
\put(28, 2){\vector(1, -1){5}}

\qbezier(0, 90)(45, 45)(60, 0)
\put(59, 2){\vector(0, -1){5}}

\qbezier(60, 90)(50, 60)(0,0)
\put(2, 2){\vector(-1, -1){5}}

\put(75, 40){$=$}
\end{picture}
\begin{picture}(85, 100)(-30, 0)
\qbezier(0, 90)(20, 45)(60, 0)
\put(59, 2){\vector(1, -1){5}}

\qbezier(30, 90)(70, 45)(30, 0)
\put(32, 2){\vector(-1, -1){5}}

\qbezier(60, 90)(30, 55)(0, 0)

\put(2, 2){\vector(-1, -1){5}}
\put(70, 0){;}
\end{picture}
\]

\item Straightening relations:
\[
\begin{aligned}
\setlength{\unitlength}{0.3mm}
\begin{picture}(150, 80)(0,0)

\qbezier(0, 0)(10, 80)(20, 30)
\put(0, 2){\vector(0, -1){5}}

\qbezier(20, 30)(30, -30)(40, 70)
\put(50, 30){$=$}
\put(70, 0){\line(0, 1){70}}
\put(70, 2){\vector(0, -1){5}}

\put(85, 30){$=$}
\qbezier(105, 70)(115, -30)(125, 30)
\qbezier(125, 30)(135, 80)(145, 0)
\put(145, 2){\vector(0, -1){5}}
\put(155, 0){;}
\end{picture}
\\
\setlength{\unitlength}{0.3mm}
\begin{picture}(150, 80)(0,0)

\qbezier(0, 0)(10, 80)(20, 30)
\put(40, 68){\vector(0, 1){5}}

\qbezier(20, 30)(30, -30)(40, 70)
\put(50, 30){$=$}
\put(70, 0){\line(0, 1){70}}
\put(70, 68){\vector(0, 1){5}}

\put(85, 30){$=$}
\qbezier(105, 70)(115, -30)(125, 30)
\qbezier(125, 30)(135, 80)(145, 0)
\put(105, 68){\vector(0, 1){5}}
\put(155, 0){;}
\end{picture}
\end{aligned}
\]

\item Orientation reversed crossing:
\[
 \setlength{\unitlength}{0.22mm}
 \begin{picture}(80, 130)(40,0)
 
\qbezier(40, 70)(40, 70)(60, 40)
\qbezier(60, 70)(60, 70)(40, 40)

\qbezier(60, 40)(75, 10) (80, 40)
\put(80, 40){\vector(0, 1){70}}
\put(100, 40){\vector(0, 1){70}}
\put(0, 0){\line(0, 1){70}}
\put(20, 0){\line(0, 1){70}}
\qbezier(20, 70)(25, 90) (40, 70)


\qbezier(40, 40)(70, -20) (100, 40)
\qbezier(60, 70)(30, 120) (0, 70)
 


\put(130, 50){$=$}
\end{picture}
%
%
 \setlength{\unitlength}{0.22mm}
 \begin{picture}(80, 130)(-60,0)
 
\qbezier(40, 70)(40, 70)(60, 40)
\qbezier(60, 70)(60, 70)(40, 40)

\put(80, 0){\line(0, 1){70}}
\put(100, 0){\line(0, 1){70}}

\put(20, 40){\vector(0, 1){70}}
\put(0, 40){\vector(0, 1){70}}


\qbezier(60, 70)(75, 90)(80, 70)
\qbezier(40, 70)(70, 120)(100, 70)

\qbezier(40, 40)(25, 20)(20, 40)
\qbezier(60, 40)(30, -10)(0, 40)
\put(110, 0){;}
\end{picture}
\]

\item Sliding relations: 
 \[
 \begin{aligned}
 \setlength{\unitlength}{0.22mm}
 \begin{picture}(80, 130)(0,0)
 
 \qbezier(0, 100)(0, 100)(20, 70)
\qbezier(0, 70)(0, 70)(20, 100)
\qbezier(20, 50)(0, 20)(0, 20)
\qbezier(0, 50)(20, 20)(20, 20)
 

\qbezier(0, 70)(-5, 60)(0, 50)

\qbezier(0, 100)(-15, 120)(-20, 100)
\qbezier(0, 20)(-15, 0)(-20, 20)
\put(-20, 100){\line(0, -1){80}}

\qbezier(20, 50)(33, 60)(40, 20)
\qbezier(20, 70)(33, 60)(40, 100)


\qbezier(20, 100)(24, 110)(25, 120)
\qbezier(40, 100)(41, 110)(43, 120)
\qbezier(20, 20)(24, 10)(25, 0)
\qbezier(40, 20)(41, 10)(43, 0)

\put(43, 120){\vector(0, 1){5}}
\put(25, 0){\vector(0, -1){5}}
\put(55, 60){$=$}
\end{picture}
\begin{picture}(20, 130)(-15,20)
\put(0, 140){\vector(0, -1){120}}
\put(20, 20){\vector(0, 1){120}}
\put(30, 20){;}
\end{picture}
%
%
 \setlength{\unitlength}{0.22mm}
 \begin{picture}(80, 130)(-100,0)
 
\qbezier(0, 100)(0, 100)(20, 70)
\qbezier(0, 70)(0, 70)(20, 100)
\qbezier(20, 50)(0, 20)(0, 20)
\qbezier(0, 50)(20, 20)(20, 20)

\put(40, 100){\line(0, -1){80}}

\qbezier(20, 70)(25, 60)(20, 50)
\qbezier(20, 100)(35, 120)(40, 100)
\qbezier(20, 20)(35, 0)(40, 20)

\qbezier(0, 100)(-2, 105)(-5, 120)
\qbezier(0, 20)(-2, 15)(-5, 0)

\qbezier(0, 70)(-15, 60)(-25, 120)
\qbezier(0, 50)(-15, 60)(-25, 0)

\put(-5, 0){\vector(0, -1){5}}
\put(-25, 120){\vector(0, 1){5}}

\put(55, 60){$=$}

\end{picture}
\begin{picture}(120, 130)(-115,20)
\put(20, 140){\vector(0, -1){120}}
\put(00, 20){\vector(0, 1){120}}

\put(30, 20){;}
\end{picture}
\end{aligned}
 \]

 \item De-loop:
  \[
 \setlength{\unitlength}{0.2mm}
 \begin{picture}(80, 120)(0,20)
 
\qbezier(0, 100)(-10, 115)(-15, 130)
 
\qbezier(0, 100)(30, 60)(30, 60)
\qbezier(30, 100)(30, 100)(0, 60) 
\qbezier(60, 100)(60, 100)(60, 60)
 
\qbezier(30, 100)(55, 130)(60, 100)
\qbezier(30, 60)(55, 30)(60, 60)
 
\qbezier(0, 60)(-10, 45)(-15, 30)

\put(-13, 35){\vector(-1, -2){5}}
\put(80, 70){$=$}
\put(120, 130){\vector(0, -1){110}}
\put(140, 30){;}
\end{picture}
 \]
\item Loop removal:
\[
 \setlength{\unitlength}{0.22mm}
 \begin{picture}(200, 80)(0,0)
\put(40, 40){\circle{80}}
\put(72, 40){\vector(0, 1){5}}
\put(85, 40){$=\delta$}
\put(120, 40){$=$}
\put(180, 40){\circle{80}}
\put(212, 40){\vector(0, -1){5}}
\put(212, 10){.}
\end{picture}
\] 
 \end{enumerate}
\end{enumerate}
\begin{proof}
Part (1) of the theorem is clear. 

To prove part (2), we can reason similarly as in Appendix \ref{sect:proof-presentation} to show that the   generators given in Figure \ref{fig:generators} together with the following diagrams 
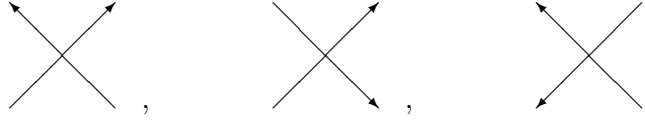
\[
\setlength{\unitlength}{.7mm}
\begin{picture}(140, 25)(-15, 0)
\put(0,0){\vector(1,1){20}}
\put(20,0){\vector(-1,1){20}}
\put(25, 0){,}

\put(50,0){\vector(1,1){20}}
\put(50,20){\vector(1,-1){20}}
\put(75, 0){,}

\put(120,20){\vector(-1,-1){20}}
\put(120,0){\vector(-1,1){20}}

\end{picture}
\]
generate all oriented Brauer diagrams. The first diagram above can be expressed in terms of generators in Figure \ref{fig:generators} as one of the diagrams in part (3)(d) of the theorem. The other two diagrams are respectively the top half and bottom half of the left most diagram in part (3)(e).  This shows that the generators in Figure \ref{fig:generators}  indeed suffice to generate all the oriented Brauer diagrams.

We prove part (3) by regarding $\OB(\delta)$ as a quotient category of the category of framed tangles following Remark \ref{rem:or-tang}.  Take the presentation of the latter category given in \cite[Theorem XII.2.2]{Ka} and impose the relevant conditions to obtain the quotient category. Then the relations \cite[Theorem XII.2.2]{Ka}  reduce to our relations (a)--(f). The relation (g) is one of the conditions imposed in taking the quotient. 
\end{proof}
\end{theorem}

\begin{remark}
One can also prove Theorem \ref{thm:tensor-cat-or}(3) directly as in Appendix \ref{sect:proof-presentation}. It is instructive to write down such an algebraic proof. 
\end{remark}

The proof of following lemma is straightforward. 
\begin{lemma}\label{lem:walled-Br}
Fix $k, \ell\in\N$, and let $(+)^k=(\underbrace{+, \dots, +}_k)$, $(-)^\ell=(\underbrace{-, \dots, -}_\ell)$ and $\eta=((+)^k, (-)^\ell)$. Then  as associative algebras,
$OB_{(+)^k}^{(+)^k}(\delta)\cong\K\Sym_k$, $
OB_{(-)^\ell}^{(-)^\ell}(\delta)\cong\K\Sym_\ell$, and 
$OB_{k, \ell}(\delta):=OB_\eta^\eta(\delta)$ is isomorphic to the walled Brauer algebra of type $(k, \ell)$ with parameter $\delta$.
\end{lemma}
The walled Brauer algebra has been much studied in the literature (see \cite{CW, BS} and the references there). 
We will not present its standard definition here; instead, we take the third isomorphism in the lemma as the definition. 

Given any $\eta\in\CN$, let $A_\eta$ and $U_\eta$ be the oriented Brauer diagrams of shape $A_q$ and $U_q$ in Figure \ref{Omega-U}
  respectively with $s(A_\eta)=(\eta, op(\eta^\vee))$, $t(A_\eta)=\emptyset$, and 
$s(U_\eta)=\emptyset$, $t(U_\eta)=(op(\eta^\vee), \eta)$. Let $I_\eta$ and $I_{\eta^\vee}$ be oriented Brauer diagrams of the shape $I_q$ in Figure \ref{Omega-U} with $s(I_\eta)=t(I_\eta)=\eta$ and $s(I_{\eta^\vee})=t(I_{\eta^\vee})=\eta^\vee$ respectively.  The following oriented analogue of Lemma \ref{isomorphism-pri}, 
is easily verified. 
\begin{lemma}\label{lem:iso-or} For any $\eta\in\CN$, 
\[
(A_\eta\ot I_\eta)(I_\eta\ot U_\eta)= I_\eta, \quad  (I_{op(\eta^\vee)}\ot A_\eta)(U_\eta\ot  I_{op(\eta^\vee)})= I_{op(\eta^\vee)}. 
\]
\end{lemma}
Using this lemma, we  prove the following analogue of Lemma \ref{isomorphism}.
\begin{lemma} \label{lem:iso-OB}
Given any $\eta, \zeta\in\CN$, write $N=\l(\eta)+l(\zeta)$. Then for any integer $k$ such that  $0\le k\le N$, there exist the following $\K$-module isomorphisms.
\[
 OB_\eta^\zeta(\delta)\cong OB_{(+)^{\frac{N}{2}}}^{(+)^{\frac{N}{2}}}(\delta)\cong OB_{(+)^{\frac{N}{2}}(-)^{\frac{N}{2}}}^{\emptyset}(\delta). 
 \]
 Furthermore,  $OB_\eta^\eta(\delta)$ is isomorphic to the walled Brauer algebra $OB_{r, s}(\delta)$ of type $(r, s)=(\#_+(\eta), \#_-(\eta))$ as $\K$-algebra.
\begin{proof}
The first statement is deduced from Lemma \ref{lem:iso-or} by reasoning similar to that required for Corollary \ref{isomorphism}, taking into
account Lemma \ref{lem:numbers}.

Now we prove the last statement. Let $\bar\eta=((+)^{\#_+(\eta)}, (-)^{\#_-(\eta)})$. Then there exist  unique  oriented Brauer diagrams
 $X_\eta^{\bar\eta}$ from $\eta$ to $ {\bar\eta}$ and $X_{\bar\eta}^\eta$ from $ {\bar\eta}$ to $\eta$, which have vertical arcs only and have the smallest numbers of crossings.  They satisfy 
\[
X_{\bar\eta}^\eta  X_\eta^{\bar\eta} =I_\eta, \quad X_\eta^{\bar\eta} X_{\bar\eta}^\eta =I_{\bar\eta}.
\]
The isomorphism $OB_\eta^\eta(\delta)\longrightarrow OB_{\bar\eta}^{\bar\eta}(\delta)$ is given by $D\mapsto  X_\eta^{\bar\eta} D X_{\bar\eta}^\eta$. 
\end{proof}
\end{lemma}

\part{Invariant theory of classical groups and supergroups}

\section{Tensor functors from Brauer categories to representation categories}

In this section, we construct functors from the Brauer categories to certain representation categories of classical groups and supergroups. 

Let $\K$ be a field. A vector superspace $V$ over $\K$ is a $\Z_2$-graded vector space $V=V_{\bar0}\oplus V_{\bar1}$, where $V_{\bar0}$ and $V_{\bar1}$ are the even and odd subspaces respectively. The dimension of $V$ is $\dim{V}=\dim{V_{\bar0}}+\dim{V_{\bar1}}$, and the superdimension is $\sdim{V}=\dim{V_{\bar0}} - \dim{V_{\bar1}}$. Denote by $[v]\in\Z_2$ the degree of a homogeneous element $v$. 

Assume that $d=\dim{V}<\infty$. 
Choose a homogeneous basis $\{b_i\mid1\le i\le d\}$ for $V$, and let $\{b^*_i\mid 1\le i \le d\}$ be the dual basis for $V^*$, that is, $b^*_i(b_j)=\delta_{i j}$ for all $i, j$.   Let 
\[
C_0=\sum_{i=1}^d b_i\otimes b_i^*. 
\]
This is canonical in the sense that it is independent of the basis chosen. 

We denote by $\GL(V)$ the general linear supergroup on $V$ defined as an affine group scheme over the category of vector superspaces.

We may assume that
the vector superspace $V$ is equipped with  a non-degenerate bilinear form $( - , -)$, 
which is homogeneous of degree $0$ and is orthosymplectic, i.e., 
\begin{eqnarray}\label{eq:orthsym-form}
\begin{aligned}
(i). &  \text{ $(V_{\bar0}, V_{\bar1})=(V_{\bar1}, V_{\bar0})=0$}; \\
(ii). &  \text{ $(v, v')=(-1)^{\alpha} (v' v)$ for all $v, v'\in V_{\alpha}$, $\alpha=0, 1$}. 
\end{aligned}
\end{eqnarray}

Let $G$ be the affine group scheme preserves the non-degenerate orthosymplectic form. This is the supergroup scheme $\OSp(V)$ of $\GL(V)$, known as the orthosymplectic supergroup.   

\begin{remark}\label{rem:super-classical}
If $V$ is purely even or purely odd, the general linear supergroup $\GL(V)$ becomes the ordinary general linear group, and $\OSp(V)$ reduces to the orthogonal group $\Or(V_{\bar 0})$ if $V=V_{\bar0}$, and to the symplectic group $\Sp(V_{\bar1})$ if $V=V_{\bar1}$. 
\end{remark}

In view of the remark, results proved for the supergroups in this section apply to the classical groups in the special cases when $V$ is purely even or purely odd. In particulra, we recover the classical FFTs and SFTs for classical groups from the FFTs and SFTs for $\GL(V)$ and $\OSp(V)$ given in Theorem \ref{thm:fts-GL} and Theorem \ref{thm:fft-sft} respectively. 

\subsection{The  functor from $\CB$ to the category of $\OSp(V)$ modules}\label{sect:functor}

Assume that the vector superspace $V$ has a non-degenerated orthosymplectic form satisfying \eqref{eq:orthsym-form}. 
The non-degeneracy of the form requires in particular $\dim{V_{\bar1}}$ be even.  
The form enables us to identify  $V$ with the dual space $V^*$ via the map $\phi: V\longrightarrow V^*$, $v\mapsto \phi_v$, where $\phi_v(x):=(v,x)$ for all $x\in V$. 

For any positive integer $t$, the vector superspace $V^{\otimes t}$
is a $G$-module in the usual way: $g(v_1\otimes\dots\otimes v_t)
= g v_1\otimes gv_2\otimes\dots\otimes gv_t$.
Moreover the form on $V$ induces a non-degenerate bilinear form 
\begin{align}\label{eq:multiform}
(( \ , \ )):  V^{\otimes t}\times V^{\otimes t} \longrightarrow \K, \quad ((v_1\otimes\dots\otimes v_t, w_t\otimes\dots\otimes w_1)) =\prod_{i=1}^t (v_i, w_i). 
\end{align}
This permits the identification of $V^{\otimes t}$ with its dual space
${V^{\otimes t}}^*=\Hom_\K(V^{\otimes t}, \K)$. Furthermore, for any homomorphism $\Phi: V^{\otimes t} \longrightarrow V^{\otimes s}$, there exists a corresponding map 
$\Phi^*: V^{\otimes s} \longrightarrow V^{\otimes t}$, which is uniquely defined by 
\begin{align}\label{eq:adjoint-map}
((\Phi(v) , w))=((v, A^*(w))), \quad \forall v\in V^{\otimes t}, \ w\in V^{\otimes s}. 
\end{align}

Let  ${\bar b}_i=\varphi^{-1}(b^*_i)$ for all $i$. Then $\{{\bar b}_i \mid 1\le i\le d\}$ is also a basis of $V$, which staisfies $( {\bar b}_i,  b_j)=\delta_{ij}$.
Define $c_0:=(\id_V\otimes\varphi^{-1})(C_0)\in V\ot V$, which can be expressed as $c_0=\sum_{i=1}^m b_i\otimes {\bar b}_i$. Then $c_0$
is canonical in that it is independent of the basis, and is invariant under $G$.
We  consider the following $G$-equivariant maps.
\begin{eqnarray}\label{P-C-C}
\begin{aligned}
&&P: V\otimes V\longrightarrow V\otimes V, &\quad v\otimes w \mapsto (-1)^{[v][w]} w\otimes v, \\
&&\check{C}: \K \longrightarrow V\otimes V,  &\quad 1\mapsto c_0, \\
&&\hat{C}: V\otimes V \longrightarrow \K, &\quad v\otimes w\mapsto (v, w).
\end{aligned}
\end{eqnarray}

They have the following properties.

\begin{lemma} \label{lem:PAU}
Let $G=\OSp(V)$. Denote the identity map on $V$ by $\id$.
\begin{enumerate}
\item The element $c_0$ belongs to
$(V\otimes V)^G$ and satisfies $P(c_0)=\epsilon c_0$.

\item The maps $P$, $\check{C}$ and $\hat{C}$ are all $G$-equivariant,
and
\begin{eqnarray}
&P^2=\id^{\ot 2}, \quad
(P\ot \id)(\id\ot P)(P\ot \id) = (\id\ot P)(P\ot\id)(\id\ot P),
\label{eq:PPP}\\
&P \check{C} =  \check{C}, \quad
\hat{C} P =  \hat{C}, \label{eq:fse-es}\\
&\hat{C}\check{C}=\sdim V,  \quad (\hat{C}\ot\id)(\id\ot\check{C})=\id=(\id\ot\hat{C})(\check{C}\ot\id), \label{eq:CC}\\
&(\hat{C}\ot\id)\circ (\id\ot P)= (\id\ot\hat{C})\circ (P\ot\id), \label{eq:CPC-I}\\
&(P\ot \id)\circ(\id\ot \check{C})=(\id\ot P)\circ(\check{C}\ot \id). \label{eq:CPC-P}
\end{eqnarray}
\end{enumerate}
\end{lemma}
\begin{proof}
Equation \eqref{eq:PPP} reflects standard properties of permutations,
and the relations \eqref{eq:fse-es} are evident.
We prove the other relations.
Consider for example $\hat{C}\check{C}=\hat{C}(\sum_i b_i\ot\bar{b}_i)=\sum_i (b_i, \bar{b}_i)$. The far right hand side
is $\sum_i (-1)^{[b_i]}(\bar{b}_i, b_i) =\sum_i (-1)^{[b_i]}= \sdim V$. This proves the first relation of \eqref{eq:CC}.
The proofs of the remaining relations are similar, and therfore omitted.
\end{proof}

\begin{definition}
We denote by $\CT_G(V)$ the full subcategory of $G$-modules with objects $V^{\otimes r}$ ($r=0, 1, \dots$),
where $V^{\otimes 0}=\K$ by convention. The usual tensor product of $G$-modules and
of $G$-equivariant maps is a bi-functor $\CT_G(V)\times \CT_G(V)\longrightarrow \CT_G(V)$, which will be
called the tensor product of the category. We call $\CT_G(V)$ the {\em category of
tensor representations of $G$}.
\end{definition}
Note that $\Hom_G(V^{\otimes r}, V^{\otimes t})=0$ unless $r+t$ is even. The zero module is not an object of $\CT_G(V)$,
thus the category is only pre-additive but not additive.

\begin{remark}
The category $\CT_G(V)$ is also a strict monoidal category with a symmetric braiding in the sense of \cite{JS}, where the braiding is given by the permutation maps $V^{\otimes r}\otimes V^{\otimes t}\longrightarrow  V^{\otimes t}\otimes V^{\otimes r}$, $v\ot w\mapsto w\ot v$.
\end{remark}

We have the following result.
\begin{theorem}\label{thm:functor}
Let $\delta_V=\sdim{V}$, and denote $G=\OSp(V)$. 
There is a unique additive covariant functor $F: \CB(\delta_V) \longrightarrow \CT_G(V)$
of pre-additive categories with the following properties:
\begin{enumerate}
\item[(i)] $F$ sends the object
$r$ to $V^{\otimes r}$ and morphism $D: k \to \ell$ to
$F(D): V^{\otimes k}\longrightarrow V^{\otimes l}$ where $F(D)$
is defined on the generators of Brauer diagrams by
\begin{eqnarray}\label{eq:F-generating}
\begin{aligned}
F\left(
\begin{picture}(30, 20)(0,0)
\put(15, -15){\line(0, 1){35}}
\end{picture}\right)=\id_V,
\quad&
F\left(
\begin{picture}(30, 20)(0,0)
\qbezier(5, -15)(15, 3)(25, 20)
\qbezier(5, 20)(15, 3)(25, -15)
\end{picture}\right) = P, \\
F\left(
\begin{picture}(30, 20)(0,0)
\qbezier(5, 20)(15, -50)(25, 20)
\end{picture}\right) = \check{C}, \quad&
F\left(
\begin{picture}(30, 20)(0,0)
\qbezier(5, -15)(15, 50)(25, -15)
\end{picture}\right) = \hat{C};
\end{aligned}
\end{eqnarray}
\item[(ii)] $F$ respects tensor products, so that for any
objects $r, r'$ and morphisms $D, D'$ in $\CB(\delta_V)$,
\[
\begin{aligned}
&F(r\otimes r')=V^{\otimes r}\otimes V^{\otimes r'}=F(r)\ot F(r'), \\
&F(D\otimes D')= F(D)\otimes F(D').
\end{aligned}
\]
\end{enumerate}
\end{theorem}
\begin{proof}
We want to show that the functor $F$
is uniquely defined, and gives rise
to an additive covariant functor from $\CB(\delta_V)$ to $\CT_G(V)$.

By Lemma \ref{lem:PAU}, the linear maps in \eqref{eq:F-generating}
are all $G$-module maps, and by Theorem \ref{thm:presentation}(1),
the above requirements define $F$ on all objects of $\CB(\delta_V)$;
it is clear that $F$ respects tensor products of objects.
As a covariant functor, $F$ preseves
composition of Brauer diagrams, and by (ii) $F$ respects
tensor products of morphisms. It remains only to show that
$F$ is well-defined.

To prove this, we need to show that the images of the generators satisfy
the relations in Theorem \ref{thm:presentation}(\ref{relations}). This is precisely
the content of equations
\eqref{eq:CC}-\eqref{eq:CPC-P} in Lemma \ref{lem:PAU}(2).

Hence for any morphism $D$ in $\CB(\delta_V)$,
$F(D)$ is indeed a well defined morphism in $\CT_G(V)$.
\end{proof}

\begin{remark}
The functor $F$ is a tensor functor between braided strict monoidal categories.
\end{remark}

\begin{lemma}\label{lem:transfrom}
Let $H_s^t = \Hom_G(V^{\otimes s}, V^{\otimes t})$ for all $s, t\in \N$.
\begin{enumerate}
\item The $\K$-linear maps
\[
\begin{aligned}
&F{\mathbb U}_p^q:=(-\otimes \id_V^{\otimes q})(\id_V^{\otimes p}\otimes F(U_q)):
H_{p+q}^r \longrightarrow H_p^{r+q}, \\
&F{\mathbb A}^r_q:=(\id_V^{\otimes r}\otimes F(A_q))(- \otimes \id_V^{\otimes q}):
H_p^{r+q} \longrightarrow H_{p+q}^r
\end{aligned}
\]
are well defined and are mutually inverse isomorphisms.
\item For each pair $k, \ell$ of objects in $\CB(\delta_V)$, the functor $F$ induces a linear map
\begin{eqnarray}\label{eq:functionF}
\begin{aligned}
{F}_k^\ell: B_k^\ell(\delta_V)\longrightarrow H_k^\ell=\Hom_G(V^{\otimes k}, V^{\otimes \ell}), \quad
D \mapsto F(D),
\end{aligned}
\end{eqnarray}
and the following diagrams are commutative.
\begin{displaymath}
    \xymatrix{
        B_p^{r+q}(\delta_V) \ar[r]^{{\mathbb A}^r_q} \ar[d]_{F_p^{r+q}} &  B_{p+q}^r(\delta_V) \ar[d]^{F_{p+q}^r} \\
         H_p^{r+q}  \ar[r]_{F{\mathbb A}^r_q}  & H_{p+q}^r}
\quad\quad
\xymatrix{
         B_{p+q}^r(\delta_V) \ar[r]^{{\mathbb U}_p^q} \ar[d]_{F_{p+q}^r} &  B_p^{r+q}(\delta_V)  \ar[d]^{F_p^{r+q}} \\
         H_{p+q}^r \ar[r]_{F{\mathbb U}_p^q}  & H_p^{r+q}. }
\end{displaymath}
\end{enumerate}
\end{lemma}
\begin{proof}
Part (1) follows by applying the functor $F$ to Corollary \ref{isomorphism}, using
Theorem \ref{thm:functor}.

Now for any $D\in B_p^{r+q}(\delta_V)$,
${\mathbb A}^r_q(D)=(I_{r+q}\otimes A_q)\circ(D \otimes I_q)$. Since
$F$ preserves both composition and tensor product of Brauer diagrams,
\[
\begin{aligned}
F({\mathbb A}^r_q(D)) &= (\id_V^{\otimes (r+q)}\otimes F(A_q))(F(D) \otimes \id_V^{\otimes q})\\
                    &=F{\mathbb A}^r_q(F(D)).
\end{aligned}
\]
This proves the commutativity of the first diagram in part (2). The commutativity of the other diagram
is proved in the same way.
\end{proof}

The following result also holds. 
\begin{lemma}\label{lem:F-involution}
Let $^*: B_t^s(\delta_V) \longrightarrow B_s^t(\delta_V)$ be the anti-involution defined in Section \ref{ssec:involutions}.
Then for any $D\in  B_t^s(\delta_V)$, we have $F(D^*)= F(D)^*$, where 
$ F(D)^*: V^{\ot s} \longrightarrow V^{\ot t}$ is defined by \eqref{eq:adjoint-map}. 
\end{lemma}
\begin{proof}
Given part (1) of Theorem \ref{thm:fft-sft}, which states that $\Hom_G(V^{\ot s}, V^{\ot t})=F(B_s^t(\delta_V))$,  
the lemma easily follows from Theorem \ref{thm:presentation} and Theorem \ref{thm:functor}.  
We note that the proof of Theorem \ref{thm:fft-sft} (1) in \cite{DLZ, LZ6} does not involve the anti-involutions.
\end{proof}

\subsection{The functor from $\OB$ to the category of $\GL(V)$ modules}\label{sect:functor-GL}

Given a vector superspace $V$ of dimension $d=\dim{V_{\bar0}}+\dim{V_{\bar1}}<\infty$, denote by $V^*$ its dual space.  
Let $\{b_i\mid 1\le i\le d\}$ be a homogeneous basis of $V$ and $\{b^*_i\mid 1\le i\le d\}$ be the dual basis of $V^*$. 

Denote by $\GL(V)$ the general linear supergroup on $V$ defined as an algebraic group. 
It naturally acts on $V$, and also acts on $V^*$ by 
$(g.\overline{w})(v)=\overline{w})(g^{-1}.v)$ for all $v\in V$, $\overline{w}\in V^*$ and $g\in\GL(V)$. The actions extend 
to any repeated tensor products of $V$ and $V^*$ by 
$g(v_1\otimes\dots\otimes v_t) = g v_1\otimes gv_2\otimes\dots\otimes gv_t$, where $v_i$ belongs to $V$ and $V^*$, and $g\in\GL(V)$.

As before, we set $C_0=\sum_{i=1}^d b_i\otimes b^*_i$,  which is independent of the bases for $V$ and $V^*$.  
Note that $C_0$ is an $\GL(V)$-invariant in $V\otimes V^*$. 

Denote the identity maps on $V$ and $V^*$ by $\id_V$ and $\id_{V^*}$ respectively. For convenience, we write $V^+=V$ and $V^-=V^*$. 

Define the following linear maps,  
\begin{align}
 &P^{\varepsilon\varepsilon'}: V^\varepsilon\otimes V^{\varepsilon'}\longrightarrow V^{\varepsilon'}\otimes V^\varepsilon, \quad v\otimes w \mapsto (-1)^{[v][w]} w\otimes v, \\
 & \check{C}^{+-}: \K\longrightarrow V\otimes V^*,  \quad 1\mapsto C_0=\sum_{i=1}^d b_i\otimes b^*_i, \\
& \check{C}^{-+}: \K\longrightarrow V^*\otimes V,  \quad 1\mapsto P^{+-}(C_0)=\sum_{i=1}^d (-1)^{[b_i]}b^*_i\otimes b_i, \\
&\hat{C}_{-+}: V^*\otimes V\longrightarrow\K, \quad \overline{w}\otimes  v\mapsto \overline{w}(v),   \\
& \hat{C}_{+-}: V\otimes V^*\longrightarrow\K, \quad v \otimes \overline{w}\mapsto (-1)^{[v][\overline{w}]} \overline{w}(v), 
\end{align}
where $\varepsilon,  \varepsilon=\pm$. 
Let
\[
\begin{aligned}
\hat{C}_{-+}^{(2)}=\hat{C}_{-+}(\id_{V^*}\otimes\hat{C}_{-+}\otimes\id_V), \\
\hat{C}_{+-}^{(2)}=\hat{C}_{+-}(\id_V\otimes\hat{C}_{-+}\otimes\id_{V^*}), \\
\check{C}^{+-(2)}= (\id_V\otimes\check{C}^{+-}\otimes \id_{V^*})\check{C}^{+-}, \\
\check{C}^{-+(2)}= (\id_{V^*}\otimes\check{C}^{-+}\otimes \id_V)\check{C}^{-+}.
\end{aligned}
\]

Note that the above maps are all $\GL(V)$-equivariant. 

\begin{lemma}\label{lem:P-C-GL}
Write $P=P^{++}$. 
 The following relations hold. 
\begin{eqnarray} 
&P^{\varepsilon'\varepsilon} P^{\varepsilon\varepsilon'}=\id_{V^\varepsilon}\otimes \id_{V^{\varepsilon'}}, \label{eq:PP-GL}\\
&(P\ot \id_V)(\id_V\ot P)(P\ot \id_V) = (\id_V\ot P)(P\ot\id_V)(\id_V\ot P),
\label{eq:PPP-GL}\\
&P^{--}=(\hat{C}_{-+}^{(2)}\otimes \id_{V^*}^{\otimes 2})(\id_{V^*}^{\otimes2}\otimes P\otimes \id_{V^*}^{\otimes2} )(\id_{V^*}^{\otimes2}\otimes\check{C}^{+-(2)}),  \label{eq:PC-1}\\
&\phantom{P^{---}}= (\id_{V^*}^{\otimes 2}\otimes \hat{C}_{+-}^{(2)})(\id_{V^*}^{\otimes 2}\otimes P\otimes\id_{V^*}^{\otimes 2})
(\check{C}^{-+(2)}\otimes\id_{V^*}^{\otimes2}),  \label{eq:PC-2} \\
&P^{-+}=(\hat{C}_{-+}\ot\id_V\ot\id_{V^*})(\id_V\ot P\ot\id_{V^*})(\id_{V^*}\ot \id_V\ot\check{C}^{+-}),  \label{eq:PC-3}\\
&P^{+-}=(\id_{V^*}\ot\id_V\ot\hat{C}_{-+})(\id_V\ot\ot P\ot\id_V)(\check{C}^{+-}\ot\id_{V^*}\ot\id_V),   \label{eq:PC-4}\\
&\hat{C}_{+-}\check{C}^{+-}=\sdim V = \hat{C}_{-+}\check{C}^{-+},   \label{eq:CC-GL}\\
&(\hat{C}_{+-}\otimes\id_V)(\id_V\otimes \check{C}^{-+}) = \id_V = (\id_V\otimes\hat{C}_{-+})(\check{C}^{+-}\otimes\id_V), \label{eq:slide-GL+}\\
& (\hat{C}^{-+}\ot\id_{V^*})(\id_{V^*}\ot\check{C}^{+-}) =\id_{V^*} =(\id_{V^*}\ot \hat{C}^{+-})(\check{C}^{-+}\ot \id_{V^*}), \label{eq:slide-GL-}\\
&(\id_V\ot\hat{C}_{+-})(P\ot \id_{V^*})(\id_V\ot\check{C}^{+-})=\id_V.  \label{eq:twist-GL}
\end{eqnarray}

\begin{proof}
Relations \eqref{eq:PP-GL} and \eqref{eq:PPP-GL} are well known, and easy to prove.  To prove \eqref{eq:PC-1}, we take any $\overline{v}, \overline{w}\in V^*$ and apply the right hand side of \eqref{eq:PC-1}
to $\overline{v}\otimes \overline{w}$. We obtain 
\[
\begin{aligned}
&(\hat{C}_{-+}^{(2)}\otimes \id_{V^*}^{\otimes 2})(\id_{V^*}^{\otimes2}\otimes P\otimes \id_{V^*}^{\otimes2} )(\overline{v}\otimes \overline{w}\otimes\check{C}^{+-(2)}(1))\\
&= \sum_{i, j=1}^d (-1)^{[b_i][b_j]} \hat{C}_{-+}^{(2)}(\overline{v}\otimes \overline{w}\otimes b_j\ot b_i)\ot b_i^*\ot b_j^*)=(-1)^{[\overline{v}][\overline{w}]}\overline{w}\ot \overline{v}. 
\end{aligned}
\]
The last expression is clearly equal to $P^{--}(\overline{v}\otimes \overline{w})$, proving \eqref{eq:PC-1}. 
The proofs of the remaining relations are similar, and thus omitted.
\end{proof}
\end{lemma}

Given any $\eta=(\varepsilon_1, \dots, \varepsilon_r)\in\CN$, we denote $V^\eta=V^{\varepsilon_1}\otimes V^{\varepsilon_2} \otimes\dots \otimes V^{\varepsilon_r}$, with $V^\emptyset=\K$ by convention. 

\begin{definition}
We denote by $\CT_\GL(V)$ the full subcategory of $\GL(V)$-modules with objects $V^\eta$ for all $\eta\in\CN$. 
The usual tensor product of $\GL(V)$-modules and
of $\GL(V)$-equivariant maps is a bi-functor $\CT_\GL(V)\times \CT_\GL(V)\longrightarrow \CT_\GL(V)$, which will be
called the tensor product of the category. We call $\CT_\GL(V)$ the {\em category of
tensor representations of $\GL(V)$}.
\end{definition}

\begin{theorem}\label{thm:functor-GL}
Let $\delta_V=\sdim{V}$. 
There is a unique additive covariant functor $F: \OB(\delta_V) \longrightarrow \CT_\GL(V)$
of pre-additive categories with the following properties:
\begin{enumerate}
\item[(i)] $F$ sends the object
$\eta\in\CN$ to $V^\eta$,  and morphism $D \in\Hom(\eta, \eta')$ to $F(D): V^\eta\longrightarrow V^{\eta'}$,  where $F(D)$
is defined on the generators of the oriented Brauer diagrams by
\begin{eqnarray}\label{eq:F-generating-GL}
\begin{aligned}
F\left(
\begin{picture}(30, 25)(0,0)
\put(15, 20){\vector(0, -1){35}}
\end{picture}\right)=\id_V,
\quad&  F\left(
\begin{picture}(30, 25)(0,0)
\put(15, -15){\vector(0, 1){35}}
\end{picture}\right)=\id_{V^*},\\
F\left(
\begin{picture}(30, 25)(0,0)
\qbezier(5, -15)(15, 3)(25, 20)
\put(22, -10){\vector(2, -3){5}}
\qbezier(5, 20)(15, 3)(25, -15)
\put(8, -10){\vector(-2, -3){5}}
\end{picture}\right) = P, \quad &\\
F\left(
\begin{picture}(30, 25)(0,0)
\qbezier(5, 20)(15, -50)(25, 20)
\put(24, 16){\vector(1, 4){2}}
\end{picture}\right) = \check{C}^{+-}, \quad& 
F\left(
\begin{picture}(30, 25)(0,0)
\qbezier(5, 20)(15, -50)(25, 20)
\put(6, 16){\vector(-1, 4){2}}
\end{picture}\right) = \check{C}^{-+},\\
F\left(
\begin{picture}(30, 25)(0,0)
\qbezier(5, -15)(15, 50)(25, -15)
\put(24, -10){\vector(1, -4){2}}
\end{picture}\right)= \hat{C}_{-+}, \quad& 
F\left(
\begin{picture}(30, 25)(0,0)
\qbezier(5, -15)(15, 50)(25, -15)
\put(6, -10){\vector(-1, -4){2}}
\end{picture}\right) = \hat{C}_{+-}; 
\end{aligned}
\end{eqnarray}
\item[(ii)] $F$ respects tensor products, so that for any
objects $\eta, \eta'$ and morphisms $D, D'$ in $\OB(\delta_V)$,
\[
\begin{aligned}
&F(\eta\otimes \eta')=V^{\eta}\otimes V^{\eta'}=F(\eta)\ot F(\eta'), \\
&F(D\otimes D')= F(D)\otimes F(D').
\end{aligned}
\]
\end{enumerate}
\begin{proof}  The proof is similar to that of Theorem \ref{thm:functor}, so we confine ourselves to some brief comments. The key point of the proof is again to show that $F$ as defined preserves the relations among the generators  of oriented Brauer diagrams. Consider, e.g., the relation of orientation reversing. Applying $F$ to the diagrams on the left and right sides of relation $(d)$ in Theorem \ref{thm:tensor-cat-or}(3),  we obtain the following maps respectively. 
\[
\begin{aligned}
&(\hat{C}_{-+}^{(2)}\otimes \id_{V^*}^{\otimes 2})(\id_{V^*}^{\otimes2}\otimes P\otimes \id_{V^*}^{\otimes2} )(\id_{V^*}^{\otimes2}\otimes\check{C}^{+-(2)}), 
\quad \text{and} \\
&(\id_{V^*}^{\otimes 2}\otimes \hat{C}_{+-}^{(2)})(\id_{V^*}^{\otimes 2}\otimes P\otimes\id_{V^*}^{\otimes 2})
(\check{C}^{-+(2)}\otimes\id_{V^*}^{\otimes2}).
\end{aligned}
\]
By relations \eqref{eq:PC-1} and \eqref{eq:PC-2}  in Lemma \ref{lem:P-C-GL}, both maps are equal to $P^{--}$. 
This proves the  orientation reversing relation. The other relations can be proved similarly by using Lemma \ref{lem:P-C-GL}. 
\end{proof}
\end{theorem}

\subsection{The  functor from $\wt\CB$ to the category of $\SO(V)$ modules}\label{sect:functor-SO}

We now return to the setting of Section \ref{sect:functor}, but assume that $V=V_{\bar0}$ is purely even of dimension $m$. Let $G=\SO(V)$, and denote by $\CT_G(V)$ the category of tensor modules of $G$. Now we choose an orthonormal basis $\{e_1, \dots, e_m\}$ for $V$, and let 
\[
\Lambda = e_1\wedge e_2\wedge\dots \wedge e_m:=\Sigma_m(e_1\ot e_2\ot \dots \ot e_m).
\]
Then  $g.\Lambda = \det(g)\Lambda=\Lambda$ for all $g\in G$. We still denote by $\Lambda$ the map $\C \longrightarrow V^{\ot m}$ such that $1\mapsto \Lambda$. 

\begin{theorem}[\cite{LZ8}]\label{thm:functor-SO} Let $G=\SO(V)$ with $\dim V=m$. There exists a tensor functor $F: \wt\CB(m)\longrightarrow \CT_G(V)$,  which coincides with the corresponding functor in Theorem \ref{thm:functor} on all objects and the morphisms in $\wt\CB_0(m)=\CB(m)$,  and 
\[
F(\Delta_m) = \Lambda, \quad F(\Delta^*_m)=\Lambda^*. 
\]
Furthermore, $F(D^*)= F(D)^*$  for any morphism $D\in \wt\CB(m)$. 
\begin{proof}
It is easy to see that  $\Lambda$ is harmonic and $\Lambda\circ \Lambda^*= F(\Sigma_m)$.  Furthermore, $F$ preserves relation (5) in Definition \ref{def:wbrcat}. Now the theorem follows from Theorem \ref{thm:functor}. 
\end{proof}
\end{theorem}


\section{Invariant theory of classical groups and supergroups}

We develop the invariant theory of the general linear supergroup $\GL(V)$ and the orthosymplectic supergroup $\OSp(V)$ using the Baruer categories, obtaining the first and second fundamental theorems (FFT and SFT) in a categorical formulation in Theorem \ref{thm:fts-GL} and Theorem \ref{thm:fft-sft}. 
By Remark \ref{rem:super-classical}, we recover the classical FFTs and SFTs for classical groups from Theorem \ref{thm:fts-GL} and Theorem \ref{thm:fft-sft}. 

The Brauer category provides useful methods to study structures of the endomorphism algebras of the  orthogonal and symplectic groups. In particular, we will obtain presentations of the endomorphism algebras in  Section \ref{sect:ker-O} and Section \ref{sect:ker-Sp}.

Henceforth we assume that $\K$ is a field of characteristic zero.

\subsection{Schur-Weyl duality}

In this section we state the two fundamental theorems of invariant theory for the general linear supergroup $\GL(V)$ in a
form convenient for use in our context.

Given any two $\Z_2$-graded $\K$-vector spaces,  $U ,W$, we may form
the tensor products $U\ot_\C W$ and $W\otimes U^*\cong \Hom_\C(U,W)$. These
are  $\Z_2$-graded in the usual way. 

Let $V$ be the finite dimensional vector superspace of Section \ref{sect:functor}, we let  $\GL(V)$ be the general linear supergroup, again defined as an algebraic group. 
Then $\GL(V)$  acts on the super spaces 
$V^{\ot r}$ for $r=1,2,\dots$.
The super-permutation $P:V\ot V\lr V\ot V$, given by
\be\label{eq:tau}
P(v\ot w)=(-1)^{[v][w]}w\ot v,
\ee
for homogeneous $v, w\in V$, and extended linearly.
Then $P\in\End_{\GL(V)}(V\ot V)=(\End_\K(V\ot V))^{\GL(V)}$. 

We have a homomorphism of $\C$-algebras
\be\label{eq:sym}
\varpi_r:\K\Sym_r\lr \End (V^{\ot r})^{\GL(V)},
\ee
in which the simple transpositions in $\Sym_r$ are mapped to the endomorphisms $\tau$ of \eqref{eq:tau}, acting on the appropriate factors
of the product.

The following results are the content of \cite[Theorems 3.3 and 3.7]{BR}. 

\begin{theorem}[\cite{BR}]\label{thm:fftgl}\label{thm:BR2}
\begin{enumerate}
\item
The map $\varpi_r$ of \eqref{eq:sym} is surjective for all $r$.
\item Assume that $(\dim{V_{\bar0}}, \dim{V_{\bar1}})=(m, \ell)$.
If $r< (m+1)(\ell+1)$, then $\varpi_r$ 
 is an isomorphism of superalgebras. If $r\ge (m+1)(\ell+1)$,
then the kernel of $\varpi_r$ is the (two-sided) ideal of $\K\Sym_r$ generated by the Young
symmetriser of the partition with $m+1$ rows and $\ell+1$ columns.
\end{enumerate}
\end{theorem}

The first and second parts of the theorem are respectively the first and  second fundamental theorems (FFT and SFT) for $\GL(V)$.

The SFT for $\GL(V)$ describes the kernel of the surjective homomorphism $\varpi_r$.
The kernel is generated by an idempotent which is explicitly described as follows. Consider the
$(m+1)\times(\ell+1)$ array of integers below, which form a standard tableau.

\smallskip

\begin{center}
\begin{tabular}{| l | l | p{3cm} |l |}
\hline
1  &  2  &  \dots  & $ \ell+1$\\ \hline
$\ell+2$  & $\ell+3$  & $\dots$ & $2\ell+2$\\ \hline
\dots      &   \dots       &    \dots  &  \dots   \\ \hline
\dots      &   \dots       &    \dots  &  \dots   \\ \hline
$m\ell+m+1$     &  $ m\ell+m+2$        &    \dots  &  $m\ell+m+\ell+1$    \\ \hline
\end{tabular}
\end{center}

\smallskip

\noindent Let $R$ and $C$ be the subgroups of $\Sym_{m\ell+m+\ell+1}$ (regarded as the subgroup of $\Sym_r$
which permutes the first $(m+1)(\ell+1)$ numbers) which stabilise the rows and columns of the array respectively.
Thus
$$
\begin{aligned}
R=\Sym\{1,2,&\dots,\ell+1\}\times\Sym\{ \ell+2  , \ell+3  , \dots , 2\ell+2\}\times\dots\\
&\dots\times\Sym\{m\ell+m+1,   m\ell+m+2,    \dots,  m\ell+m+\ell+1\},\\
\end{aligned}
$$
while
$$
\begin{aligned}
C=\Sym\{1,\ell+2,&\dots,m\ell+m+1\}
\times\Sym\{ 2  , \ell+3  , \dots , m\ell+m+2\}\times\dots\\
&\dots\times\Sym\{\ell+1,   2\ell+2,    \dots,  m\ell+m+\ell+1\},\\
\end{aligned}
$$
where $\Sym\{X\}$ denotes the group of permutations of the set $X$.

Then in the group ring $\K\Sym_{m\ell+m+\ell+1}\subseteq \K\Sym_r$, let $e=e(m,\ell)$ be the (even)
element defined by
\be\label{eq:eml}
e(m,\ell)=\left(\sum_{\pi\in R}\pi\right)\left(\sum_{\sigma\in C}\ve(\sigma)\sigma\right)=\alpha^+(R)\alpha^-(C),
\ee
where $\ve$ is the sign character of $\Sym_r$, and for any subset $H\subseteq \Sym_r$, we write $\alpha^+(H)$ (resp. $\alpha^-(H)$)
for the element $\sum_{h\in H}h$ (resp. $\sum_{h\in H}\ve(h)h$) of $\K\Sym_r$.

It is known that $(|R|!|C|!)\inv e(m,\ell)$ is a primitive idempotent in $\K\Sym_{m\ell+m+\ell+1}$.
It is also well known that $\K\Sym_r=\oplus_{\mu}I(\mu)$, where $\mu$ runs over the partitions of $r$,
and $I(\mu)$ is a simple ideal of $\K\Sym_r$ for each $\mu$. In this notation, the ideal $I(m,\ell)$ of $\K\Sym_r$
which is generated by $e(m,\ell)$ is the sum of the $I(\mu)$ over those partitions $\mu$ which
contain an $(m+1)\times(\ell+1)$ rectangle.

\begin{corollary}\label{cor:sftgl}
If $r< (m+1)(\ell+1)$ then $\Ker(\varpi_r)=0$. Otherwise,
$\Ker(\varpi_r)=I_r(m,\ell):=\oplus_\mu I(\mu)$ over those partitions $\mu$ of $r$ which contain
a rectangle of size $(m+1)\times(\ell+1)$.
\end{corollary}

\subsection{The fundamental theorems of invariant theory for $\GL(V)$}

Assume that  the vector superspace $V$ has $(\dim{V_{\bar0}}, \dim{V_{\bar1}}) =(m, \ell)$ respectively. Consider the oriented Brauer category $\OB(\delta_V)$ with parameter $\delta_V=\sdim{V}=m-\ell$. 
For any $r\ge r_c=(m+1)(\ell+1)$, the group algebra $\K\Sym_{r_c}$ is contained as a subalgebra in $OB_\eta^\eta(\delta_V)$ with $\eta=(\underbrace{+, +, \dots, +}_r)$ by Lemma \ref{lem:walled-Br}.  Hence $I_r(m, \ell)\subseteq OB_\eta^\eta(\delta_V)$. 

Introduce the following tensor ideal of $\OB(\delta_V)$. 
\begin{definition}\label{def:ideal-OB} Let $\CJ(m, \ell)$ be the subspace of $\oplus_{\eta, \zeta\in\CN}OB_\eta^\zeta(\delta_V)$
spanned by the morphisms in $\OB(\delta_V)$ generated by $e(m, \ell)$ by composition and tensor product.
Set $\CJ(m, \ell)_\eta^\zeta =\CJ(m, 2n)\cap OB_\eta^\zeta(\delta_V)$ for any $\eta, \zeta\in\CN$. 
\end{definition}

We have the following result. 

\begin{theorem} \label{thm:fts-GL} Assume that $\K$ has characteristic $0$.  Let $\GL(V)$ be the general linear supergroup on the vector superspace $V$. Assume that 
$(\dim{V_{\bar0}}, \dim{V_{\bar1}}) =(m, \ell)$. 
\begin{enumerate}
\item  The functor $F: \OB(\delta_V)\longrightarrow \CT_\GL(V)$ is full.
\item For any $\eta, \theta\in\CN$, denote by $F_\eta^\theta: OB_\eta^\theta(\delta_V)\longrightarrow \Hom_{\GL(V)}(V^\eta, V^\theta)$ the linear map given by the restriction of the functor $F$ to the space of homomorphisms. Then the kernel of $F_\eta^\theta$ is equal to $\CJ(m, \ell)_\eta^\theta$. 
\end{enumerate}
\end{theorem}

The functor $F$ maps the isomorphisms in Lemma \ref{lem:iso-OB} to isomorphisms among homomorphism spaces of $\GL(V)$-modules. Therefore, we have the following result. 

\begin{lemma}
For any $\eta\in\CN$, the two-sided ideal  $\CJ(m, \ell)_\eta^\eta$ in $OB_\eta^\eta(\delta_V)$ is isomorphic to $\CJ(m, \ell)_{\bar{\eta}}^{\bar{\eta}}$ as associative algebra, where $\bar{\eta}=((+)^{\#_+(\eta)}, (-)^{\#_-(\eta)})$. Furthermore, 
$\CJ(m, \ell)_\eta^\eta\cong \CJ(m, \ell)_{(+)^{l(\eta)}}^{(+)^{l(\eta)}}\cong I_{l(\eta)}(m, \ell)$ as vector space. In particular, $\CJ(m, \ell)_\eta^\eta=0$ if $l(\eta)<(m+1)(\ell+1)$. 
\end{lemma}

\subsection{The fundamental theorems of invariant theory for $\OSp(V)$}
Denote $G=\OSp(V)$, and assume that $(\dim{V_{\bar0}}, \dim{V_{\bar1}})=(m, 2n)$. 
We now return to the category $\CB(\delta_V)$ of Brauer diagrams with parameter $\delta_V=\sdim{V}$ and the covariant functor $F: \CB(\delta_V)\longrightarrow \CT_G(V)$.
Recall that the group algebra $\K\Sym_r$ is embedded in the Brauer algebra $B_r(\delta_V)$
of degree $r$. In particular, $I_r(m,2n)\subseteq B_r^r(\delta_V)$, and hence  
the idempotent $e(m, 2n)$ defined by \eqref{eq:eml} belongs to $B_{r_0}^{r_0}(\delta_V)$, where $r_0=(m+1)(2n+1)$. 

Introduce the following tensor ideal of $\CB(\delta_V)$. 
\begin{definition}\label{def:ideal-B} 
Denote by $\CJ(m, 2n)$ the subspace of $\oplus_{k,\ell}B_k^\ell(\delta_V)$
spanned by the morphisms
in $\CB(\delta_V)$ generated by $e(m, 2n)$ by composition and tensor product.
Set $\CJ(m, 2n)_k^\ell =\CJ(m, 2n)\cap B_k^\ell(\delta_V)$.
\end{definition}

The first and second fundamental theorems of classical invariant theory for
the orthosymplectic supergroup are respectively given by  part (1) and part (2) of the following theorem.
\begin{theorem} \label{thm:fft-sft} Assume that $\K$ is a field of characteristic $0$. Let $G=\OSp(V)$, where the  vector superspace
$V$ has $(\dim{V_{\bar0}}, \dim{V_{\bar1}})=(m, 2n)$.  
\begin{enumerate}
\item  The functor $F: \CB(\delta_V)\longrightarrow \CT_G(V)$ is full.
That is, $F$ is surjective on $\Hom$ spaces.
\item The kernel of the map ${F}_k^\ell$ is given by 
$\Ker{F}_k^\ell=\CJ(m, 2n)_k^\ell$ for all $k, \ell$.
\end{enumerate}
\end{theorem}
\begin{proof}
Part (1) of the theorem is proved in \cite{DLZ, LZ6} by algebraic geometric means. What shown in op. cit. is that $F_r^0: B_r^0(\delta_V)\longrightarrow \Hom_G(V^{\otimes r}, \K)$ is surjective for all $r$ (the odd $r$ case is trivial as $B_{2k+1}^0(\delta_V)=0=\Hom_G(V^{\otimes 2k+1}, \K)$).  By Lemma \ref{lem:transfrom}, this implies that $F_r^s: B_r^s(\delta_V)\longrightarrow \Hom_G(V^{\otimes r}, V^{\otimes r})$ is surjective for all $r, s$.  Part (2)  is proved in \cite{LZ9} by using the Brauer category to reduce the problem to the second fundamental theorem of $\GL(V)$. 
\end{proof}

The following result is proved in \cite{Zy}.
\begin{lemma}[\cite{Zy}]\label{lem:ker} Retain the notation above, and set $r_c= (m+1)(n+1)$. Then 
$\CJ(m, 2n)_k^\ell\ne 0$ if and only if $k+\ell\ge r_c$.
\end{lemma}

\begin{remark} The orthosymplectic group reduces to the orthogonal group $\Or(V_{\bar0})$ if $n=0$, and to the symplectic group $\Sp(V_{\bar1})$ if $m=0$.  In these special cases, Theorem \ref{thm:fft-sft} recovers the first and second fundamental theorems of invariant theory for $\Or(V_{\bar0})$ and $\Sp(V_{\bar1})$.  
\end{remark}
\begin{remark}  
For any $\OSp(V)$, a description of  $Ker{F}_r^s$ as a vector superspace is given in \cite{Zy}. The results can be greatly sharpened in the special cases of the orthogonal and symplectic groups \cite{LZ5}.
A treatment of tensor ideals of Deligne categories was given in \cite{Ck}.
\end{remark}

\subsection{Equivalence between $\wt\CB(m)$ and the category of $\SO(V)$ tensor modules}\label{sect:equiv-SO}

Now we assume that $V$ is a purely even vector space of dimension $m$, which is equipped with a non-degenerate symmetric bilinear form.  Let $\SO(V)$ be the special orthogonal group on $V$. 

\begin{theorem}[\cite{LZ8}] \label{thm:main-SO} Let $G=\SO(V)$ with $\dim V=m$, and let $\CT_G(V)$ be the category of tensor $G$-modules.  Then the functor 
$F: \wt\CB(m)\lr\CT_G(V)$ given in Theorem \ref{thm:functor-SO} is an {\em equivalence} of categories
\end{theorem}

\begin{proof} The functor obviously restricts to  an isomorphism between the sets of objects of the categories, thus we 
 only need to show that $F$ defines isomorphisms on $\Hom$ spaces. But Theorem \ref{thm:fft-sft} (1) states precisely that $F$ is surjective on $\Hom$ spaces (the FFT). We are therefore reduced to proving the injectivity of $F$ on $\Hom$ spaces, which is the SFT for $\SO_m$. 

By Lemma \ref{lem:str}(2), each element $\beta\in\ker(F_s^t):\Hom_{\wt\CB(m)}(s,t)\lr\Hom_G(V^{\ot s},V^{\ot t})$
is uniquely of the form $\beta=\beta_0+\beta_1$, where $\beta_i\in\wt B_{s,i}^t$ ($i=0,1$). Moreover $F_s^t$ maps
$\wt B_{s,0}^t$ to $\Hom_{\Or(V)}(V^{\ot s},V^{\ot t})$, and $\wt B_{s,1}^t$ to the space of skew invariants 
for $\Or(V)$. It follows that $\beta\in\ker(F_s^t)$ if and only if $\beta_i\in\ker(F_s^t)$ for $i=0,1$.

Now Theorem \ref{thm:fft-sft} (2) states that the image of an element $\gamma$ of $\wt\CB(m)_0$ under $F$ is zero if and only if 
$\gamma$ is in the ideal $\langle\Sigma_{m+1}\rangle$
of morphisms generated under the operations of a tensor category by $\Sigma_{m+1}$.
{
This proves that $\beta_0$ is in $\langle\Sigma_{m+1}\rangle$. 

As for $\beta_1$, note that because of its form, we have 
$\beta_1\circ(\Delta_m^*\ot I^{\ot r})\in\wt\CB_0$ and hence lies in $\langle\Sigma_{m+1}\rangle$
 for some $r$, and so $\beta_1=(m!)\inv \left(\beta_1\circ(\Delta_m^*\ot I^{\ot r})\right)\circ(\Delta_m\ot I^{\ot r})$
 is also in the ideal $\langle\Sigma_{m+1}\rangle$. 
 Hence $\beta\in\langle\Sigma_{m+1}\rangle$.  
By Theorem \ref{thm:comp},  $\langle\Sigma_{m+1}\rangle$ is zero in $\wt\CB(m)$, and the proof is complete.
 }
\end{proof}

\subsection{Remarks concerning ${\rm SOSp}(V)$ invariants}
The invariant theory of the special orthosymplectic supergroup ${\rm SOSp}(V)$ was developed in \cite{LZ7} in a commutative superalgebra setting. An ${\rm SOSp}(V)$-invariant called the super Pfaffian was constructed, which together with the Brauer invariants of $\OSp(V)$ generates all the ${\rm SOSp}(V)$-invariants \cite[Theorem 5.2]{LZ7}.  

It should be possible to construct a generalisation of the enhanced Brauer category $\wt\CB$ above by replacing $\Delta_m$ with a generator mimicking properties of the super Pfaffian. This new enhanced Brauer category is expected to be equivalence to the category of tensor modules for ${\rm SOSp}(V)$.  This is indeed the case for the supergroup ${\rm SOSp}(\C^{1|2n})$.

\part{Endomorphism algebras of orthogonal and symplectic groups}

\section{Endomorphism algebras of the orthogonal group}\label{sect:ker-O}

Recall from Section \ref{sect:Brauer-algebra} that $B_r^r(\delta_V)$ is the Brauer algebra of degree $r$.
Thus $\Ker{F_r^r}$ is a two-sided ideal of $B_r^r(\delta_V)$, and $B_r^r(\delta_V)/Ker{F}_r^r$
is canonically isomorphic to the endomorphism algebra $\End_G(V^{\ot r})$ by Theorem \ref{thm:fft-sft}(2).
In order to understand the algebraic structure of $\End_G(V^{\ot r})$, we need to understand that
of $Ker{F}_r^r$.  
The diagrammatic techniques presented above allow us to do this in the special cases of the orthogonal group $\Or(V)$ and symplectic group $\Sp(V)$ \cite{LZ5}. This provides the solution of a problem raised by the work of Brauer \cite{B} and Weyl \cite{W}.  

In this section and Section \ref{sect:ker-Sp}, we describe the ideals $Ker{F}_r^r$, and develop presentations of  
endomorphism algebras of the orthogonal and symplectic groups.  

In the remainder of this section, we assume that $V=V_{\bar0}$ is purely even with $\delta_V=m$. Thus the orthosymplectic supergroup reduces to the orthogonal group $\SO(V)$.  

The material below is taken from \cite{LZ5}. 

\subsection{Generators of the kernel}

The two-sided ideal $I(m, 0)$ (see Lemma \ref{cor:sftgl}) in $\K\Sym_r$  is the sum of two-sided simple ideals 
$\oplus_\mu I(\mu)$ over those partitions $\mu$ of $r$ which contain
the partition $(1, 1, \dots, 1)$  of $m+1$  (corresponding to the Young diagram with only one column of $m+1$ boxes). 

Set $\epsilon=1$ and consider $\Sigma_+(m+1)$. 
For $p=0, 1, \dots, m+1$, let $E_{m+1-p}$ denote the element of the
Brauer algebra $B_{m+1}^{m+1}(m)$ of degree $m+1$ shown in Figure \ref{Ep}.
\begin{figure}[h]
\begin{center}
\begin{picture}(100, 60)(0,0)
\put(-35, 27){$E_{m+1-p}=$}

\put(25, 40){\line(0, 1){20}}
\put(33, 50){...}
\put(50, 40){\line(0, 1){20}}

\qbezier(58, 40)(100, 90)(108, 0)
\qbezier(72, 40)(95, 70)(97, 0)
\put(97, 7){...}
\put(98, -5){\tiny$p$}

\put(20, 20){\line(1, 0){60}}
\put(20, 20){\line(0, 1){20}}
\put(80, 20){\line(0, 1){20}}
\put(20, 40){\line(1, 0){60}}
\put(40, 28){\tiny${m+1}$}

\put(25, 20){\line(0, -1){20}}
\put(33, 10){...}
\put(50, 20){\line(0, -1){20}}

\qbezier(58, 20)(95, -30)(108, 60)
\qbezier(72, 20)(95, -10)(97, 60)
\put(97, 53){...}
\put(98, 60){\tiny$p$}
\end{picture}
\end{center}
\caption{}
\label{Ep}
\end{figure}

\begin{lemma}\label{lem:E-Z}
For all $0\le k\le m+1$, the elements $E_k$ are linear combinations of
Brauer diagrams over $\Z$.
\end{lemma}
This is evident from the definition of these elements.
They also have the following properties.
\begin{lemma}\label{lemma:Ep-1}
\begin{enumerate}
\item $*E_p = E_{m+1-p}$ for all $p$.
\item $F_p E_p = E_p F_p=p!(m+1-p)! E_p$.
\item $e_i E_p = E_p e_i=0$ for all $i\le m$.
\end{enumerate}
\end{lemma}
\begin{proof} Both (1) and (2) follow easily from the pictorial representation of
$E_p$ given in Figure \ref{Ep}. If $i\ne p$, then $e_i F_p=F_p e_i =0$. Thus (3)
holds for all $i\ne p$. The $i=p$ case of (3) follows from the fact that
\[
\begin{picture}(80, 60)(0, -30)
\put(0, 10){\line(1, 0){60}}
\put(0, -10){\line(1, 0){60}}
\put(0, 10){\line(0, -1){20}}
\put(60, 10){\line(0, -1){20}}
\put(18, -3){$m+1$}

\put(10, 10){\line(0, 1){15}}
\put(20, 15){$\cdots$}
\put(40, 10){\line(0, 1){15}}

\qbezier(50, -10)(60, -40)(65, 0)
\qbezier(50, 10)(60, 40)(65, 0)

\put(10, -10){\line(0, -1){15}}
\put(20, -20){$\cdots$}
\put(40, -10){\line(0, -1){15}}

\put(75, -3){$=0$,}
\end{picture}
\]
which is implied by Lemma \ref{lem:Sigma-1}(2) when $r=m+1$ and $\epsilon=1$.
\end{proof}

The arguments used in the proof of \cite[Corollary 5.13]{LZ5} lead to
\begin{corollary}\label{cor:ann}
Let $D$ be any diagram in $B_{m+1}^{m+1}(n)$ which has fewer
than $m+1$ through strings. Then $D E_i = E_i D = 0$ for all $i$.
\end{corollary}

Note that $E_0=E_{m+1}=\Sigma_{+1}(m+1)$.
\begin{proposition}\label{lem:o-generat}
Assume $r>m$. As a two-sided ideal of the Brauer algebra $B_r^r(m)$, $\Ker{F}_r^r$
is generated by $E_p$ for all $0\le p\le m+1$.
\end{proposition}
\begin{proof}
The proof of Proposition \ref{lem:sp-generat} can easily be modified
to prove the assertion above.  The two required modifications are that for
any $(2r, 0)$ Brauer diagram $A$ with
associated invariant functional $\gamma=F(A)$,
(i) the definition \eqref{eq:AS} of $A_S$ needs to be changed to
\[
A_S= \sum_{\pi\in\Sym_S} (-1)^{|\pi|}A\circ\pi;
\]
(ii) we only need to consider subsets $S$ of $[1, 2r]$ which will not lead to
the trivial vanishing of $A_S$.
With these modifications, the arguments following \eqref{eq:AS}
may be repeated verbatim, leading to the conclusion that
$\Ker{F}_r^r$ is generated as a two-sided ideal of
$B_r^r(-2n)$  by elements of the form Figure \ref{Ep-1}.
\begin{figure}[h]
\begin{center}
\begin{picture}(100, 60)(-20,0)

\put(20, 20){\line(1, 0){60}}
\put(20, 20){\line(0, 1){20}}
\put(80, 20){\line(0, 1){20}}
\put(20, 40){\line(1, 0){60}}
\put(42, 28){\tiny${m+1}$}

\qbezier(30, 40)(10, 70)(10, 0)
\qbezier(45, 40)(5, 90)(-5,0)
\put(0, 10){...}
\put(0, 0){\tiny$p$}

\put(50, 40){\line(0, 1){20}}
\put(70, 40){\line(0, 1){20}}
\put(55, 50){...}

\put(30, 20){\line(0, -1){20}}
\put(35, 10){...}
\put(50, 20){\line(0, -1){20}}

\qbezier(55, 20)(95, -30)(105, 60)
\qbezier(70, 20)(90, -10)(90, 60)
\put(92, 53){...}
\put(95, 60){\tiny$p$}
\end{picture}
\end{center}
\caption{}
\label{Ep-1}
\end{figure}

Post-multiplying the diagram in Figure \ref{Ep-1} by the invertible element $X_{m+1-p, p}$ , we obtain
Figure \ref{Ep} up to a sign. This completes the proof.
\end{proof}

\begin{remark}
Figure \ref{Ep-1} is the $p=q$ analogue of Figure \ref{Dpq-1}.
In the present case,  diagrams of the form Figure \ref{Dpq-1} with $p>q$ vanish identically,
since $\Sigma_{+1}(m+1)$ is the total antisymmetriser in $\Sym_{m+1}$.
\end{remark}

\subsection{Formulae for the $E_i$}

If $k,l$ are integers such that $1\leq k<l$, write
$A(k,l):=\Sigma_{+1}(\Sym_{\{k,k+1,\dots,l\}})$ for the total antisymmetriser
in $\Sym_{\{k,k+1,\dots,l\}}$.
By convention,  $A(k,l)=1$ if $k\geq l$.
Represent $A(1, t)A(t+1, t+s)$ in $B_{t+s}^{t+s}(m)$ pictorially by
\[
\begin{picture}(120, 80)(0, -40)
\put(5, 10){\line(1, 0){45}}
\put(5, -10){\line(1, 0){45}}
\put(5, 10){\line(0, -1){20}}
\put(50, 10){\line(0, -1){20}}
\put(25, -3){$t$}

\put(10, 10){\line(0, 1){25}}
\put(45, 10){\line(0, 1){25}}
\put(23, 25){...}

\put(10, -10){\line(0, -1){25}}
\put(45, -10){\line(0, -1){25}}
\put(23, -30){...}

\put(60, 10){\line(1, 0){45}}
\put(60, -10){\line(1, 0){45}}
\put(60, 10){\line(0, -1){20}}
\put(105, 10){\line(0, -1){20}}
\put(80, -3){$s$}

\put(100, -10){\line(0, -1){25}}
\put(80, -30){...}
\put(65, -10){\line(0, -1){25}}
\put(100, 10){\line(0, 1){25}}
\put(65, 10){\line(0, 1){25}}
\put(80, 25){...}
\put(110, -5){.}
\end{picture}
\]
The lemma below is \cite[Lemma 6.6.]{LZ5}, which is the graphical reformulation of
some of the computations in the proofs of \cite[Corollary 5.2]{LZ4} and
\cite[Theorem 5.10]{LZ4}.

\begin{lemma}[\cite{LZ5}]\label{lem:AfU}
For all $k=0, 1, \dots, i$
\begin{equation}\label{eq:AfU}
\begin{aligned}
\begin{picture}(120, 80)(5, -40)
\put(5, 10){\line(1, 0){45}}
\put(5, -10){\line(1, 0){45}}
\put(5, 10){\line(0, -1){20}}
\put(50, 10){\line(0, -1){20}}
\put(23, -3){\tiny$i$}

\qbezier(45, 10)(55, 30) (65, 10)
\put(10, 10){\line(0, 1){25}}
\put(35, 10){\line(0, 1){25}}
\put(18, 25){...}

\put(10, -10){\line(0, -1){25}}
\put(25, -10){\line(0, -1){25}}
\put(13, -30){...}

\qbezier(30, -10)(55, -50) (80, -10)
\qbezier(45, -10)(55, -30) (65, -10)
\put(55, -29){.}
\put(55, -23){.}
\put(55, -26){.}
\put(55, -40){\tiny$k$}

\put(60, 10){\line(1, 0){45}}
\put(60, -10){\line(1, 0){45}}
\put(60, 10){\line(0, -1){20}}
\put(105, 10){\line(0, -1){20}}
\put(67, -3){\tiny${m+1-i}$}

\put(100, -10){\line(0, -1){25}}
\put(88, -30){...}
\put(85, -10){\line(0, -1){25}}
\put(100, 10){\line(0, 1){25}}
\put(75, 10){\line(0, 1){25}}
\put(83, 25){...}
\put(110, -5){$=$}
\end{picture}
\begin{picture}(120, 80)(-5, -40)
\put(-6, -5){$k^2$}
\put(5, 10){\line(1, 0){45}}
\put(5, -10){\line(1, 0){45}}
\put(5, 10){\line(0, -1){20}}
\put(50, 10){\line(0, -1){20}}
\put(20, -3){\tiny${i-1}$}

\put(10, 10){\line(0, 1){25}}
\put(45, 10){\line(0, 1){25}}
\put(23, 25){...}

\put(10, -10){\line(0, -1){25}}
\put(25, -10){\line(0, -1){25}}
\put(13, -30){...}

\qbezier(30, -10)(55, -50) (80, -10)
\qbezier(45, -10)(55, -30) (65, -10)
\put(55, -29){.}
\put(55, -23){.}
\put(55, -26){.}
\put(53, -40){\tiny${k-1}$}

\put(60, 10){\line(1, 0){45}}
\put(60, -10){\line(1, 0){45}}
\put(60, 10){\line(0, -1){20}}
\put(105, 10){\line(0, -1){20}}
\put(72, -3){\tiny${m-i}$}

\put(100, -10){\line(0, -1){25}}
\put(88, -30){...}
\put(85, -10){\line(0, -1){25}}
\put(100, 10){\line(0, 1){25}}
\put(65, 10){\line(0, 1){25}}
\put(80, 25){...}
\put(110, -5){$+$}
\end{picture}
\begin{picture}(120, 80)(-20, -30)
\put(-15, 5){$\zeta_{i, k}$}
\put(0, -20){\line(1, 0){40}}  \put(55, -20){\line(1, 0){40}}
\put(0, -5){\line(1, 0){40}}   \put(55, -5){\line(1, 0){40}}
\put(0, 20){\line(1, 0){45}}   \put(50, 20){\line(1, 0){45}}
\put(0, 35){\line(1, 0){45}}   \put(50, 35){\line(1, 0){45}}
\put(45, 20){\line(0, 1){15}}  \put(95, 20){\line(0, 1){15}}
\put(0, 20){\line(0, 1){15}}   \put(50, 20){\line(0, 1){15}}
\put(0, -20){\line(0, 1){15}}  \put(55, -20){\line(0, 1){15}}
\put(40, -20){\line(0, 1){15}} \put(95, -20){\line(0, 1){15}}

\put(15, 25){\tiny${i-1}$}    \put(70, 25){\tiny${m-i}$}
\put(10, -15){\tiny${i-k}$} \put(70, -15){\tiny$j$}

\put(5, 35){\line(0, 1){10}}    \put(55, 35){\line(0, 1){10}}
\put(18, 40){...}               \put(68, 40){...}
\put(40, 35){\line(0, 1){10}}   \put(90, 35){\line(0, 1){10}}

\put(5, 20){\line(0, -1){25}}   \put(20, 20){\line(0, -1){25}}
\put(75, 20){\line(0, -1){25}}  \put(90, 20){\line(0, -1){25}}
\put(8, 5){...}                 \put(78, 5){...}

\put(5, -20){\line(0, -1){10}}  \put(90, -20){\line(0, -1){10}}
\put(35, -20){\line(0, -1){10}} \put(60, -20){\line(0, -1){10}}
\put(18, -25){...}                \put(67, -25){...}

\qbezier(25, 20)(48, -10) (70, 20)
\qbezier(40, 20)(48, 10) (55, 20)
\put(48, 12){.}
\put(48, 9){.}
\put(48, 6){.}
\put(52, 9){\tiny$k$}
\qbezier(25, -5)(48, 5) (70,-5)

\put(105, -5){,}
\end{picture}
\end{aligned}
\end{equation}
where $j={m+1-i-k}$ and $\zeta_{i, k}= \frac{1}{(i-k-1)!(m-i-k)!}$.
\end{lemma}
\begin{proof} When $k=0$, \eqref{eq:AfU} is an identity.

We use Lemma \ref{lem:Sigma-1}(1) twice to obtain
\begin{eqnarray}\label{eq:AfU-1}
\begin{aligned}
\begin{picture}(130, 80)(-10, -40)
\put(5, 10){\line(1, 0){45}}
\put(5, -10){\line(1, 0){45}}
\put(5, 10){\line(0, -1){20}}
\put(50, 10){\line(0, -1){20}}
\put(25, -3){$t$}

\put(10, 10){\line(0, 1){25}}
\put(35, 10){\line(0, 1){25}}
\put(18, 25){...}

\put(10, -10){\line(0, -1){25}}
\put(35, -10){\line(0, -1){25}}
\put(18, -30){...}

\put(100, -10){\line(0, -1){25}}
\put(83, -30){...}
\put(75, -10){\line(0, -1){25}}
\put(100, 10){\line(0, 1){25}}
\put(75, 10){\line(0, 1){25}}
\put(83, 25){...}

\put(60, 10){\line(1, 0){45}}
\put(60, -10){\line(1, 0){45}}
\put(60, 10){\line(0, -1){20}}
\put(105, 10){\line(0, -1){20}}
\put(80, -3){$s$}

\qbezier(45, 10)(55, 45)(65, 10)
\qbezier(45, -10)(55, -45)(65, -10)
\put(110, -5){$=$}
\end{picture}
\begin{picture}(130, 80)(-15, -40)
\put(-15, -5){$\psi_{t, s}$}
\put(5, 10){\line(1, 0){45}}
\put(5, -10){\line(1, 0){45}}
\put(5, 10){\line(0, -1){20}}
\put(50, 10){\line(0, -1){20}}
\put(20, -3){$t-1$}

\put(10, 10){\line(0, 1){25}}
\put(45, 10){\line(0, 1){25}}
\put(23, 25){...}

\put(10, -10){\line(0, -1){25}}
\put(45, -10){\line(0, -1){25}}
\put(23, -30){...}

\put(60, 10){\line(1, 0){45}}
\put(60, -10){\line(1, 0){45}}
\put(60, 10){\line(0, -1){20}}
\put(105, 10){\line(0, -1){20}}
\put(72, -3){$s-1$}

\put(100, -10){\line(0, -1){25}}
\put(80, -30){...}
\put(65, -10){\line(0, -1){25}}
\put(100, 10){\line(0, 1){25}}
\put(65, 10){\line(0, 1){25}}
\put(80, 25){...}
\put(110, -5){$+$}
\end{picture}
\begin{picture}(120, 80)(-25, -30)
\put(-15, 5){$\phi_{t, s}$}
\put(0, -20){\line(1, 0){45}}  \put(50, -20){\line(1, 0){45}}
\put(0, -5){\line(1, 0){45}}   \put(50, -5){\line(1, 0){45}}
\put(0, 20){\line(1, 0){45}}   \put(50, 20){\line(1, 0){45}}
\put(0, 35){\line(1, 0){45}}   \put(50, 35){\line(1, 0){45}}
\put(45, 20){\line(0, 1){15}}  \put(95, 20){\line(0, 1){15}}
\put(0, 20){\line(0, 1){15}}   \put(50, 20){\line(0, 1){15}}
\put(0, -20){\line(0, 1){15}}  \put(50, -20){\line(0, 1){15}}
\put(45, -20){\line(0, 1){15}} \put(95, -20){\line(0, 1){15}}

\put(10, 25){$t-1$}       \put(60, 25){$s-1$}
\put(10, -15){$t-1$}      \put(60, -15){$s-1$}

\put(5, 35){\line(0, 1){10}}    \put(55, 35){\line(0, 1){10}}
\put(18, 40){...}               \put(68, 40){...}
\put(40, 35){\line(0, 1){10}}   \put(90, 35){\line(0, 1){10}}

\put(5, -20){\line(0, -1){10}}  \put(90, -20){\line(0, -1){10}}

\put(5, 20){\line(0, -1){25}}   \put(20, 20){\line(0, -1){25}}
\put(75, 20){\line(0, -1){25}}  \put(90, 20){\line(0, -1){25}}
\put(8, 5){...}                 \put(78, 5){...}

\put(40, -20){\line(0, -1){10}} \put(55, -20){\line(0, -1){10}}
\put(18, -28){...}                \put(67, -28){...}

\qbezier(30, 20)(48, -5) (65, 20)
\qbezier(30, -5)(48, 10) (65,-5)

\put(105, -5){,}
\end{picture}
\end{aligned}
\end{eqnarray}
where
\[
\psi_{t, s}=m+2-t-s, \quad \phi_{t, s}=\frac{1}{(t-2)!(s-2)!}.
\]
The case $k=1$ of \eqref{eq:AfU} can be obtained by setting $t=i$ and $s=m+1-i$.

Now use induction on $k$. Post-composing $I_{i-k-1}\otimes U\otimes I_{m-i-k}$
to \eqref{eq:AfU} we obtain
\begin{equation*}
\begin{aligned}
\begin{picture}(120, 80)(5, -40)
\put(5, 10){\line(1, 0){45}}
\put(5, -10){\line(1, 0){45}}
\put(5, 10){\line(0, -1){20}}
\put(50, 10){\line(0, -1){20}}
\put(23, -3){\tiny$i$}

\qbezier(45, 10)(55, 30) (65, 10)
\put(10, 10){\line(0, 1){25}}
\put(35, 10){\line(0, 1){25}}
\put(18, 25){...}

\put(10, -10){\line(0, -1){25}}
\put(25, -10){\line(0, -1){25}}
\put(13, -30){...}

\qbezier(30, -10)(55, -50) (80, -10)
\qbezier(45, -10)(55, -30) (65, -10)
\put(55, -29){.}
\put(55, -23){.}
\put(55, -26){.}
\put(50, -40){\tiny${k+1}$}

\put(60, 10){\line(1, 0){45}}
\put(60, -10){\line(1, 0){45}}
\put(60, 10){\line(0, -1){20}}
\put(105, 10){\line(0, -1){20}}
\put(67, -3){\tiny${m+1-i}$}

\put(100, -10){\line(0, -1){25}}
\put(88, -30){...}
\put(85, -10){\line(0, -1){25}}
\put(100, 10){\line(0, 1){25}}
\put(75, 10){\line(0, 1){25}}
\put(83, 25){...}
\put(110, -5){$=$}
\end{picture}
\begin{picture}(120, 80)(-5, -40)
\put(-6, -5){$k^2$}
\put(5, 10){\line(1, 0){45}}
\put(5, -10){\line(1, 0){45}}
\put(5, 10){\line(0, -1){20}}
\put(50, 10){\line(0, -1){20}}
\put(20, -3){\tiny${i-1}$}

\put(10, 10){\line(0, 1){25}}
\put(45, 10){\line(0, 1){25}}
\put(23, 25){...}

\put(10, -10){\line(0, -1){25}}
\put(25, -10){\line(0, -1){25}}
\put(13, -30){...}

\qbezier(30, -10)(55, -50) (80, -10)
\qbezier(45, -10)(55, -30) (65, -10)
\put(55, -29){.}
\put(55, -23){.}
\put(55, -26){.}
\put(55, -40){\tiny$k$}

\put(60, 10){\line(1, 0){45}}
\put(60, -10){\line(1, 0){45}}
\put(60, 10){\line(0, -1){20}}
\put(105, 10){\line(0, -1){20}}
\put(72, -3){\tiny${m-i}$}

\put(100, -10){\line(0, -1){25}}
\put(88, -30){...}
\put(85, -10){\line(0, -1){25}}
\put(100, 10){\line(0, 1){25}}
\put(65, 10){\line(0, 1){25}}
\put(80, 25){...}
\put(110, -5){$+$}
\end{picture}
\begin{picture}(120, 80)(-20, -30)
\put(-15, 5){$\zeta_{i, k}$}
\put(0, -20){\line(1, 0){40}}  \put(55, -20){\line(1, 0){40}}
\put(0, -5){\line(1, 0){40}}   \put(55, -5){\line(1, 0){40}}
\put(0, 20){\line(1, 0){45}}   \put(50, 20){\line(1, 0){45}}
\put(0, 35){\line(1, 0){45}}   \put(50, 35){\line(1, 0){45}}
\put(45, 20){\line(0, 1){15}}  \put(95, 20){\line(0, 1){15}}
\put(0, 20){\line(0, 1){15}}   \put(50, 20){\line(0, 1){15}}
\put(0, -20){\line(0, 1){15}}  \put(55, -20){\line(0, 1){15}}
\put(40, -20){\line(0, 1){15}} \put(95, -20){\line(0, 1){15}}

\put(15, 25){\tiny${i-1}$}    \put(70, 25){\tiny${m-i}$}
\put(10, -15){\tiny${i-k}$} \put(70, -15){\tiny$j$}

\put(5, 35){\line(0, 1){10}}    \put(55, 35){\line(0, 1){10}}
\put(18, 40){...}               \put(68, 40){...}
\put(40, 35){\line(0, 1){10}}   \put(90, 35){\line(0, 1){10}}

\put(5, 20){\line(0, -1){25}}   \put(20, 20){\line(0, -1){25}}
\put(75, 20){\line(0, -1){25}}  \put(90, 20){\line(0, -1){25}}
\put(8, 5){...}                 \put(78, 5){...}

\put(5, -20){\line(0, -1){10}}  \put(90, -20){\line(0, -1){10}}
\put(20, -20){\line(0, -1){10}} \put(75, -20){\line(0, -1){10}}
\put(10, -25){...}                \put(78, -25){...}

\qbezier(25, 20)(48, -10) (70, 20)
\qbezier(40, 20)(48, 10) (55, 20)
\put(48, 12){.}
\put(48, 9){.}
\put(48, 6){.}
\put(52, 9){\tiny$k$}
\qbezier(25, -5)(48, 5) (70,-5)
\qbezier(25, -20)(48, -35) (70,-20)

\put(105, -5){.}
\end{picture}
\end{aligned}
\end{equation*}
By using \eqref{eq:AfU-1} in the bottom half of the second diagram on the right hand side, we
obtain \eqref{eq:AfU} for $k+1$, completing the proof.
\end{proof}

Following \cite[\S 4.2]{LZ4}, we introduce the elements of $B_{m+1}^{m+1}(m)$ below.
For $p=0,1,\dots,m+1$, let
\[
F_p:=A(1,p)A(p+1,m+1),
\]
where $F_0$ is interpreted as $A(1,m+1)$.
For $j=0,1,2,\dots, i$, define $e_i(j)=e_{i,i+1}e_{i-1,i+2}\dots e_{i-j+1,i+j}$.
Note that $e_i(0)=1$ by convention.
We have the following formulae for the $E_i$.
\begin{lemma}\label{lem:Ep}
For $i=0, 1, \dots, m+1$, let $min_i=min(i, m+1-i)$. Then
\begin{eqnarray}\label{eq:E-formula}
E_i=\sum_{j=0}^{min_i}(-1)^j c_i(j) \Xi_i(j) \quad\text{with}\quad \Xi_i(j)=F_i e_i(j) F_i,
\end{eqnarray}
where $c_i(j)=\left((i-j)!(m+1-i-j)!(j!)^2 \right)^{-1}$.
\end{lemma}

\begin{remark}
For $0\le i\le \left[\frac{m+1}{2}\right]$, the lemma states that the $E_i$ are the elements defined
in \cite[Definition 4.2]{LZ4} with the same notation.
\end{remark}

\begin{proof}
We have $*\Xi_i(j)=\Xi_{m+1-i}(j)$. For $i\le\left[\frac{m}{2}\right]$,
\[
*\left(\sum_{j=0}^{min_i}(-1)^j c_i(j)\Xi_i(j)\right) = \sum_{j=0}^{min_i}(-1)^j c_{m+1-i}(j)\Xi_{m+1-i}(j),
\]
since $c_i(j)=c_{m+1-i}(j)$. Therefore, equation \eqref{eq:E-formula} will hold for
all $i$ by Lemma \ref{lemma:Ep-1}(1), if we can show that it
holds for $0\le i\le \left[\frac{m}{2}\right]$.
This will be done in two steps.

(i).
We first show that for each $i\le \left[\frac{m}{2}\right]$, there exist scalars $x_i(j)$ such that
\[
\begin{aligned}
E_i=\sum_{j=0}^i x_i(j)\Xi_i(j).
\end{aligned}
\]

The case $i=0$ is obvious as we have $E_0=A(1, m+1)$. Thus we only need to consider the case with $i\ge 1$.

Let us label the vertices of $E_i$ (see Figure \ref{Ep}) in the bottom row by $1, 2, \dots, m+1$
from left to right,
and those in the top row by $1', 2', \dots, (m+1)'$ from left to right.
Let $L=\{1, 2, \dots, i\}$, $R=\{i+1, i+2, \dots, m+1\}$,
$L'=\{1', 2', \dots, i'\}$ and $R'=\{(i+1)', (i+2)', \dots, (m+1)'\}$.
Since $A(1, m+1)$ has through strings only,
a Brauer diagram in $E_i$ can only have the following types of edges
(an edge is represented by its pair of vertices)
\[
\begin{aligned}
(a, t)  \in L\times R, \quad
(a', t') \in L'\times R',\\
(a', b) \in L'\times L, \quad
(s', t)\in R'\times R,
\end{aligned}
\]
and the numbers of edges in $L\times R$ and in $L'\times R'$
must be equal.
Thus it follows Lemma \ref{lemma:Ep-1}(2) and
the antisymmetrising property of $A(1, i)$ and $A(i+1, m+1)$ that
$E_i$ is a linear combination of $\Xi_i(j)$.

(ii). To determine the scalar $x_i(0)$, we observe that
the terms in $A(1, m+1)$ which do not contain $s_i$ make up
$F_i=A(1, i)A(i+1, m+1)$. Note that
\[
\begin{picture}(80, 60)(0,0)

\put(0, 20){\line(1, 0){40}}
\put(0, 40){\line(1, 0){40}}
\put(0, 20){\line(0, 1){20}}
\put(40, 20){\line(0, 1){20}}
\put(18, 28){\tiny$p$}

\qbezier(10, 40)(55, 80)(70, 0)
\qbezier(30, 40)(45, 60)(50, 0)
\put(55, 5){...}

\qbezier(10, 20)(55, -20)(70, 60)
\qbezier(30, 20)(45, -0)(50, 60)
\put(55, 48){...}

\put(80, 25){$=$}
\end{picture}
\begin{picture}(50, 60)(-20,0)

\put(0, 20){\line(1, 0){40}}
\put(0, 40){\line(1, 0){40}}
\put(0, 20){\line(0, 1){20}}
\put(40, 20){\line(0, 1){20}}
\put(18, 28){\tiny$p$}

\put(10, 40){\line(0, 1){20}}
\put(15, 50){...}
\put(30, 40){\line(0, 1){20}}
\put(10, 20){\line(0, -1){20}}
\put(15, 10){...}
\put(30, 20){\line(0, -1){20}}

\put(45, 25){.}
\end{picture}
\]
Thus $x_i(0)\Xi_i(0)=A(1, i)A(i+1, m+1)$, and hence $x_i(0)= (i!(m+1-i)!)^{-1}=c_i(0)$.

Now we determine the $x_i(k)$ for all $k>0$. By Lemma \ref{lemma:Ep-1}(3), $e_i E_i=0$.
Using \eqref{eq:AfU} in this relation,
we obtain
\[
(k+1)^2 x_i(k+1) + (i-k)! (m+1-i-k)! \zeta_{i, k} x_i(k)=0, \quad  0\le k\le i.
\]
The recurrent relation with $x_i(0)=c_i(0)$ yields $x_i(k)=(-1)^j c_i(k)$.
\end{proof}

The following result is an easy consequence of Lemma \ref{lem:Ep}.
Recall the elements $X_{s, t}\in\Sym_{s+t}$ shown in Figure \ref{X}.
\begin{corollary}\label{cor:XEX}
For all $i=0, 1, \dots, m+1$, we have $X_{i, m+1-i} E_i X_{m+1-i, i}= E_{m+1-i}$.
\end{corollary}
\begin{proof}
It is easy to show pictorially that $X_{i, M+1-i}\Xi_i(j) X_{m+1-i, i}= \Xi_{m+1-i}(j)$
for all $j\le i$. Since $c_i(j)=c_{m+1-i}(j)$, this proves the claim of the corollary.
\end{proof}

\subsection{Presentation of endomorphism algebras}

The following theorem is \cite[Theorem 6.10]{LZ5}, which is a diagrammatic reformulation of \cite[Theorem 4.3]{LZ4}.
\begin{theorem}[\cite{LZ4, LZ5}] \label{thm:o-main}
The algebra map  ${F}_r^r: B_r^r(m)\longrightarrow \End_{\Or(V)}(V^{\otimes r})$
is injective if $r\le m$. If
$r>m$, the two-sided ideal $\Ker{F}_r^r$ of the Brauer algebra $B_r^r(m)$ is generated by
the element $E=E_\ell$ with $\ell=\left[\frac{m+1}{2}\right]$.
\end{theorem}
\begin{proof} Only the second part of the theorem needs explanation.
By Proposition \ref{lem:o-generat} and Corollary \ref{cor:XEX}, the elements
$E_i$ with $i=0, 1, \dots \ell=\left[\frac{m+1}{2}\right]$ generates $\Ker{F}_r^r$.
Using some general properties of the
symmetric group and Corollary \ref{cor:ann},  we showed in \cite[\S 7]{LZ4} that
$E_{i-1}$ is contained in the ideal generated by $E_i$ for each $i=1, \dots,\ell$.
The theorem follows.
\end{proof}

%
%

\section{Endomorphism algebras of the symplectic group}\label{sect:ker-Sp}

In this section, we assume that $V=V_{\bar1}$ with $\delta_V=\sdim{V}=-2n$.  Thus 
$G=\Sp(V)$.   We study the structure of the endomorphism algebra $\End_G(V^{\ot r})$ for all $r$.  The material is taken from \cite{LZ5} (see also \cite{HX}). 

\subsection{Generators of the kernel}

The two-sided ideal $I(0, 2n)$ (see Lemma \ref{cor:sftgl}) in $\K\Sym_r$  is the sum of two-sided simple ideals 
$\oplus_\mu I(\mu)$ over those partitions $\mu$ of $r$ which contain
the partition $(2n+1)$  (that is, the Young diagram with only one row of $2n+1$ boxes).

Let $\epsilon=-1$, and denote $\Sigma_{-1}(r)$ by $\Sigma(r)$.  Then $I(0, 2n)$ is generated by 
$\Sigma(2n+1)$.

For any $s<r$, there is a natural embedding $B_s^s(-2n)\hookrightarrow B_r^r(-2n)$,
$b\mapsto b\ot I_{r-s}$,
of the Brauer algebra of degree $s$ in that of degree $r$ as associative algebras. Thus we may
regard $B_s^s(-2n)$ as the subalgebra of $B_r^r(-2n)$ consisting of elements of the form
$b\ot I_{r-s}$.

Let $D(p, q)$ denote the element of the Brauer algebra $B_k^k(-2n)$ of degree $k=2n+1-p+q$ shown in
Figure \ref{Dpq}.

\begin{figure}[h]
\begin{center}
\begin{picture}(100, 60)(0,0)
\put(-50, 30){$D(p, q)=$}
\put(5, 40){\line(0, 1){20}}
\put(6, 50){...}
\put(18, 40){\line(0, 1){20}}

\put(30, 58){\tiny${p-q}$}
\qbezier(22, 40)(36, 70)(52, 40)
\put(35, 52){.}
\put(35, 50){.}
\put(35, 48){.}
\qbezier(30, 40)(36, 54)(44, 40)

\qbezier(58, 40)(100, 90)(108, 0)
\qbezier(72, 40)(95, 70)(97, 0)
\put(97, 7){...}
\put(98, -5){\tiny$q$}

\put(0, 20){\line(1, 0){80}}
\put(0, 20){\line(0, 1){20}}
\put(80, 20){\line(0, 1){20}}
\put(0, 40){\line(1, 0){80}}
\put(30, 28){\tiny${2n+1}$}

\put(5, 20){\line(0, -1){20}}
\put(20, 10){...}
\put(50, 20){\line(0, -1){20}}

\qbezier(58, 20)(95, -30)(108, 60)
\qbezier(72, 20)(95, -10)(97, 60)
\put(97, 53){...}
\put(98, 60){\tiny$p$}
\end{picture}
\end{center}
\caption{}
\label{Dpq}
\end{figure}

\begin{proposition}\label{lem:sp-generat}
Assume that $r>n$. As a two-sided ideal of the Brauer algebra $B_r^r(-2n)$,
$\Ker{F}_r^r$  is generated by
$D(p, q)$ and $\ast D(p, q)$ with $p+q\le r$ and $p\le n$.
\end{proposition}
\begin{proof}
Let $A$ be a single $(2r, 0)$ Brauer diagram with $r>n$. Then $F(A)$ is some functional $\gamma$
on $V^{\ot 2r}$.  For any $\pi\in\Sym_{2r}\subset B_{2r}^{2r}(-2n)$,
$A\circ \pi$ is defined. Note that $A$ has only one row of vertices at the bottom,
which will be labelled $1, 2, \dots, 2r$ from left to right.
Choose a subset $S$ of $[1, 2r]$
of cardinality $2n+1$, and consider
$\Sym_S\subset \Sym_{2r}\subset B_{2r}^{2r}(-2n)$. Define
\begin{eqnarray}\label{eq:AS}
A_S=\sum_{\pi\in\Sym_S} A\circ\pi.
\end{eqnarray}
Then by Theorem \ref{thm:fft-sft}(2),
$\Ker{F}_{2r}^0$ is spanned by $A_S$ for all $A$ and $S$. Given $A_S$, we define
\[
A_S^\natural=A_S\circ(I_r\otimes U_r)\in B_r^r(-2n).
\]
Then $\Ker{F}_{r}^r$ is spanned by $A_S^\natural$ for all $A$ and $S$ by Lemma \ref{lem:transfrom}(2).

We can considerably simplify the description of $\Ker{F}_{2r}^0$ and $\Ker{F}_{r}^r$.
There exist elements $\sigma=(\sigma_1, \sigma_2)$ in the parabolic subgroup $\Sym_r\times\Sym_r$ of $\Sym_{2r}$, which map $S$ to $S'=\{i+1, i+2, \dots, i+2n+1\}\subset[1, 2r]$ for  some $i\le 2r-2n-1$. Let $\sigma_2^{-\tau}=\ast(\sigma_2^{-1})$, where $\ast$ is the anti-involution of $B_r^r(-2n)$. Then
\begin{eqnarray}
&\sigma_2^{-\tau}\circ A_S^\natural\circ \sigma_1^{-1}=(A_S\circ \sigma^{-1})^\natural,\nonumber\\
&A_S\circ \sigma^{-1}=\sum_{\pi\in\Sym_{S'}}(A\circ\sigma^{-1})\circ\pi. \label{eq:symmS}
\end{eqnarray}

By appropriately choosing $\sigma$, we can ensure that $A\circ\sigma^{-1}$ is of the
form shown in Figure \ref{Asigma}.
\begin{figure}[h]
\begin{center}
\begin{picture}(290, 60)(0,0)
\qbezier(-3, 0)(140, 106)(283, 0)
\qbezier(15, 0)(140, 85)(265, 0)
\put(140, 50){.}
\put(140, 48){.}
\put(140, 46){.}

\qbezier(25, 0)(35, 15)(35, 0)
\qbezier(50, 0)(60, 35)(65, 0)
\put(40, 3){...}\put(42, -4){\tiny$t$}

\qbezier(215, 0)(220, 35)(230, 0)
\qbezier(245, 0)(245, 15)(255, 0)
\put(232, 3){...}\put(234, -4){\tiny$t'$}

\qbezier(70, 0)(100, 65)(120, 0) \put(118, -2){$\bullet$}
\qbezier(85, 0)(110, 65)(155, 0)\put(153, -2){$\bullet$}
\put(75, 3){...}

\qbezier(145, 0)(155, 70)(210, 0)
\qbezier(135, 0)(155, 60)(195, 0)
\put(133, -2){$\bullet$}\put(143, -2){$\bullet$}
\put(125, 3){...}
\put(196, 3){...}

\qbezier(105, 0)(135, 50)(170, 0)
\put(103, -2){$\bullet$}\put(168, -2){$\bullet$}
\put(90, 3){...}
\put(176, 3){...}
\end{picture}
\end{center}
\caption{}
\label{Asigma}
\end{figure}
The vertices labelled by $\bullet$ are those in $S'$, which all appear in the middle, and the other vertices
all appear at the left end and right end. Here $t$ denotes the number of edges in $A\circ \sigma^{-1}$
with both vertices in $\{1, 2, \dots, i\}$,
and $t'$ that of the edges with both vertices in $\{i+2n+2, i+2n+3, \dots, 2r\}$.
Note that after such a $\sigma$ is chosen, $\pi\in \text{Sym}_{S'}$ acting on $A\circ \sigma^{-1}$
permutes only vertices labeled by $\bullet$.  Thus every term on the right hand side of \eqref{eq:symmS}
is of the form Figure \ref{Asigma} with the same $t$ and $t'$.

Now $(A_S\circ \sigma^{-1})^\natural$ can be expressed as
$D_1\otimes D_2$, where $D_1\in B_{r_1}^{r_1}(-2n)$ for $r_1$ maximal, $D_2\in B_k^k(-2n)$
with $k>n$ satisfying $r_1+k=r$.
There are several possibilities for $D_2$ depending on
$i$, $t$ and $t'$. Assume $i+2n+1> r$.  If
$t=t'$, then $D_2$ is as shown in Figure \ref{Dpq-1}.
\begin{figure}[h]
\begin{center}
\begin{picture}(100, 60)(-20,0)

\put(0, 20){\line(1, 0){80}}
\put(0, 20){\line(0, 1){20}}
\put(80, 20){\line(0, 1){20}}
\put(0, 40){\line(1, 0){80}}
\put(30, 28){\tiny${2n+1}$}

\qbezier(5, 40)(-15, 60)(-18, 0)
\qbezier(18, 40)(-20, 80)(-32,0)
\put(-28, 10){...}
\put(-28, 0){\tiny$q$}

\put(30, 58){\tiny${p-q}$}
\qbezier(22, 40)(36, 70)(52, 40)
\put(35, 52){.}
\put(35, 50){.}
\put(35, 48){.}
\qbezier(30, 40)(36, 54)(44, 40)

\put(58, 40){\line(0, 1){20}}
\put(73, 40){\line(0, 1){20}}
\put(61, 50){...}

\put(5, 20){\line(0, -1){20}}
\put(20, 10){...}
\put(50, 20){\line(0, -1){20}}

\qbezier(55, 20)(93, -30)(103, 60)
\qbezier(70, 20)(90, -10)(91, 60)
\put(92, 53){...}
\put(95, 60){\tiny$p$}
\end{picture}
\end{center}
\caption{}
\label{Dpq-1}
\end{figure}
If $t<t'$, then $D_2=E\circ(I_s\otimes D_3)$ for some $s$, where $D_3$ is as shown in Figure \ref{Dpq-1},
and $E$ is the product of some $e_i$'s composed with a permutation in
$\Sym_{2n+1+q-p}$ ($D_3$ and $E$ may not be unique). Analogously, $D_2=(D_3\otimes I_{s})\circ E$ if $t>t'$.
Assume that $i+2n+1\le r$. Then $D_2=E\circ(I_{s_1}\otimes \Sigma(2n+1)\otimes I_{s_2})$ for some
$E$ in $B_k^k(-2n)$, and fixed nonnegative integers $s_1$ and $s_2$ satisfying $s_1+s_2+2n+1=k$.

Therefore, $\Ker{F}_r^r$  is generated as a two sided ideal of $B_r^r(-2n)$ by elements
of the form of Figure \ref{Dpq-1} with $2n+1+q-p\le r$.
If $p>n$, we apply the anti-involution $\ast$ of $B_k^k(-2n)$ to the element of Figure \ref{Dpq-1} to obtain
the element shown in Figure \ref{Dpq-2},
\begin{figure}[h]
\begin{center}
\begin{picture}(100, 60)(-20,0)

\put(0, 20){\line(1, 0){80}}
\put(0, 20){\line(0, 1){20}}
\put(80, 20){\line(0, 1){20}}
\put(0, 40){\line(1, 0){80}}
\put(30, 28){\tiny${2n+1}$}

\put(5, 40){\line(0, 1){20}}
\put(20, 40){\line(0, 1){20}}
\put(10, 50){...}
\put(10, 55){\tiny$q$}

\put(35, 58){\tiny${p-q}$}
\qbezier(27, 40)(41, 70)(57, 40)
\put(40, 52){.}
\put(40, 50){.}
\put(40, 48){.}
\qbezier(35, 40)(41, 54)(49, 40)

\qbezier(62, 40)(90, 80)(102, 0)
\qbezier(75, 40)(85, 60)(90, 0)
\put(90, 10){...}

\put(30, 20){\line(0, -1){20}}
\put(75, 20){\line(0, -1){20}}
\put(48, 10){...}
\put(48, 0){\tiny$p$}

\qbezier(5, 20)(-5, -5)(-10, 60)
\qbezier(20, 20)(-15, -25)(-25, 60)
\put(-22, 50){...}

\end{picture}
\end{center}
\caption{}
\label{Dpq-2}
\end{figure}
which we denote by $D$. Recall the element $X_{s, t}$ of Figure \ref{X},
which belongs to $\Sym_{s+t}$, where $\Sym_{s+t}$ is regarded as embedded in $B_{s+t}^{s+t}(-2n)$.
Then $X_{2n+1-p, q}\circ D\circ X_{2n+1-2p+q, p}$ is
of the form shown in Figure \ref{Dpq-1}, but with $p$ replaced by $2n+1-p\le n$.

Therefore, we only need to consider Figure \ref{Dpq-1} with $p\le n$ and its $\ast$ image.
Post-composing $X_{2n+1-p, q}$ to Figure \ref{Dpq-1} turns the latter into the form shown in Figure \ref{Dpq}.
Since $X_{2n+1-p, q}$ is invertible in $B_r^r(-2n)$,
$\Ker{F}_r^r$ as a two-sided ideal of $B_r^r(-2n)$ is generated by elements of
$D(p, q)$ and $\ast D(p, q)$ with $2n+1+q-p\le r$ and $p\le n$.
\end{proof}

\subsection{The element $\Phi$}

For each $k$ such that $0\le k\le \left[\frac{n+1}{2}\right]$,
define the element $E(k)= \prod_{j=1}^k e_{n+2-2j}$ of
$B_{n+1}^{n+1}(-2n)$, where $E(0)$ is the identity by convention.
Then define
\[
\Xi_k= \Sigma(n+1) E(k) \Sigma(n+1),
\]
which may be represented pictorially as

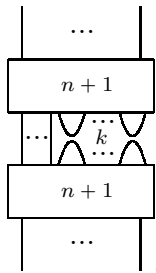
\begin{center}
\begin{picture}(100, 100)(0,-40)
\put(5, 40){\line(0, 1){20}}
\put(23, 50){...}
\put(50, 40){\line(0, 1){20}}

\put(0, 20){\line(1, 0){55}}
\put(0, 20){\line(0, 1){20}}
\put(55, 20){\line(0, 1){20}}
\put(0, 40){\line(1, 0){55}}
\put(20, 28){\tiny${n+1}$}

\put(5, 20){\line(0, -1){20}}
\put(6, 10){...}
\put(16, 20){\line(0, -1){20}}

\qbezier(19, 20)(24, 2)(29, 20)
\put(31, 16){...}
\qbezier(42, 20)(47, 2)(52, 20)
\put(33, 8){\tiny$k$}
\qbezier(19, 0)(24, 18)(29, 0)
\put(31, 4){...}
\qbezier(42, 0)(47, 18)(52, 0)

\put(0, 0){\line(1, 0){55}}
\put(0, 0){\line(0, -1){20}}
\put(55, 0){\line(0, -1){20}}
\put(0, -20){\line(1, 0){55}}
\put(20, -12){\tiny${n+1}$}
\put(5, -20){\line(0, -1){20}}
\put(23, -30){...}
\put(50, -20){\line(0, -1){20}}
\put(55, -40){.}
\end{picture}
\end{center}
Now define the following element of $B_{n+1}^{n+1}(-2n)$.
\begin{eqnarray}\label{Phi}
\begin{aligned}
\Phi=\sum_{k=0}^{\left[\frac{n+1}{2}\right]} a_k \Xi_k \quad \text{with} \quad
a_k =\frac{1}{(2^k k! )^2 (n+1-2k)!}.
\end{aligned}
\end{eqnarray}

\begin{lemma}\label{lem:Phi-Z}
The element $\Phi$ is precisely the sum of all the Brauer diagrams in  $B_{n+1}^{n+1}(-2n)$.
In particular $\Phi$ is defined over the ring $\Z$ of integers.
\end{lemma}
\begin{proof}
Note that $\Xi_k=\sum_{(\pi,\sigma)\in(\Sym_{n+1})^2}\pi E(k)\sigma$ is simply the sum
of all the diagrams with $t=n+1-2k$ through strings, each one occurring with coefficient equal to
the order of the centraliser in $ (\Sym_{n+1})^2$ of $E(k)$. But this order is evidently $a_k\inv$.
\end{proof}

We have the following result.
\begin{lemma}\label{lem:Phi}
The element $\Phi$ has the following properties:
\begin{enumerate}
\item \label{lem:Phi-1} $e_i\Phi=\Phi e_i=0$ for all $e_i\in B_{n+1}^{n+1}(-2n)$;
\item \label{lem:Phi-2} $\Phi^2 = (n+1)!\Phi$;
\item \label{lem:Phi-3} $\ast\Phi=\Phi$;
\item \label{lem:Phi-4} $\Phi\in \Ker{F}_{n+1}^{n+1}$.
\end{enumerate}
\end{lemma}

\begin{proof}
Part (\ref{lem:Phi-3}) follows from the fact that
$\ast\Xi_k=\Xi_k$ for all $k$. Part (\ref{lem:Phi-2}) immediately follows from (\ref{lem:Phi-1}).

Since $*(e_i\circ\Phi)=\Phi\circ e_{n+1-i}$, we only need to show that $e_i\circ\Phi=0$
for all $i$ in order to prove part (\ref{lem:Phi-1}).
In view of the symmetrising property of $\Sigma(n+1)$,
it suffices to show that $e_n\circ\Phi=0$. Consider
$(I_{n-1}\otimes A_1)\circ \Xi_k$, which can be shown to be equal to
\begin{eqnarray}\label{eq:ASigma}
\begin{aligned}
\begin{picture}(100, 100)(20,-40)
\put(-30, 8){$-4k^2$}
\put(5, 40){\line(0, 1){20}}
\put(23, 50){...}
\put(50, 40){\line(0, 1){20}}

\put(0, 20){\line(1, 0){55}}
\put(0, 20){\line(0, 1){20}}
\put(55, 20){\line(0, 1){20}}
\put(0, 40){\line(1, 0){55}}
\put(20, 28){\tiny${n-1}$}

\put(5, 20){\line(0, -1){20}}
\put(6, 10){...}
\put(16, 20){\line(0, -1){20}}

\qbezier(19, 20)(24, 2)(29, 20)
\put(31, 16){...}
\qbezier(42, 20)(47, 2)(52, 20)
\put(30, 8){\tiny$k$-1}
\qbezier(19, 0)(24, 18)(29, 0)
\put(31, 4){...}
\qbezier(42, 0)(47, 18)(52, 0)

\qbezier(57, 0)(62, 18)(67, 0)

\put(0, 0){\line(1, 0){70}}
\put(0, 0){\line(0, -1){20}}
\put(70, 0){\line(0, -1){20}}
\put(0, -20){\line(1, 0){70}}
\put(20, -12){\tiny${n+1}$}
\put(5, -20){\line(0, -1){20}}
\put(28, -30){$\dots$}
\put(65, -20){\line(0, -1){20}}

\put(80, 8){$+$}
\end{picture}
\begin{picture}(100, 100)(-90,-40)
\put(-110, 8){$(n+1-2k)(n-2k)$}
\put(5, 40){\line(0, 1){20}}
\put(23, 50){...}
\put(50, 40){\line(0, 1){20}}

\put(0, 20){\line(1, 0){55}}
\put(0, 20){\line(0, 1){20}}
\put(55, 20){\line(0, 1){20}}
\put(0, 40){\line(1, 0){55}}
\put(20, 28){\tiny${n-1}$}

\put(5, 20){\line(0, -1){20}}
\put(6, 10){...}
\put(16, 20){\line(0, -1){20}}

\qbezier(19, 20)(24, 2)(29, 20)
\put(31, 16){...}
\qbezier(42, 20)(47, 2)(52, 20)
\put(33, 8){\tiny$k$}
\qbezier(19, 0)(24, 18)(29, 0)
\put(31, 4){...}
\qbezier(42, 0)(47, 18)(52, 0)

\qbezier(57, 0)(62, 18)(67, 0)

\put(0, 0){\line(1, 0){70}}
\put(0, 0){\line(0, -1){20}}
\put(70, 0){\line(0, -1){20}}
\put(0, -20){\line(1, 0){70}}
\put(20, -12){\tiny${n+1}$}
\put(5, -20){\line(0, -1){20}}
\put(28, -30){$\dots$}
\put(65, -20){\line(0, -1){20}}
\end{picture}
\end{aligned}
\end{eqnarray}
by using Lemma \ref{lem:Sigma} with $\delta=-2n$. Note that each Brauer diagram summand of the
first term has $n+1-2k$ through strings, while the summands in the
second term have $n-1-2k$ through strings.
Using \eqref{eq:ASigma} one shows by simple calculation that
\[
\sum a_k (I_{n-1}\otimes A_1)\circ \Xi_k=0.
\]
Hence $(I_{n-1}\otimes A_1)\circ\Phi=0$, which implies statement (1).

To prove part (\ref{lem:Phi-4}), we note that the trace of $\frac{F(\Phi)}{(n+1)!}$ is equal to
the dimension of the subspace $F(\Phi)(V^{\otimes(n+1)})$, since
$\frac{F(\Phi)}{(n+1)!}$ is an idempotent by part (\ref{lem:Phi-2}).
In order to evaluate $tr\left(\frac{F(\Phi)}{(n+1)!}\right)$, we first consider
$tr\left(\frac{F(\Xi_k)}{(n+1)!}\right)$,
which is given by
\[
\begin{picture}(145, 60)(0, -30)
\put(0, 0) {$(-1)^{n+1}$}
\put(50, 10){\line(1, 0){70}}
\put(50, -10){\line(1, 0){70}}
\put(50, 10){\line(0, -1){20}}
\put(120, 10){\line(0, -1){20}}
\put(70, -3){$n+1$}

\qbezier(55,10)(60,40)(65, 10)
\qbezier(85,10)(90,40)(95, 10)
\qbezier(105, 10)(135, 40)(140, 0)
\qbezier(115, 10)(130, 30)(130, 0)
\put(70, 15){...}
\put(73, 20){\tiny$k$}

\qbezier(55,-10)(60,-40)(65, -10)
\qbezier(85,-10)(90,-40)(95, -10)
\qbezier(105, -10)(135, -40)(140, 0)
\qbezier(115, -10)(130, -30)(130, 0)
\put(70, -18){...}
\put(73, -25){\tiny$k$}

\put(132, 0){\tiny{...}}

\put(145, 0){$=$}
\end{picture}
\begin{picture}(150, 60)(-45, -30)
\put(-30, 0) {$(-1)^{n+1}\frac{(2n-2k)!}{(n-1)!}$}
\put(50, 10){\line(1, 0){50}}
\put(50, -10){\line(1, 0){50}}
\put(50, 10){\line(0, -1){20}}
\put(100, 10){\line(0, -1){20}}
\put(70, -3){$2k$}

\qbezier(55,10)(60,40)(65, 10)
\qbezier(85,10)(90,40)(95, 10)
\put(70, 15){...}
\put(73, 20){\tiny$k$}

\qbezier(55,-10)(60,-40)(65, -10)
\qbezier(85,-10)(90,-40)(95, -10)
\put(70, -18){...}
\put(73, -25){\tiny$k$}
\put(105, -10){,}
\end{picture}
\]
where the last step uses Lemma \ref{lem:Sigma-1}(2) with $\epsilon=-1$.
Using \eqref{eq:k-1}, one can show that
\[
\begin{picture}(150, 60)(50, -30)
\put(50, 10){\line(1, 0){50}}
\put(50, -10){\line(1, 0){50}}
\put(50, 10){\line(0, -1){20}}
\put(100, 10){\line(0, -1){20}}
\put(70, -3){$2k$}

\qbezier(55,10)(60,40)(65, 10)
\qbezier(85,10)(90,40)(95, 10)
\put(70, 15){...}
\put(73, 20){\tiny$k$}

\qbezier(55,-10)(60,-40)(65, -10)
\qbezier(85,-10)(90,-40)(95, -10)
\put(70, -18){...}
\put(73, -25){\tiny$k$}
\put(105, 0){$=$}
\put(115, 0){$(-1)^k 2^{2k} \frac{n! k!}{(n-k)!}$}
\put(190, -7){.}
\end{picture}
\]
Putting these formulae together, we arrive at
\[
\begin{aligned}
tr\left(\frac{F(\Phi)}{(n+1)!}\right) &= \frac{n!}{(n-1)!}\sum_{k=0}^{\left[\frac{n+1}{2}\right]}
a_k (-1)^k 2^{2 k} \frac{k! (2n-2k)!}{(n-k)!}\\
&=\sum_{k=0}^{\left[\frac{n+1}{2}\right]}
(-1)^k \begin{pmatrix}n \\ k\end{pmatrix}
\begin{pmatrix}2n-2k \\ n-1\end{pmatrix}.
\end{aligned}
\]
There is a binomial coefficient identity stating that the far right hand side is equal to zero.
Hence $F(\Phi)$ is the zero map on $V^{\otimes(n+1)}$.
\end{proof}

The corollary below follows from Lemma \ref{lemma:Ep-1}
and the fact that $\pi \Sigma(n+1)\pi'= \Sigma(n+1)$  for all $\pi, \pi'\in\Sym_{n+1}$.

\begin{corollary}\label{cor:uniqueness-Phi}
The element $\Phi/(n+1)!$ is the central
idempotent in $B_{n+1}^{n+1}(-2n)$ which corresponds to the trivial representation $\rho_1$
of $B_{n+1}^{n+1}(-2n)$, defined by $\rho_1(s_i)=1$ and $\rho_1(e_i)=0$ for all $i$.
It generates a $1$-dimensional two-sided ideal of $B_{n+1}^{n+1}(-2n)$.
\end{corollary}


%
%
\subsection{Presentation of endomorphism algebras}

Recall the natural embedding of the Brauer algebra of degree $s$ in that of degree $t$ for any $t>s$.
\begin{definition}[\cite{LZ5}]\label{def:ideal}
For each $r>n$, let $\langle \Phi\rangle_r$ be the two-sided ideal in the Brauer algebra
$B_r^r(-2n)$ generated by $\Phi$.
\end{definition}

\begin{remark}
A priori, elements such as $(I_{r-q}\ot A_q\ot I_q)(z\ot X_{q, q})(I_{r-q}\ot U_q\ot I_q)$
are not included in $\langle \Phi\rangle_r$ even if $z\in \langle \Phi\rangle_r$.
\end{remark}

We have the following result.
\begin{lemma}\label{lem:Sigma-Phi}
The element $\Sigma(2n+1)$ belongs to $\langle \Phi\rangle_{2n+1}$.
\end{lemma}
\begin{proof}
Consider $B_r^r(-2n)$ for $r>n$. Let $E_r^r(k)= \prod_{j=1}^k e_{r-2j+1}$, and define
\[
\begin{aligned}
\Upsilon(r)_k&=\Sigma(r) E_r^r(k) \Sigma(r), \quad k\ge 1,\\
\Upsilon(r)_{\ge k}&=\text{linear span of $\langle \Phi\rangle_r\cup\{\Upsilon(r)_i\mid i\ge k\}$}.
\end{aligned}
\]
We first want to show that
\begin{eqnarray}\label{eq:step}
\Sigma(r)\in \Upsilon(r)_{\ge\left[\frac{r+1-n}{2}\right]}.
\end{eqnarray}
From the formula for $\Phi$, we obtain
\[
\begin{aligned}
r!(n+1)! \Sigma(r)& =\Sigma(r)\left(\left(\Phi - \sum_{k=1}^{\left[\frac{n+1}{2}\right]}a_k
\Xi_k\right)\otimes I_{r-n-1}\right)\Sigma(r).
\end{aligned}
\]
Thus $\Sigma(r)\in \Upsilon(r)_{\ge 1}$.

Note that for any $z\in \Upsilon(r-2k)_{\ge 1}$,
$\Sigma(r)(z\otimes I_{2k}) E_r^r(k) \Sigma(r)$ belongs to $\Upsilon(r)_{\ge k+1}$.
We can always re-write $\Upsilon(r)_k$ as
\[
\Upsilon(r)_k=\frac{1}{(r-2k)!}\Sigma(r)(\Sigma(r-2k)\otimes I_{2k}) E_r^r(k) \Sigma(r).
\]
If $r-2k>n$, then $\Sigma(r-2k)\in \Upsilon(r-2k)_{\ge 1}$. This implies that
$\Upsilon(r)_k\in \Upsilon(r)_{\ge k+1}$ if $r-2k>n$. Hence
$
\Upsilon(r)_{\ge 1}=\Upsilon(r)_{\ge 2} = \dots = \Upsilon(r)_{\ge\left[\frac{r+1-n}{2}\right]},
$
 and \eqref{eq:step} is proved.

Now consider $\Sigma(2n+1)$.
It follows from \eqref{eq:step} that $\Sigma(2n+1)^2$ can be expressed as a linear combination of
elements in $\langle \Phi\rangle_{2n+1}$ and also elements of the form
\[
\Sigma(2n+1) E_{2n+1}^{2n+1}(i) \Sigma(2n+1) E_{2n+1}^{2n+1}(j) \Sigma(2n+1),
\quad i, j\ge 1+\left[\frac{n}{2}\right].
\]
Using the symmetrising property of $\Sigma(2n+1)$, we can write this element
as $\Sigma(2n+1)(I_{2n+1-2i}\otimes U_i) \Psi_{i j}(I_{2n+1-2j}\otimes A_j) \Sigma(2n+1)$ with
\[
\Psi_{i j}=
(I_{2n+1-2i}\otimes A_i)\Sigma(2n+1)(I_{2n+1-2j}\otimes U_j).
\]
Now $\Psi_{i j}=0$ for all  $i, j\ge 1+\left[\frac{n}{2}\right]$. Hence
$\Sigma(2n+1)^2$ belongs to $\langle \Phi\rangle_{2n+1}$, and so does also $\Sigma(2n+1)$.
\end{proof}

The following results describe the endomorphism algebras $\End_{\Sp(V)}(V^{\otimes r})$ in terms of generators and relations. 
\begin{theorem}[\cite{LZ5}]\label{thm:sp-main}
The  homomorphism  ${F}_r^r: B_r^r(-2n)\longrightarrow \End_{\Sp(V)}(V^{\otimes r})$
of algebras is injective if $r\le n$. If $r\ge n+1$,
then $\Ker{F}_r^r$ is the two-sided ideal of the Brauer algebra $B_r^r(-2n)$ which is generated by
the element $\Phi$ defined by \eqref{Phi}.
\end{theorem}

\begin{proof} Only the second statement requires proof. Thus we assume that $r\ge n+1$.
Consider first the case $r=n+1$. Then there is only one $D(p, q)$ with $p=n$ and $q=0$ (see Figure \ref{Dpq}).
Using $\Sigma(n+1)=\Phi - \sum_{k=1}^{\left[\frac{n+1}{2}\right]}a_k \Xi_k$, we have
\[
D(n, 0)=\frac{D(n, 0)\Phi}{(n+1)!} - \sum_{k=1}^{\left[\frac{n+1}{2}\right]} a_k\frac{D(n, 0)\Xi_k}{(n+1)!}.
\]
Note that

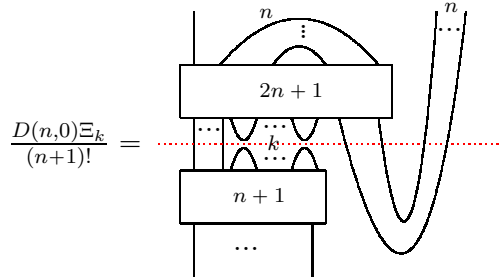
\begin{center}
\begin{picture}(100, 80)(0,-30)
\put(-65, 7){$\frac{D(n, 0)\Xi_k}{(n+1)!}=$}
\put(5, 40){\line(0, 1){20}}

\put(30, 58){\tiny$n$}
\qbezier(15, 40)(45, 75)(75, 40)
\qbezier(35, 40)(45, 54)(55, 40)
\put(45, 54){.}
\put(45, 52){.}
\put(45, 50){.}

\put(0, 20){\line(1, 0){80}}
\put(0, 20){\line(0, 1){20}}
\put(80, 20){\line(0, 1){20}}
\put(0, 40){\line(1, 0){80}}
\put(30, 28){\tiny${2n+1}$}

\put(5, 20){\line(0, -1){20}}
\put(6, 15){...}
\put(16, 20){\line(0, -1){20}}

{\color{red}\qbezier[60](-8, 10)(56, 10)(120, 10)}

\qbezier(19, 20)(24, 3)(29, 20)
\put(31, 16){...}
\qbezier(42, 20)(47, 3)(52, 20)
\put(33, 8){\tiny$k$}
\qbezier(19, 0)(24, 17)(29, 0)
\put(31, 4){...}
\qbezier(42, 0)(47, 17)(52, 0)

\qbezier(60, 20)(90, -100)(108, 60)
\qbezier(75, 20)(87, -75)(97, 60)
\put(97, 53){...}
\put(100, 60){\tiny$n$}

\put(0, 0){\line(1, 0){55}}
\put(0, 0){\line(0, -1){20}}
\put(55, 0){\line(0, -1){20}}
\put(0, -20){\line(1, 0){55}}
\put(20, -12){\tiny${n+1}$}
\put(5, -20){\line(0, -1){20}}
\put(20, -30){...}
\put(50, -20){\line(0, -1){20}}
\end{picture}
\end{center}
where the dotted-line indicates that the diagram is the composition of
the two diagrams above and below the line. The diagram above the dotted line is the tensor product
of an element in $\langle \Sigma(2n+1)\rangle_{n+1-2k}^1$ with $I_n$.  Since
$\langle \Sigma(2n+1)\rangle_{n+1-2k}^1=0$ for all $k\ge 1$, we have $\frac{D(n, 0)\Xi_k}{(n+1)!}=0$.
This proves $D(n, 0)\in \langle \Phi\rangle_{n+1}$.

Now we use induction on $r$ to show that the theorem holds for $r>n+1$.
If $p=0$, the diagram corresponds to $\Sigma(2n+1)$, which
belongs to $\langle \Phi\rangle_{2n+1}$ by Lemma \ref{lem:Sigma-Phi}.
Assume $n\ge p\ge 1$, and let $r=2n+1-p+q$. Consider $D(p, q)\circ\Sigma(2n+1-p)$
by using the the formula
\[
\Sigma(2n+1-p)=\left(\left(\Phi - \sum_{k=1}^{\left[\frac{n+1}{2}\right]}a_k \Xi_k\right)
\otimes I_{n-p}\right)\frac{\Sigma(2n+1-p)}{(n+1)!}.
\]
We obtain an expression for $D(p, q)$ of the form
\begin{eqnarray}\label{eq:D-Phi}
 D(p, q) = \sum_{k\ge 1} c_k D(p, q; k) + D^0,
\end{eqnarray}
where  $c_k$ are scalars, $D^0\in \langle \Phi\rangle_r$, and

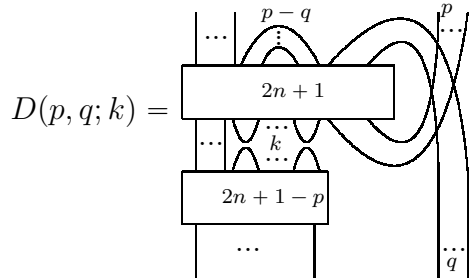
\begin{center}
\begin{picture}(100, 80)(0,0)
\put(-65, 20){$D(p, q; k) =$}
\put(5, 40){\line(0, 1){20}}
\put(8, 50){...}
\put(20, 40){\line(0, 1){20}}

\put(30, 58){\tiny${p-q}$}
\qbezier(22, 40)(36, 70)(52, 40)
\put(35, 52){.}
\put(35, 50){.}
\put(35, 48){.}
\qbezier(30, 40)(36, 54)(44, 40)

\qbezier(55, 40)(100, 90)(108, 0)
\qbezier(70, 40)(95, 70)(97, 0)

\put(0, 20){\line(1, 0){80}}
\put(0, 20){\line(0, 1){20}}
\put(80, 20){\line(0, 1){20}}
\put(0, 40){\line(1, 0){80}}
\put(30, 28){\tiny${2n+1}$}

\put(5, 20){\line(0, -1){20}}
\put(6, 10){...}
\put(16, 20){\line(0, -1){20}}

\qbezier(19, 20)(24, 2)(29, 20)
\put(31, 16){...}
\qbezier(42, 20)(47, 2)(52, 20)
\put(33, 8){\tiny$k$}
\qbezier(19, 0)(24, 18)(29, 0)
\put(31, 4){...}
\qbezier(42, 0)(47, 18)(52, 0)

\qbezier(55, 20)(95, -30)(108, 60)
\qbezier(70, 20)(95, -10)(97, 60)
\put(97, 53){...}
\put(98, 60){\tiny$p$}

\put(0, 0){\line(1, 0){55}}
\put(0, 0){\line(0, -1){20}}
\put(55, 0){\line(0, -1){20}}
\put(0, -20){\line(1, 0){55}}
\put(15, -12){\tiny${2n+1-p}$}
\put(5, -20){\line(0, -1){20}}
\put(20, -30){...}
\put(50, -20){\line(0, -1){20}}

\put(97, 0){\line(0, -1){40}}
\put(108, 0){\line(0, -1){40}}

\put(98, -30){...}
\put(100, -36){\tiny$q$}
\end{picture}
\end{center}

\vspace{2cm}

The diagram $D(p, q; k)$ is the composition of

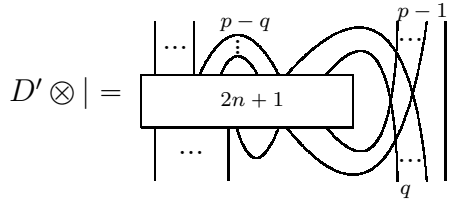
\begin{center}

\begin{picture}(100, 80)(0,0)
\put(-50, 30){$D'\otimes |=$}
\put(5, 40){\line(0, 1){20}}
\put(8, 50){...}
\put(20, 40){\line(0, 1){20}}

\put(30, 58){\tiny${p-q}$}
\qbezier(22, 40)(36, 70)(52, 40)
\put(35, 52){.}
\put(35, 50){.}
\put(35, 48){.}
\qbezier(30, 40)(36, 54)(44, 40)

\qbezier(55, 40)(100, 90)(108, 0)
\qbezier(70, 40)(95, 70)(97, 0)
\put(97, 7){...}
\put(98, -5){\tiny$q$}

\put(0, 20){\line(1, 0){80}}
\put(0, 20){\line(0, 1){20}}
\put(80, 20){\line(0, 1){20}}
\put(0, 40){\line(1, 0){80}}
\put(30, 28){\tiny${2n+1}$}

\put(5, 20){\line(0, -1){20}}
\put(14, 10){...}
\put(33, 20){\line(0, -1){20}}

\qbezier(36, 20)(43, -3)(52, 20)

\qbezier(55, 20)(95, -30)(108, 60)
\qbezier(70, 20)(95, -10)(97, 60)
\put(97, 53){...}
\put(97, 62){\tiny${p-1}$}
\put(115, 60){\line(0, -1){60}}
\end{picture}
\end{center}

\bigskip

\noindent
with the following element of $B_r(-2n)$

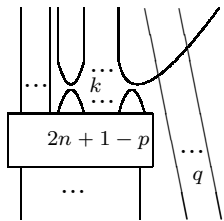
\begin{center}
\begin{picture}(100, 80)(-5, -35)
\qbezier(19, 20)(24, 2)(29, 20)
\put(31, 16){...}
\put(31, 8){\tiny$k$}
\qbezier(19, 0)(24, 18)(29, 0)
\put(31, 4){...}
\qbezier(42, 0)(47, 18)(52, 0)

\put(5, 20){\line(0, -1){20}}
\put(6, 10){...}
\put(16, 20){\line(0, -1){20}}

\put(5, 20){\line(0, 1){20}}
\put(16, 20){\line(0, 1){20}}

\put(19, 20){\line(0, 1){20}}
\put(29, 20){\line(0, 1){20}}
\put(42, 20){\line(0, 1){20}}

\qbezier(42, 20)(47, -5)(80, 40)

\put(0, 0){\line(1, 0){55}}
\put(0, 0){\line(0, -1){20}}
\put(55, 0){\line(0, -1){20}}
\put(0, -20){\line(1, 0){55}}
\put(15, -12){\tiny${2n+1-p}$}
\put(5, -20){\line(0, -1){20}}
\put(20, -30){...}
\put(50, -20){\line(0, -1){20}}

\put(65, -15){...}
\put(70, -25){\tiny$q$}

\put(52, 40){\line(1, -5){16}}
\put(65, 40){\line(1, -5){16}}
\end{picture}
\end{center}

\noindent
Note that $D'$ belongs to $\ker{{F}_{r-1}}$.
Thus $D'\in \langle \Phi\rangle_{r-1}$ by the
induction hypothesis and it follows that
$D(p, q; k)\in \langle \Phi\rangle_r$.
This completes the proof.
\end{proof}

\section{Remarks concerning the invariant theory of quantum groups}

In this section we give a brief indication as to how the ideas we have discussed may be extended to quantised 
enveloping algebras, both in the classical and super cases. General references for the material below are 
\cite{L,Ja,Ka}, while some references for applications in our context are \cite{LZ1,LZZ1, LZZ2,LZ5}.

\subsection{Quantum groups and $R$-matrices} Let $\fg$ be a complex semi-simple Lie algebra of rank $r>0$.
Associated to $\fg$ we have its universal enveloping algebra $\U(\fg)$ as well as its $q$-deformation $\U_q(\fg)$, 
which is an infinite dimensional Hopf algebra over a field $\K$ which is a finite extension of the function field
$\C(q)$, where $q$ is an indeterminate over $\C$. Unlike $\U(\fg)$, $\U_q(\fg)$ is not co-commutative, but 
suitably interpreted, we have $\lim_{q\to 1}\U_q(\fg)=\U(\fg)$.

Let $\fh$ be a Cartan subalgebra of $\fg$ and $\fb\supseteq \fh$ a Borel subalgebra, with $\Phi\subset\fh^*$ the root system
og $\fg$ with respect to $\fh$, and $\Pi=\{\alpha_1,\dots\alpha_r\}\subset \Phi$ the set of simple roots of $\Phi$ corresponding to the choice of $\fb$.
If $W$ is the Weyl group, there is a unique inner product $(-,-)$ on $\fh^*$ which is $W$-invariant and satisfies the condition
$( \alpha,\alpha)=2$ for short roots $\alpha\in\Phi$.

The algebra $\U_q(\fg)$ has generators $\{E_i,F_i,K_i^{\pm 1}\mid i=1,\dots, r\}$ and relations which include the following
(see, e.g., \cite[\S 6]{LZ1}).
\be\label{eq:defq}
\begin{aligned}
&K_iK_j=K_jK_i,\quad  K_iK_i\inv=K_i\inv K_i=1,\\
&K_iE_jK_i\inv=q^{(\alpha_i,\alpha_j)}E_j,\quad  K_iF_jK_i\inv=q^{-(\alpha_i,\alpha_j)}F_j,\\
&[E_i,F_j]=\delta_{ij}\frac{K_i-K_i\inv}{q-q\inv},\quad \text{and}\\
&\text{the quantum Serre relations.}
\end{aligned}
\ee
It has the structure of a Hopf algebra, with the coproduct defined by
\be\label{eq:cop}
\Delta(K_i)=K_i\ot K_i,\;\;\;\Delta(E_i)=E_i\ot K_i+1\ot E_i,\;\;\;\Delta(F_i)=F_i\ot 1+K_i\inv\ot F_i.
\ee

Corresponding to each simple Lie superalgebra $\fg$, there exists a similarly defined quantum supergroup $\U_q(\fg)$ \cite{BGZ, Z93, Z98}, which has the structure of a Hopf superalgebra. 

It was proved by Drinfeld \cite{D} for his version \cite{D} of $\U_q(\fg)$ over the formal power series ring that there is an invertible element $R\in{\U_q(\fg)\widehat\ot\U_q(\fg)}$, a suitably defined
``completion'' of $\U_q(\fg)\ot\U_q(\fg)$, which satisfies, among others, the following relations 
\be\label{eq:yb}
\begin{aligned}
R\Delta(u)=&\Delta'(u)R, \\ 
R_{12}R_{13}R_{23}=&R_{23}R_{13}R_{12},\\
\end{aligned}
\ee
where $\Delta'$ is the opposite coproduct. This element is called the universal $R$-matrix, and the second relation above is known as the Yang-Baxter equation.  

As explained in \cite{LZ1}, when suitably interpreted, the universal $R$ matrix leads to an $R$-matrix for each pair of locally $\U_q(\fb)$-finite modules of type-${\bf 1}$  \cite{Ja} for the Jimbo version \cite{Ji} of  $\U_q(\fg)$ over $\K$. 
Hereafter we will consider this version of the quantum group and its type-${\bf 1}$ representations only. 

If $V$ and $W$ are $\U_q(\fg)$-modules, which are locally finite for the action of $\U_q(\fb)$, then $R$ defines a $\K$-linear map $R_{V, W}: V\ot W\lr V\ot W$.
Moreover if $\tau_{V, W}: V\ot W\lr W\ot V$, $v\ot w\mapsto w\ot v$,  is the functorial linear map and we write $\check R_{V, W}:=\tau_{V, W}\circ R_{V, W}$, the equations \eqref{eq:yb}
translate into

\begin{theorem}\label{thm:rmat}
Let $V_1,V_2$ and $V_3$ be any highest weight $\U_q(\fg)$-modules. Then

(1) $\check R$ defines a $\U_q(\fg)$-isomorphism: $V_1\ot V_2\lr V_2\ot V_1$.

(2) There is the following equality of isomorphisms $V_1\ot V_2\ot V_3\lr V_3\ot V_2\ot V_1$.
\[
\begin{aligned}
&(\check{R}_{V_1, V_2}\ot\id_{V_3}) ( \id_{V_2}\ot \check{R}_{V_1, V_3})(\check{R}_{V_2, V_3}\ot\id_{V_1})\\
&=( \id_{V_1}\ot \check{R}_{V_2, V_3})  (\check{R}_{V_1, V_3}\ot\id_{V_2})  ( \id_{V_1}\ot \check{R}_{V_1, V_2}).
\end{aligned}
\]
\end{theorem}

\subsection{Braid group actions on tensor spaces} Let $\B_n$ be the $n$-string braid group. It is well known that 
$\B_n$ has a presentation $\B_n=\langle \sigma_1,\dots,\sigma_{n-1}\mid \sigma_i\sigma_j=\sigma_j\sigma_i\text{ if }
|i-j|\geq 2,\;\; \sigma_i\sigma_{i+1}\sigma_i=\sigma_{i+1}\sigma_i\sigma_{i+1}\;\;\forall i  \rangle$. This group, also
known as the Artin braid group of type $A_{n-1}$, has a 
well known depiction in terms of braid diagrams.

A closely related group is the Artin braid group $\Gamma_n$ of type $B_n$, which has generators $\xi,\sigma_1,\dots,\sigma_{n-1}$.
The relations are those already given for the $\sigma_i$ in the presentation of $\B_n$ with additional relations: $\xi\sigma_1\xi\sigma_1=
\sigma_1\xi\sigma_1\xi$ and $\xi\sigma_i=\sigma_i\xi$ for $i>1$. The group $\Gamma_n$ has a well known depiction in terms of
either ``cylindrical braids'' \cite{GL03} or ``polar braids'', that is, braids which may encircle a ``pole'' \cite{ILZ2}. 

\begin{proposition}\label{prop:braids}
\begin{enumerate}
\item The $\sigma_i$, $i=1,\dots,n-1$ generate a subgroup of $\Gamma_n$ which is isomorphic to $\B_n$.
\item Let $\B_{n+1}=\langle\sigma_0,\sigma_1,\dots,\sigma_{n-1}  \rangle$ be the $(n+1)$-string braid group.
Then the map $\iota:\B_{n+1}\lr\Gamma_n$ defined by $\sigma_0\mapsto\xi^2, $ and $\sigma_i\mapsto \sigma_i\in\Gamma_n$
for $i=1,2,\dots,n-1$ defines a monomorphism of groups.
\end{enumerate}
\end{proposition}

Proofs of these statements may be found in \cite{GL03,GL04,ILZ2}. The above statements lead to the following invariant theoretic 
observation.

\begin{corollary}\label{cor:braid-tensor}
Let $V,W$ be highest weight $\U_q(\fg)$-modules, and assume that $V$ is finite dimensional. 
\begin{enumerate}
\item 
There are algebra homomorphisms
\be\label{eq:eta}
\eta_n:\K\B_n\lr \End_{\U_q(\fg)}(V^{\ot n}),\;\;n=1,2,\dots
\ee
defined by $\sigma_i\mapsto (\id_V)^{\ot(i-1)}\ot \check R_{V, V}\ot(\id_V)^{\ot(n-i-1)}$
for $i=1,2,\dots, n-1$.

\item
There are algebra homomorphisms
\be\label{eq:nu}
\nu_n:\K\Gamma_n\lr \End_{\U_q(\fg)}(W\ot V^{\ot n}),\;\; n=1,2,\dots
\ee
defined by $\xi\mapsto \check R_{V,W}\circ\check R_{W,V}$ and $\sigma_i\mapsto \id_W\ot (\id_V)^{\ot(i-1)}\ot \check R_{V, V}\ot(\id_V)^{\ot(n-i-1)}$
for $i=1,2,\dots, n-1$.
\end{enumerate}
\end{corollary}

The homomorphisms $\eta_n$ and $\nu_n$ have been the subject of much literature. They play an extremely important role in the area of quantum topology (see, e.g.,  \cite{RT,  T1,  ZGB}). 
From the viewpoint of invariant theory, the very first natural questions which arise are: 
\begin{description}
\item[Question A]
For which $\fg, V, W$ are the maps $\mu_n$ and $\nu_n$ surjective? 
\item[Question B] Which algebras arise as images of $\mu_n$ and $\nu_n$, for simple or affine Kac-Moody Lie algebras and superalgebras $\fg$?
\end{description}
These questions provide a context for the Hecke algebras, 
BMW \cite{BW} algebras and their affine analogues, and for problems with applications in quantum computing.

\subsection{Relating classical and quantum invariant theory and roots of unity} In relation to Question A, it is shown in \cite{LZ1}
that $\mu_n$ is surjective for all $n$ when $V$ is a strongly multiplicity free module for $\U_q(\fg)$. It is beyond the scope of this survey 
to go into details of this result, or even the definition of strongly multiplicity free modules. However we note that the proof
involves a comparison of the quantum and classical cases by means of a ``base change'' argument.

This involves taking $A$-forms of $\U_q(\fg)$ and of the modules $V$ and $V^{\ot n}$, where $A$ is a suitable subring of $\K$.
One then uses careful deformation arguments to compare the dimensions (over the respective fields $\K$ and $\C$) of the endomorphism 
algebras in the classical and quantum cases (cf. \cite[\S 7]{LZ1}).

Question A for quantum supergroups \cite{BGZ, Z93, Z98}  is addressed in 
\cite{LZZ2} by using results from deformation quantisation. Quantum analogues of FFTs given in Theorem \ref{thm:fts-GL} and Theorem \ref{thm:fft-sft} are proved, establishing full tensor functors from categories of oriented and non-oriented tangles to categories of tensor representations of the quantum general linear supergroup and quantum orthosymplectic supergroup.  This includes the quantum groups associated with classical Lie algebras as special cases, thus gives an independent proof of the main results of \cite{LZ1} discussed above by a different method.  The work \cite{LZZ2} is also closely related to papers \cite{Z93, Z95}  on the construction of quantum supergroup invariants of knots and $3$-manifolds.

Question B may also be addressed in this way. A particular case when the method works well is if there is a finite dimensional cellular algebra
$\CA_n$ through which $\eta_n$ factors, for example the BMW algebra in the case when $V$ is the natural module for  $\Or_m$ or $\Sp_{2r}$. The cellularity can then be used to explicitly use the representation theory of the relevant algebra  to compare the classical and quantum cases. This is explained in \cite[\S 9]{LZ5}

Integral forms of the quantum group and its modules also permit an analysis of the invariants when $q$ is specialised to a root of unity. This is explained in \cite{DPS, ALZ}.

\subsection{Some category equivalences} There are very few known cases when we have an actual equivalence of categories
between a subcategory of the category of representations of a Lie algebra or its quantum analogue on the one hand, and a category
of diagrams on the other. One example we have seen in the classical case, is the full tensor subcategory of representations of the 
special orthogonal group $\SO_m(\C)$ generated by the natural representation $V$,
 and the enhanced Brauer category $\wt\CB(m)$ discussed in \S \ref{s:eb}. It is proved in Theorem \ref{thm:main-SO} (see also \cite{LZ8}) that these categories are equivalent.

In the quantum case, it is well known (e.g., through Khovanov's theory of categorification for $\U_q(\fsl_2)$)
that there is a quotient category $\TL(q)$ of the tangle category \cite{FY}, called the Temperley-Lieb category,  
such that if $V$ is the two-dimensional Weyl module for $\U_q(\fsl_2)$, and $Rep^0(\fsl_2)$
is the tensor category of $\U_q(\fsl_2)$-representations generated by $V$, then there is an equivalence
of categories: $\CF:\TL(q)\lr Rep^0(\fsl_2)$ (see \cite{LZ2,LZ3}).

The other case where such an equivalence is known is given in the recent work \cite{ILZ2}, where it is shown that
there is a family of categories $\TLB(q,Q)$, such that $\TL(q)$ is a subcategory of $\TLB(q,Q)$ for each $Q$. If 
$M(m)$ is the (projective) Verma module for $\U_q(\fsl_2)$ with highest weight $m$, and the category of representations
with objects $M(m)\ot V^{\ot n}$, $n=0,1,2,\dots$ is denoted $Rep(\fsl_2)$, then there exists an equivalence of categories
\[
\wt\CF: \TLB(q,q^m)\lr Rep(\fsl_2)
\]
whose restriction to $\TL(q)$ is the classical equivalence $\CF$ of the last paragraph.

It would be desirable to have more such equivalences, since they make potentially complex questions about representations
amenable  to diagrammatic methods, which could be essentially combinatorial.


\appendix
\section{An algebraic proof of the presentation for the Brauer category}\label{sect:proof-presentation}

In principle, we can deduce a proof of Theorem \ref{thm:presentation} from 
 \cite{FY}.  Here  we provide an independent  proof, which is taken from \cite{LZ5}.

\begin{proof}[Proof of Theorem \ref{thm:presentation}]
 We first prove (1). The fact that the elementary Brauer diagrams $I,X$, $A$ and $U$
generate all Brauer diagrams under the operations of $\circ$ and $\ot$ may be seen as follows.
Fix the nodes of an arbitrary diagram $D$ from $k$ to $\ell$, and draw all the arcs as
piecewise smooth curves, in such a way
that there are at most two arcs through any point, and that no two crossings
or turning points have the same vertical coordinate. We may now draw a set of horizontal lines
(possibly after a small perturbation of the diagram) such that

(i) each line is not tangent to any of the arcs,

(ii) between successive lines there is precisely one crossing or turning point.

Then the part of the diagram between successive lines may be thought of as the $\ot$-product
of the four generators, all except one being equal to $I$.
Thus we have exhibited $D$ as a word in the generators, of the form
$D=D_1\circ D_2\circ\dots\circ D_n$, where each $D_i$ is of the form
\be\label{eq:elem}
D_i=I^{\ot r}\ot Y\ot I^{\ot s},
\ee
with $Y$ being one of $A,U$ or $X$. Such an expression will be called
a {\em regular expression}, and the
factors $D_i$ {\em elementary diagrams}.
A product of elementary diagrams
in which $Y=X$ for each factor will be called a {\em permutation diagram}.
An example of a particular regular expressions is given in Figure \ref{fig:Ts}.

\begin{figure}[h]
\begin{center}
\begin{picture}(130, 70)(-10,0)
\qbezier(0, 70)(0, 30)(0, 0)
\put(3, 10){$...$}
\qbezier(15,70)(15, 30)(15, 0)

{\color{red}\qbezier[60](-10, 16)(56, 16)(130, 16)}

\qbezier(30, 70)(30, 30)(30, 15)
\qbezier(30, 15)(35, -10)(45, 15)
\qbezier(45, 15)(53, 25)(90, 70)

{\color{red}\qbezier[60](-10, 33)(56, 33)(130, 33)}


\qbezier(45, 70)(53, 30)(60, 15)
\qbezier(60, 0)(60, 10)(60, 15)

{\color{red}\qbezier[60](-10, 50)(56, 50)(130, 50)}

\qbezier(60, 70)(68, 30)(75, 15)
\qbezier(75, 0)(75, 10)(75, 15)


\qbezier(75, 70)(83, 30)(90, 15)
\qbezier(90, 0)(90, 10)(90, 15)

\put(77, 10){$...$}

\qbezier(105, 70)(105, 30)(105, 0)
\put(108, 10){$...$}
\qbezier(120, 70)(120, 30)(120, 0)
\end{picture}
\end{center}
\caption{Regular expression}
\label{fig:Ts}
\end{figure}

This completes the proof of (1).

We now turn to the proof that the stated relations form a complete set.
Observe first that any expression for a diagram $D$ as a word in the generators
provides a regular expression for $D$ by repeated use of the relation
\eqref{eq:identity} and its dual. Accordingly we say that two regular expressions
$\mf{D},\mf{D}'$ are {\em equivalent}, and write $\mf D\sim \mf D'$ if
one can be obtained from the other by a sequence of applications of the relations in part (2)
of the Theorem. This is clearly an equivalence relation on regular expressions.

However, a word in the generators does not in general yield a Brauer diagram, but
rather a diagram multiplied by $\delta^k$ for some nonnegative integer $k$,
where $k$ is the number of deleted loops.
For any Brauer diagram $D$ and any $N\in\Z_+$, the above argument shows that
we can always represent $\delta^N D$
as a word in the generators, and hence also as a regular expression. We therefore
need to work with morphisms of the form
$\delta^N D$, where $D$ is a diagram. We refer to such a morphism
as  a {\em scaled Brauer diagram}, or simply a {\em scaled diagram}.
Every Brauer diagram is clearly a scaled diagram.

The discussion above shows that to prove the theorem, it will suffice to show that
\be\label{eq:task}
\text{Any two regular expressions for a scaled diagram are equivalent.}
\ee

We shall extend the notion of equivalence to any expression of the form $D_1\circ\dots\circ D_n$,
where the $D_i$ are diagrams.

\begin{definition}\label{def:equiv}
The two compositions $D_1\circ\dots\circ D_n$ and $D_1'\circ\dots\circ D_m'$ are said to be equivalent
if one can be obtained from the other using only the relations
in Theorem \ref{thm:presentation} (2), and the properties of $\circ$ and $\ot$.
\end{definition}

To prove \eqref{eq:task} we require some analysis of regular expressions and equivalence.
We shall return to the proof after carrying this out.
\end{proof}

\begin{definition}\label{def:valency}
\begin{enumerate}
\item The {\em valency} of scaled diagram $D\in B_k^l$ is the pair $(k,l)$.
\item If $D=I^{\ot r}\ot Y\ot I^{\ot s}$ is elementary, the {\em abscissa}
$a(D)$ of $D$ is $r+1$, while the {\em type} $t(D)=Y$ $(=A,U$ or $X)$.
\item The {\em length} of a regular expression $E_1\circ\dots\circ E_n$ is $n$.
\end{enumerate}
\end{definition}

We shall repeatedly apply the following elementary observation, which we refer to
as the ``commutation principle''.
\begin{remark}\label{rem:commute}
\begin{enumerate}
\item Let $E_1,E_2$ be elementary diagrams such that $E_1\circ E_2$ makes sense.
If $|a(E_1)-a(E_2)|>1$ then $E_1\circ E_2\sim E_1'\circ E_2'$,
where $t(E_1')=t(E_2)$ and $t(E_2')=t(E_1)$.
\item If $D,D'$ are scaled diagrams of valency $(k,l)$ and $(k',l')$ respectively,
then $D\ot D'=(I^{\ot l}\ot D')\circ (D\ot I^{\ot k'})=(D\ot I^{\ot l'})\circ (I^{\ot k}\ot D')$.
\end{enumerate}
\end{remark}
Part (2) of the Remark states the obvious relations among diagrams depicted in Figure \ref{fig:commute}.
\begin{figure}[h]
\begin{center}
\begin{picture}(100, 60)(70,0)

\qbezier(16, 40)(16, 50)(16, 60)
\qbezier(29, 40)(29, 50)(29, 60)
\put(18, 55){$...$}

\qbezier(10, 40)(20, 40)(35, 40)
\qbezier(10, 25)(20, 25)(35, 25)
\put(18, 30){$D$}
\qbezier(10, 40)(10, 15)(10, 30)
\qbezier(35, 40)(35, 15)(35, 30)

\qbezier(16, 25)(16, 15)(16, 5)
\qbezier(29, 25)(29, 15)(29, 5)
\put(18, 10){$...$}


\qbezier(46, 40)(46, 50)(46, 60)
\qbezier(59, 40)(59, 50)(59, 60)
\put(48, 55){$...$}

\qbezier(40, 40)(50, 40)(65, 40)
\qbezier(40, 25)(60, 25)(65, 25)
\put(48, 30){$D'$}
\qbezier(40, 40)(40, 15)(40, 30)
\qbezier(65, 40)(65, 15)(65, 30)

\qbezier(46, 25)(46, 15)(46, 5)
\qbezier(59, 25)(59, 15)(59, 5)
\put(48, 10){$...$}

\put(70, 30){$=$}

\qbezier(96, 30)(96, 50)(96, 60)
\qbezier(109, 30)(109, 50)(109, 60)
\put(98, 55){$...$}

\qbezier(90, 30)(110, 30)(115, 30)
\qbezier(90, 15)(110, 15)(115, 15)
\put(98, 18){$D$}
\qbezier(90, 30)(90, 15)(90, 15)
\qbezier(115, 30)(115, 15)(115, 15)

\qbezier(96, 15)(96, 10)(96, 5)
\qbezier(109, 15)(109, 10)(109, 5)
\put(98, 10){$...$}


\qbezier(126, 50)(126, 50)(126, 60)
\qbezier(139, 50)(139, 50)(139, 60)
\put(128, 55){$...$}

\qbezier(120, 50)(130, 50)(145, 50)
\qbezier(120, 35)(140, 35)(145, 35)
\put(128, 40){$D'$}
\qbezier(120, 50)(120, 45)(120, 35)
\qbezier(145, 50)(145, 45)(145, 35)

\qbezier(126, 35)(126, 15)(126, 5)
\qbezier(139, 35)(139, 15)(139, 5)
\put(128, 10){$...$}

\put(150, 30){$=$}

\qbezier(176, 50)(176, 55)(176, 60)
\qbezier(189, 50)(189, 55)(189, 60)
\put(178, 55){$...$}

\qbezier(170, 50)(180, 50)(195, 50)
\qbezier(170, 35)(180, 35)(195, 35)
\put(178, 40){$D$}
\qbezier(170, 50)(170, 45)(170, 35)
\qbezier(195, 50)(195, 45)(195, 35)

\qbezier(176, 35)(176, 15)(176, 5)
\qbezier(189, 35)(189, 15)(189, 5)
\put(178, 10){$...$}


\qbezier(206, 30)(206, 40)(206, 60)
\qbezier(219, 30)(219, 40)(219, 60)
\put(208, 55){$...$}

\qbezier(200, 30)(210, 30)(225, 30)
\qbezier(200, 15)(210, 15)(225, 15)
\put(208, 20){$D'$}
\qbezier(200, 30)(200, 20)(200, 15)
\qbezier(225, 30)(225, 20)(225, 15)

\qbezier(206, 15)(206, 10)(206, 5)
\qbezier(219, 15)(219, 10)(219, 5)
\put(208, 10){$...$}
\end{picture}
\end{center}
\caption{Commutativity}
\label{fig:commute}
\end{figure}

This follows from the fact that $(A\ot B)\circ (A'\ot B')\sim (A\circ A')\ot(B\circ B')$
for $A,A',B,B'$ of appropriate valency,
and the relation \eqref{eq:identity}.

The next two results will be used in the reduction of the proof of Theorem \ref{thm:presentation}
(2) to a single case.

\begin{lemma}\label{lem:red-perm} Let $P,Q$ be permutation diagrams of valency
$(l,l)$ and $(k,k)$ respectively and let $D\in B_k^l$ be a scaled diagram.
If any two regular expressions for ${P}\circ{D}\circ{Q}$
are equivalent, then so are any two regular expressions for ${D}$.
\end{lemma}
\begin{proof} Let $\mf D$, $\mf D'$ be two regular expressions for $D$,
and suppose for the moment that $P$ is an elementary permutation diagram.
Then $P\circ\mathfrak{D}$ and $P\circ\mathfrak{D}'$ are regular expressions
for $P\circ D$, and hence are equivalent by hypothesis. Now
$P\circ P\circ\mathfrak{D}$ is a regular expression, and
it is evident that $P\circ P\circ\mathfrak{D}$ is equivalent
to $P\circ P\circ\mathfrak{D}'$. But from \eqref{eq:XX}, $P\circ P\circ\mathfrak{D}\sim \mf{D}$
and $P\circ P\circ\mathfrak{D}'\sim\mathfrak{D}'$, whence
$\mf D$ and $\mf D'$ are equivalent. This proves the Lemma for elementary $P$ and $Q=\id$.

Applying the above statement repeatedly, we see that for any permutation diagram $P$,
if any two regular expressions for $P\circ D$ are equivalent, the same is true for
$D$. A similar argument applies to prove the corresponding statement for $D\circ Q$,
for any permutation diagram $Q$.
\end{proof}

It follows that in proving \eqref{eq:task}, we may pre- and post-multiply
$D$ by arbitrary permutation diagrams, and replace $D$ by the resulting scaled diagram.

For the second reduction, we require the following definitions.
\begin{definition}\label{def:raising}
\begin{enumerate}
\item Define $R:B_k^l\to B_{k-1}^{l+1}$ (for $k\geq 1$) (the {\em raising operator}) by
$R(D)=(D\ot I)\circ (I^{\ot (k-1)}\ot U)$, and (the {\em lowering operator})
$L:B_k^l\to B_{k+1}^{l-1}$ by
$L(D)=(I^{\ot (l-1)}\ot A)\circ(D\ot I)$.
\item If $\mf D=D_1\circ D_2\circ\dots\circ D_n$ is a regular expression for the
scaled diagram $D\in B_k^l$, define the regular expression $R(\mf D)$ for $R(D)$ by
$R(\mf D)=(D_1\ot I)\circ (D_2\ot I)\circ\dots\circ (D_n\ot I)\circ (I^{\ot k-1}\ot U)$,
and similarly define the regular expression $L(\mf D)$ for $L(D)$.
Note that if $E$ is elementary, then so is $E\ot I$, so that the above definition makes sense.
\end{enumerate}
\end{definition}

\begin{lemma}\label{lem:red-raise}
\begin{enumerate}
\item For any regular expression $\mf D$ for a scaled diagram $D\in B_k^l$,
we have $R\circ L(\mf D)\sim \mf D$ and $L\circ R(\mf D)\sim\mf D$.
\item Suppose $D$ is a scaled diagram of valence $(k,l)$ with $k\geq 1$.
The regular expressions $\mf D,\mf D'$ for $D$ are equivalent if and only if
$L(\mf D)$ and $L(\mf D')$ (or $R(\mf D)$ and $R(\mf D')$) are equivalent.
\end{enumerate}
\end{lemma}
\begin{proof}
To prove (1), let $\mf D=E_1\circ\dots\circ E_n$ be a regular expression for $D\in B_k^l$. Then
$$
\begin{aligned}
R\circ L&(\mf D)=R\left((I^{\ot(l-1)}\ot A)\circ(E_1\ot I)\dots\circ (E_n\ot I)\right)\\
&=(I^{\ot(l-1)}\ot A\ot I)\circ(E_1\ot I\ot I)\dots\circ(E_n\ot I)\circ (I^{\ot k}\ot U)\\
&\sim(I^{\ot(l-1)}\ot A\ot I)\circ(I^{\ot l}\ot U)\circ E_1\circ\dots\circ E_n
\text{ by several applications of }\ref{rem:commute}\\
&\sim I^{\ot l}\circ E_1\circ\dots\circ E_n\text {by }\eqref{eq:straight}\\
&\sim E_1\circ\dots\circ E_n\text{ by }\eqref{eq:identity}\\
&=\mf D.\\
\end{aligned}
$$

This shows that $R\circ L(\mf D)\sim\mf D$, and the proof that $L\circ R(\mf D)\sim\mf D$ is similar.

Now to prove (2), suppose first that $\mf D, \mf D'$ are equivalent regular expressions
for $D$. Then the same sequence of moves using the relations in Theorem
\ref{thm:presentation} (2) which convert $\mf D$ into $\mf D'$ may be applied to
$L(\mf D)$ to convert it into $L(\mf D')$. This shows that if $\mf D,\mf D'$ are equivalent
regular expressions for $D$, then $L(\mf D),L(\mf D')$ are equivalent regular expressions
for $L(D)$. A similar argument proves the corresponding statement for $R(D)$.

To prove the converse,
suppose that any two regular expressions for $R(D)$ are equivalent,
and that $\mf D_1$ and $\mf D_2$ are two regular expressions for $D$. Then $R(\mf D_1)$ and
$R(\mf D_2)$ are two regular expressions for $R(D)$, and hence by hypothesis are
equivalent. Hence by the above, $L\circ R(\mf D_1)$ and $L\circ R(\mf D_2)$ are two
equivalent regular expressions for $L\circ R(D)$, which is equal to $D$ by (1).
But by (1), $L\circ R(\mf D_1)\sim \mf D_1$ and $L\circ R(\mf D_2)\sim\mf D_2$, whence
$\mf D_1\sim\mf D_2$.
\end{proof}

The following lemma is the key computation involving the relations in
Theorem \ref{thm:presentation} (2).

\begin{lemma}\label{lem:stack}
Let $\mf T_s:=E_s\circ E_{s-1}\circ\dots\circ E_0$ be a regular expression, where
$t(E_0)=U$, $a(E_0)=a$, $t(E_i)=X$ and $a(E_i)=a+i$ for $i\geq 1$. The diagram $\mf T_s$
is shown in Figure \ref{fig:Ts}.
Let $E$ be an elementary diagram of type $A$ or $X$ which does not `commute with'
$E_s\circ E_{s-1}\circ\dots\circ E_0$, i.e. such that $a-1\leq a(E)\leq a+s+1$.
Then
\begin{enumerate}
\item If $t(E)=A$, then $E\circ\mf T_s$ is equivalent to a shorter regular expression unless
$s=0$ and $a(E)=a(E_0)$. In the latter case, $E\circ\mf T_s$ is the identity multiplied by $\delta$.
\item Suppose $t(E)=X$; then

(i) if $a+1\leq a(E)\leq a+s-1$, then $E\circ\mf T_s\sim \mf T_s\circ E'$ for
an elementary diagram $E'$ of type $X$. (Thus $E$ may be `moved through' $E\circ\mf T_s$).

(ii) if $a(E)=a$ or $a+s$, then $E\circ\mf T_s$ is equivalent to
a shorter regular expression.

(iii) if $a(E)=a-1$ or $a+s+1$ then $E\circ \mf T_s\sim \mf T_{s+1}$.

\item Let $\mf T_s$ be as above and let $E$ be elementary of type $A$ or $X$.
Then $E\circ \mf T_s$ is equivalent to a shorter regular expression
(possibly multiplied by $\delta$) or to $\mf T_s\circ E'$
for some elementary $E'$, or to $\mf T_{s+1}$.
\end{enumerate}
\end{lemma}

\begin{proof} Consider first the case where $t(E)=A$.

If $s=0$ and $a(E)=a(E_0)$, the claim follows from the loop removal relation \eqref{eq:AU}.

If $a(E)=a+s+1$, then
applying \eqref{eq:slide}, $E\circ E_s\sim E'\circ E_s'$, where $t(E')=t(E)=A$,
$t(E_s')=t(E_s)=X$, $a(E')=a+s$ and $a(E_s')=a+s+1$. It now follows by repeated application
of Remark \ref{rem:commute} about commutation, that $E\circ\mf T_s\sim E''\circ \mf T_{s-1}\circ E'''$,
where $t(E'')=A$ and $a(E'')=a+s$. Repeating this argument $s$ times, we see that
$E\circ \mf T_s$ is equivalent to a regular expression
of length $s+1$ which includes $F\circ E_0$ as a subexpression, where $t(F)=A$
and $a(F)=a+1$. Applying \eqref{eq:straight}, we see that $F\circ E\sim I^{\ot k}$ for some $k$,
and hence $E\circ\mf T_s$ is equivalent to a regular expression of length $s-1$.

If $a(E)=a+s$, then by \eqref{eq:AX}, $E\circ E_s\sim E$, and we have again shortened $E\circ\mf T_s$.

If $a\leq a(E)\leq a+s-1$, then by commutation, $E\circ\mf T_s$ is equivalent to a regular
expression with a subexpression of the form $E\circ E_i\circ E_{i-1}$, where $t(E_i)=X$
and $a(E)=a(E_i)-1$. Applying \eqref{eq:straight}, this is equivalent
to an expression $E'\circ E_i'\circ E_{i-1}$, where $a(E_i')=a(E_{i-1})$, and $t(E_i')=X$.
Using either \eqref{eq:XX} (if $i>1$) or the $^*$ of \eqref{eq:AX}, we again reduce the
length to show that $E\circ\mf T_s$ is equivalent to a shorter regular expression.

Finally if $a(E)=a-1$, we use commutation to show that $E\circ\mf T_s$ is equivalent
to a regular expression of length $s+1$ with a subexpression of the form
$E'\circ E_0$, where $t(E')=A$ and $a(E')=a-1=a(E_0)-1$. Applying \eqref{eq:straight},
we see that $E'\circ E_0\sim I^{\ot k}$ for some $k$, and this completes the proof of (1).

Now consider the case where $t(E)=X$.

If $a+1\leq a(E)\leq a+s-1$, then after applying the commutation rule, $E\circ\mf T_s$
is equivalent to a regular expression of length $s+1$ which has a subexpression of
the form $E\circ E_{a(E)+1}\circ E_{a(E)}$. But using the braid relation \eqref{eq:braid},
this is equivalent to $E'\circ E_{a(E)}\circ E_{a(E)+1}$, where $E'=E_{a(E)+1}$.
Again using commutation, we may now move the last factor below $E_0$ (since $a(E)+1\geq a+2$).
It follows that $E\circ\mf T_s\sim \mf T_s\circ E'$, where $t(E')=X$. This proves (i).

If $a(E)=a+s+1$ then evidently $E\circ\mf T_s= \mf T_{s+1}$.
If $a(E)=a+s$, the relation $X\circ X=I\ot I$ \eqref{eq:XX} shows that $E\circ E_s\sim I^{\ot r}$
for some $r$, and hence $E\circ\mf T_s$ is equivalent to a shorter regular expression.
If $a(E)=a-1$, then we may use commutation to see that
$E\circ\mf T_s\sim E_s\circ\dots\circ E_1\circ E\circ E_0$. Using the relation \eqref{eq:slide}
we see that this is equivalent to $E_s\circ\dots\circ E_1\circ E_1\circ E_0'$, where $t(E_0')=U$.
Applying \eqref{eq:XX}, we see that $E\circ\mf T_s$ is equivalent to a shorter regular expression.
Finally, if $a(E)=a$, we again use commutation to see that $E\circ\mf T_s$ is equivalent to
$E_s\circ E_{s-1}\circ\dots\circ E\circ E_1\circ E_0$. Again applying \eqref{eq:slide},
we obtain a factor $E\circ E$, and applying \eqref{eq:XX}, we again shorten the regular
expression $E\circ\mf T_s$. This completes the proof of (2).

The statement (3) is a summary of the previous two statements.
\end{proof}

\begin{proof}[Completion of the proof of Theorem \ref{thm:presentation} (2)]
It remains to prove \eqref{eq:task}.
It follows from Lemmas \ref{lem:red-raise} and \ref{lem:red-perm} that to complete the
proof of the theorem, it suffices to prove \eqref{eq:task} for any scaled diagram which can be
obtained from $D$ by raising or lowering, or multiplication by a permutation diagram.
It follows that we may take $D$ to be the scaled diagram
$D=\delta^N U^{\ot r}$ $(N\in\Z_+)$. Hence we shall be done if we prove
the following result.
\be\label{eq:main}
\text {Any two regular expressions for $D=\delta^N U^{\ot r}$ are equivalent.}
\ee

We shall prove \eqref{eq:main} by induction on $r$, starting with $r=0$.
For convenience, we adopt the following local convention:
\begin{enumerate}
\item scaled diagrams will
be simply called ``diagrams";
\item a regular expression $\mf D$ is said to be ``$\delta$-equivalent" to another regular expression
$\mf D'$ if it can be changed to $\delta^k \mf D'$ for some $k\in\Z_+$ by the relations in
Theorem \ref{thm:presentation} (2).
\end{enumerate}

Let $r=0$ and suppose $\mf D:=D_1\circ\dots\circ D_n$ is a regular expression for the empty
scaled diagram $\delta^N$ in $B_0^0$. We need to show that $\mf D$ is $\delta$-equivalent to the empty regular
expression; we do this by showing that every non-empty regular expression for the
empty scaled diagram is $\delta$-equivalent to one of shorter length.

Now by valency considerations, we must have $D_1=A$ and $D_n=U$. Let $i$ be the least
integer such that $t(D_i)=U$; then for all $j<i$, $t(D_j)=A$ or $X$. Applying
Lemma \ref{lem:stack} repeatedly, we see that since at least one of the $D_j$
for $j<i$ is of type $A$, $\mf D$ is $\delta$-equivalent to a shorter regular expression.
This proves the result for $r=0$

Now take $r>0$ and let $\mf D=D_1\circ\dots\circ D_n$ be a regular
expression for $D$. Then since at least $r$ of the $D_i$ must have type $U$, we
have $n\geq r$. Moreover if $n=r$, which happen only if $N=0$,
then the $D_i$ are all of type $U$, and have odd
abscissa, and any such regular expression represents $D$.
Any two such regular expressions (which will be called minimal) are equivalent
by the commutation rule (see Remark \ref{rem:commute}).

It therefore suffices to show that if $n>r$, then $\mf D$ is $\delta$-equivalent to a shorter
regular expression.

Clearly we have $t(D_n)=U$; if $t(D_1)=U$ then $\mf D':=D_2\circ\dots\circ D_n$ is a
regular expression for $U^{\ot (r-1)}$, and we conclude by induction on $r$ that $\mf D'$
is $\delta$-equivalent to a shorter regular expression. Thus we are finished.
Let $p=p(\mf D)$ be the least index such that $D_p$ is of type $U$. We have seen that if $p=1$
then we are finished by induction. It will therefore suffice to show that
$\mf D$ is either equivalent to a regular expression $\mf D'$ with $p(\mf D')<p(\mf D)$,
or is $\delta$-equivalent to a shorter regular expression $\mf D'$.

Thus we take $p>1$; then $t(D_p)=U$, and $t(D_i)=A$ or $X$ for $i<p$.
We now apply Lemma \ref{lem:stack} to conclude that either
we may commute one of the $D_i$ ($i<p$) past $D_p$, or $D_1\circ\dots\circ D_p
\sim \mf T_{p-1}$ or at least one of the $D_i$ ($i<p$) is of type $A$.
In the first case, we obtain a regular expression with small $p$-value;
in the second case, in the diagram $D_1\circ\dots\circ D_n$ if the nodes are
numbered $1,2,\dots,2r$ from left to right, node $a(D_p)$ would be joined to
node $a(D_p)+p$. Hence $p=1$, which has been excluded.

In the third case, suppose $i$ is the largest index such that $1\leq i\leq p-1$
and $D_i$ is of type $A$. Then either some $D_j$ ($i\leq j\leq p-1$)
can be commuted past $D_p$ by application of Remark (\ref{rem:commute}), or else
we are in the situation of Lemma \ref{lem:stack} (1). In the former case, we have reduced
$p$; in the latter, by {\it loc. cit.} $D_i\circ\dots\circ D_p$ is $\delta$-equivalent to a shorter
regular expression.

We have now shown that either $\mf D$ is $\delta$-equivalent to a shorter regular expression,
or equivalent to a regular expression which has the same length as $\mf D$ but a smaller $p$ value.

This completes the proof of \eqref{eq:main}, and hence of Theorem \ref{thm:presentation}.
\end{proof}

\begin{remark}
We note that to prove part (2) of the theorem, we could
have proceeded by regarding $\CB(\delta)$ as a quotient category of
the category of (unoriented) tangles (see Remark \ref{rem:quotient-cat}) and deduce
the relations among the generators of Brauer diagrams
from a complete set of relations among the generators of tangles
given in \cite[\S 3.2]{T1} (suppressing information about orientation).
This way we obtain all relations except the one which
enforces the removal of free loops and multiplication by powers of $\delta$,
i.e., \eqref{eq:AU}.
\end{remark}



\end{document}